\documentclass[a4paper,reqno,11pt]{amsart}
\makeatletter
\@namedef{subjclassname@2020}{%
  \textup{2020} Mathematics Subject Classification}
\makeatother
\usepackage[left=1 in, right=1 in,top=1 in, bottom=1 in]{geometry}

\usepackage{bbm}
\usepackage{amsmath}
\usepackage{mathrsfs}
\usepackage{mathrsfs}
\usepackage{bm}
\usepackage{amsmath}
\usepackage{amsfonts}
\usepackage{amssymb}
\usepackage{amsthm}
\usepackage{lineno}
\usepackage[RGB]{xcolor}

\newcommand{\be}{\begin{equation}}
\newcommand{\bea}{\begin{equation}\begin{aligned}}
\newcommand{\beas}{\begin{equation*}\begin{aligned}}
\newcommand{\eeas}{\end{aligned}\end{equation*}}
\newcommand{\eea}{\end{aligned}\end{equation}}
\newcommand{\ee}{\end{equation}}

\def\diva{\diverge_{\mathcal{A}}}

\renewcommand{\(}{\left(}
\renewcommand{\)}{\right)}
\renewcommand{\[}{\left[}

\providecommand{\norm}[1]{\left\Vert#1\right\Vert}
\providecommand{\ns}[1]{\norm{#1}^2}
\providecommand{\as}[1]{\abs{#1}^2}
\providecommand{\abs}[1]{\left\vert#1\right\vert}
\providecommand{\Rn}[1]{\mathbb{R}^{#1}}

\providecommand{\se}[1]{\mathcal{E}_{#1}}

\providecommand{\fd}[1]{\mathfrak{D}_{#1}}

\def\sg{\mathbb{D}}
\def\sgz{\mathbb{D}^0}

\def\nab{\nabla}
\def\dt{\partial_t}

\def\hal{\frac{1}{2}}

\def\ls{\lesssim}

\def\ls{\lesssim}
\def\p{\partial}

\def\da{\Delta_{\mathcal{A}}}
\def\naba{\nab_{\mathcal{A}}}

\def\a{\mathcal{A}}
\def\n{\mathcal{N}}

\def\S{\mathbb{S}}

\newcommand{\beq}{\begin{equation}}
\newcommand{\eeq}{\end{equation}}
\newcommand{\bal}{\begin{align}}
\newcommand{\eal}{\end{align}}
\renewcommand{\(}{\left(}
\renewcommand{\)}{\right)}
\renewcommand{\[}{\left[}

\def\f{\mathcal{F}_{2N}}

\def\rest{\hskip 1pt{\hbox to 10.8pt{\hfill\vrule height 7pt width 0.4pt depth 0pt\hbox{\vrule height 0.4pt
width 7.6pt depth 0pt}\hfill}}}

\def\evalu{\hskip 1pt{\hbox to 2pt{\hfill \vrule height -6pt width 0.4pt depth0pt}}}

\DeclareMathOperator{\diverge}{div}

\newtheorem{thm}{Theorem}[section]
\newtheorem{lem}[thm]{Lemma}
\newtheorem{prop}[thm]{Proposition}

\newtheorem{rem}[thm]{Remark}

\title[compressible viscous surface waves]{Global Well-posedness of Compressible Viscous Surface Waves without Surface Tension}

\author{Ting Sun}
\address{School of Mathematical Sciences\\
	Xiamen University\\
	Xiamen, Fujian 361005, China}
\email[T. Sun]{sunting12120721@163.com}
\author{Yanjin Wang}
\address{School of Mathematical Sciences\\
	Xiamen University\\
	Xiamen, Fujian 361005, China}
\email[Y. J. Wang]{yanjin$\_$wang@xmu.edu.cn}

\thanks{This work was supported by the National Natural Science Foundation of China (12171401, 12231016).}

\subjclass[2020]{35B40, 35Q30, 35R35, 76N06, 76N10.}

\keywords{Free boundary problem; Compressible Navier--Stokes equations; Viscous surface waves; Global well-posedness; Decay.}

\date{\today}

\begin{document}

\begin{abstract}
We consider the free boundary problem for a layer of  compressible viscous barotropic fluid lying above a fixed rigid  bottom
and below the atmosphere of  positive constant pressure. The fluid dynamics is governed by the compressible Navier--Stokes equations with gravity, and the effect of surface tension is neglected on the upper free boundary. We prove the global well-posedness of the reformulated problem in flattening coordinates near the equilibrium in both two and three dimensions without any low frequency assumption of the initial data. The key ingredients here are the new control of the {\it Eulerian spatial derivatives} of the solution, which benefits a crucial nonlinear cancellation of the highest order spatial regularity of the free boundary, and the time weighted energy estimates.
\end{abstract}
\maketitle

\section{Introduction}\label{introud}
\numberwithin{equation}{section}
\subsection{Formulation of the problem.}
We consider a compressible, viscous, barotropic fluid evolving in a  moving domain
\begin{equation*}
\Omega(t) = \left\{ y \in   \Rn{d} \;\vert\; -b< y_d  < \eta(t,y_h)\right\},\quad  t\ge 0.
\end{equation*}
Here the dimension $d=2$ or $3$, and $y=(y_h,y_d)$ for $y_h=(y_1,y_{d-1})$ the horizontal spatial coordinate and $y_d$ the vertical one. The lower boundary of $\Omega(t)$, denoted by $\Sigma_b$,  is assumed to be rigid and given by the fixed constant $b>0$, but the upper boundary, denoted by $\Sigma(t)$, is a free boundary that is the graph of the unknown function $\eta: [0, \infty) \times \Rn{d-1}\to \Rn{}$.

The fluid is described by its density and velocity, which are given for each $t\ge 0$ by $ \rho(t,\cdot):\Omega(t) \to  \Rn{+} $ and $ u(t,\cdot):\Omega(t) \to \Rn{d} $, respectively. The pressure is a function of density: $P=P(\rho)$, which is assumed to be  positive, smooth, and strictly increasing. We define the symmetric gradient of $u$ by $ \sg u  =   \nabla u  + (\nabla u)^T $ and  its deviatoric (trace-free) part by $\sgz u = \sg u - \frac{2}{d} \diverge{u} I$
with $I$ the $d \times d$ identity matrix, then the viscous stress tensor is given by
$$
\mathbb{S} u =\mu \mathbb{D}^0u+\mu' \diverge u I,
$$
where $\mu $ is the shear viscosity and  $\mu' $ is the bulk viscosity; for technical reasons (see Korn's inequality of Lemma \ref{xm5}) we assume they satisfy the following conditions:
\begin{equation}\label{viscosity}
\mu >0,\ \mu '> 0\text{ if }d=2;\ \mu >0,\ \mu '\ge 0\text{ if }d=3.
\end{equation}

 For each $t>0$ we require that $(\rho, u, \eta)$ satisfy the following free boundary problem for the compressible Navier--Stokes equations:
\begin{equation}\label{NS}
\begin{cases}
\partial_{t}\rho+\diverge(\rho u)=0& \text{in }\Omega(t)
\\ \partial_t(\rho u)+\diverge(\rho u\otimes u)+\nabla P(\rho)-\diverge \mathbb{S} u =-\rho ge_{d} &\text{in}~ \Omega(t)
\\ \partial_t \eta = u\cdot \n &\text{on }\Sigma(t)
\\ \(P(\rho)I-\mathbb{S} u \)\n =p_{atm} \n  &\text{on }\Sigma(t)
\\u=0 & \text{on }\Sigma_{b} .
\end{cases}
\end{equation}
The first equation in \eqref{NS} is the continuity equation, and the second one is the momentum equations in which $-\rho g   e_d$ is the gravitational force with the constant $g>0$ the acceleration of gravity and $e_d$ the vertical unit vector. The third equation is the kinematic boundary condition which implies that the free boundary is advected with the fluid, where $\n=(-D\eta,1)$ is the upward-pointing non-unit normal to $\Sigma(t)$ for $D$ the horizontal gradient. The fourth equation is the dynamic boundary condition which states the balance  of normal stresses on the free boundary, without taking into account the effect of surface tension, where the constant $p_{atm}>0$ is the constant atmospheric pressure. The last equation is the usual no-slip boundary condition for the velocity on the rigid bottom.
To complete the statement of the problem \eqref{NS}, we must specify the initial conditions; we suppose that the initial upper boundary $\Sigma(0)$ is given by the graph of the function $\eta(0)=\eta_0: \Rn{d-1}\rightarrow \mathbb{R}$, which yields the initial domain $\Omega(0)$ on which we specify the initial density, $ \rho(0)=\rho_0: \Omega(0) \rightarrow  \mathbb{R}^+$, and the velocity, $u(0)=u_0: \Omega(0) \rightarrow \Rn{d}$.

Note that there exists a stratified steady-state equilibrium solution to \eqref{NS} with $(u,\eta)=(0,0)$ provided certain necessary and sufficient conditions are satisfied. Indeed, the corresponding equilibrium domain is given by
$$
\Omega=\left\{y\in \mathbb{R}^d|-b<y_d<0\right\},
$$
and \eqref{NS} reduces to an ODE for the equilibrium density $\bar{\rho}=\bar{\rho}(y_d)$:
\begin{equation}\label{ode}
\begin{cases}
\displaystyle\frac{d(P(\bar{\rho}))}{dy_d}=-\bar{\rho}g   \quad\text{in }(-b,0),\\
P(\bar{\rho}(0))=p_{atm}.
\end{cases}
\end{equation}
We claim that  \eqref{ode} admits a unique solution $\bar{\rho}>0$ if and only if
\beq\label{cd}
p_{atm} \in P(\mathbb{R}^+), \text{ and } 0<b<\displaystyle\frac{1}{g}\int_{\rho^\star}^{+\infty}\frac{P'(s)}{s}\,ds \text{ with } \rho^\star:=P^{-1}(p_{atm})>0.
\eeq
To see this,  we  introduce the enthalpy function $h: (0,\infty)\rightarrow\mathbb{R}$  by
\beq\label{bh}
h(z)=\displaystyle\int_{\rho^\star}^{z}\frac{P'(s)}{s}ds,
\eeq
which is smooth, strictly increasing, and positive on $(\rho^\star, \infty)$ with $h(\rho^\star)=0$. Then the solution to \eqref{ode} is given by
\beq\label{hdef}
\bar{\rho}(y_d)=h^{-1}(h(\rho^\star)-gy_d)=h^{-1}(-gy_d),
\eeq
which gives a well-defined, smooth, and strictly decreasing function $\bar{\rho}:[-b,0]\rightarrow[\rho^\star,\infty)$ if and only if
$$
gb\in h((\rho^\star,\infty)) \Longleftrightarrow 0<b<\displaystyle\frac{1}{g}\int_{\rho^\star}^{+\infty}\frac{P'(s)}{s}ds.
$$

We will assume in this paper that \eqref{viscosity} and \eqref{cd} hold, and our purpose is to establish the global well-posedness of the problem \eqref{NS} near the equilibrium  $(\rho,u,\eta)=(\bar{\rho},0,0)$.


\subsection{Related results.}
Free boundary problems in fluid mechanics have attracted much interest in the mathematical community. It would be impossible to present a thorough survey of the literature here, and we will mention only the works most relevant to our present setting, that is, the viscous surface wave problem for layers of viscous fluid.

 For the incompressible viscous surface wave problem, Beale \cite{beale_1} proved the local well-posedness for the case without surface tension  in Lagrangian coordinates, and, thereafter, the global well-posedness and long-time behavior of solutions to the  problem without surface tension have been intriguing questions, see Sylvester \cite{Sylvester} and Tani and Tanaka \cite{TT2}.
 Note that it is not  expected  the global well-posedness to be  available for the free boundary problems in fluid mechanics with the generic large initial data,
 and we refer to   Castro,  C\'ordoba, Fefferman, Gancedo and G\'omez-Serrano \cite{CCFGG,CCFGG2} and Coutand and Shkoller \cite{Coutand1,Coutand}
  for the formation of finite-time singularities starting from certain smooth initial data. When the effect of surface tension is included, Beale \cite{beale_2}  proved the global well-posedness of the problem with the small initial data in flattening coordinates. Moreover, Beale and Nishida \cite{beale_nishida}  proved that the solution of the horizontally infinite setting in $3D$ obtained in \cite{beale_2} decays at an algebraic rate provided $\norm{\eta_{0}}_{L^{1}\(\mathbb{R}^{2}\)} $ is sufficiently small,
 see also Hataya \cite{hataya2},
while Nishida, Teramoto and Yoshihara \cite{NTY}  showed that the solution
of the  horizontally periodic setting  decays at an exponential rate.
For the case without surface tension, in the horizontally periodic setting,
 Hataya \cite{hataya} proved the global existence of the small solution with an algebraic decay rate,
 Guo and Tice \cite{GT_per} showed that the solution decays at an almost exponential rate, and Tan and Wang \cite{TW} further justified
the global-in-time vanishing surface tension limit of the problem.
In the horizontally infinite setting,  Hataya and Kawashima \cite{hataya3} announced a global existence of the small solution with an
algebraic decay in $3D$ provided $\norm{\eta_{0}}_{L^{1}(\mathbb{R}^{2})} $ is sufficiently small,
and Guo and Tice \cite{GT_inf} proved the global existence of the small solution with an  algebraic decay rate in $3D$
provided some negative Sobolev norm of the initial data is sufficiently small, see also Gui \cite{G1} for a similar result in Lagrangian coordinates.
Very recently, the second author \cite{WYJ1} removed the low frequency assumption of the initial data ($i.e.$, the assumptions that the initial data belongs to negative Sobolev spaces in \cite{GT_inf} and that $\eta_{0}\in L^{1}(\mathbb{R}^{2}) $ in \cite{hataya3}) and proved the global well-posedness of the problem in both $2D$ and $3D$ by exploiting a crucial nonlinear cancellation and the anisotropic decay estimates.

For the compressible viscous surface wave problem, we focus only on the case away from vacuum.
In the  horizontally periodic setting,
Jin and Padula \cite{JP} proved the global well-posedness for a layer of barotropic fluid with surface tension
   in Lagrangian coordinates,
  Jang, Tice and Wang \cite{WYJ2} showed the global well-posedness for multiple layers of
barotropic fluid with and without surface tension and established the global-in-time vanishing surface tension limit in flattening coordinates,
and  Huang and Luo \cite{HL} obtained the global well-posedness for a layer of heat conducting fluid without surface tension.
In the horizontally infinite setting,
Tani and Tanaka \cite{TT} proved the  global well-posedness for a layer of heat conducting fluid with surface tension,
and Gui and Zhang \cite{GZ} proved the global well-posedness for a layer of barotropic fluid without surface tension in $3D$
in Lagrangian coordinates provided some negative Sobolev norm of the initial data is sufficiently small as in \cite{G1}.
Inspired by these works, especially the second author \cite{WYJ1},
our goal of this paper is to prove the global well-posedness of the problem \eqref{NS} without surface tension in both $2D$ and $3D$,
 without any low frequency assumption of the initial data mentioned above.

\section{Main results}

\subsection{Reformulation in flattening coordinates.}
The movement of the free boundary $\Sigma(t)$
and the subsequent change of the domain $\Omega(t)$ create numerous mathematical difficulties. As Beale \cite{beale_2}, we will flatten the coordinate domain via the mapping
$$
\Omega\ni x\mapsto(x_{h},x_d+\varphi(t,x)):=\Phi(t,x)=(y_{h},y_{d})\in \Omega(t),
$$
where $\varphi(t,x)=(1+x_d/b)\bar{\eta}$ for $\bar{\eta}=\mathcal{P}\eta$ the harmonic extension of $\eta$ defined by \eqref{pp2}.

 If $\eta$ is sufficiently small and regular, then the mapping $\Phi$ is
a diffeomorphism, which allows us to transform  \eqref{NS} defined in $\Omega(t)$ to be one in $\Omega$. To this end, we define the following differential operators:
$$
\dt^\a:=\dt-J^{-1}\dt \varphi \p_d ,\ \partial_i^\a=(\naba )_i:=(\a\nabla)_i ,\ \diva  := \naba \cdot, \text{ and }\da  := \diva \naba ,
$$
where we have denoted the matrix $\mathcal{A}:=\(\nabla \Phi^{-1}\)^T$
and $J := \det{\nab \Phi}=1+\p_d\varphi$ the Jacobian of the coordinate transformation,
and we write
$$
\S_{\a} u : =\mu \sgz_{\a} u+ \mu' \diva u I, \  \sgz_{\a} u = \sg_{\a} u - \frac{2}{d} \diva u I,
 \text{ and }  \sg_{\a} u  =  \naba u+ (\naba u)^T.
$$
 Note that by \eqref{hdef},
 \begin{align}\label{qq}
 \nabla_\mathcal{A} P(\rho)+\rho g e_{d}
 = \rho\nabla_\mathcal{A}\( h(\rho)+g(x_d+  \varphi) \)
 = \rho\nabla_\mathcal{A}\( h(\rho)-h(\bar\rho)+g  \varphi \).
\end{align}
This motivates us to define the special perturbation
\beq\label{tsg}
q:= h(\rho)-h(\bar\rho)+g\varphi= h(\rho)+g(x_d+\varphi).
\eeq
Then
\beq\label{ho}
\rho=
\bar{\rho}+\frac{1}{h'(\bar{\rho})}(q-g\varphi) +\mathcal{R}_{  h^{-1}},
\eeq
where
\beq\label{Rh}
\mathcal{R}_{  h^{-1}}
= (q-g\varphi)^2\int_0^1\left(h^{-1}\right)''(h(\bar\rho)+s(q-g\varphi))(1-s)ds
,
\eeq
and on the upper boundary $\Sigma:=\mathbb{R}^{d-1}\times\{0\}$, we have $\rho =  h^{-1} \(q-g\eta\)$ and hence,
\beq\label{deqdf}
P(\rho)-p_{atm}
=P \circ h^{-1} \(q-g\eta\) -P \circ h^{-1}(0)
=\rho^\star(q-g\eta)+\mathcal{R}_{P\circ h^{-1}},
\eeq
where
\beq\label{Rp}
\mathcal{R}_{P\circ h^{-1}}
=(q-g\eta)^2\int_0^1\left(P\circ h^{-1}\right)''(s(q-g\eta))(1-s)ds
.
\eeq
Note that the first equation in \eqref{NS} is transformed to be
\beq\label{deq}
\partial_t^\mathcal{A}\rho
+\diverge_\mathcal{A}(\rho u)=0 ,
\eeq
and hence, by \eqref{tsg},
\begin{align}\label{deq2}
\partial_t^\mathcal{A} q
&=\partial_{t}^\mathcal{A} \( h(\rho) +g(x_d+  \varphi)\)
=\partial_{t}^\mathcal{A}  h(\rho)
= h'(\rho)\partial_t^\mathcal{A}\rho
=-h'(\rho)\diverge_\mathcal{A}(\rho u).
\end{align}
Then using \eqref{qq}, \eqref{deqdf} and \eqref{deq2}, the system in new coordinates reads as
\begin{equation}\label{neweq}
\begin{cases}
 \partial_t^\mathcal{A} q+h'(\rho)\diverge_\mathcal{A}(\rho u)= 0    &\text{in }\Omega
\\\rho \(\partial_{t}^\mathcal{A}u+u\cdot \nabla_\mathcal{A} u\)
+ \rho \nabla_\mathcal{A}q-\diverge_\mathcal{A}\mathbb{S}_\mathcal{A} u =0  &\text{in }\Omega
\\ \partial_t\eta=u\cdot \mathcal{N}   &\text{on }\Sigma
\\ \(\rho^\star qI-\mathbb{S}_\mathcal{A} u \)\mathcal{N}=\left(\rho^\star g\eta  -\mathcal{R}_{P\circ h^{-1}}\right)\mathcal{N}
 &\text{on }\Sigma
\\ u=0     &\text{on }\Sigma_b
\\ (q,u,\eta)|_{t=0}=(q_0,u_0,\eta_0).
\end{cases}
\end{equation}

\subsection{Statement of the results}

 We will work in a high-regularity context, essentially with regularity up to $2N$ temporal derivatives for an integer $N\ge 4$.  This requires us to use the initial data $(q_0,u_0,\eta_0)$ to construct the initial data
 $\(\dt^j q(0),\dt^j u(0),\dt^j \eta(0)\)$ for $j=1,\dotsc,2N$
   by using the equations \eqref{neweq}. These data must then satisfy various boundary conditions, which in turn require $(q_0,u_0,\eta_0)$ to satisfy $2N$ compatibility conditions.  We refer the reader to \cite{JTW} for the construction of those initial data and the precise description of the necessary
 $2N$ compatibility conditions.

We write $H^{k}(\Omega)$ with $k\geq 0$ and $H^{s}(\Sigma)$ with $s\in \mathbb{R}$ for the usual Sobolev spaces,
with norms denoted by $\norm{\cdot}_{k}$ and $\abs{\cdot}_{s}$, respectively. For a vector-valued function $v\in \mathbb{R}^{d}$, we write $v=(v_{h},v_{d})$ for $v_{h}$ the horizontal component of $v$ and $v_d$ the vertical one. Let $d = 2$ or $ 3$ and $N \geq 4$.
We define the high-order energy as
\begin{align}\label{e}
\mathcal{E}_{2N}:= \sum_{j=0}^{2N}\norm{\partial_t^ju}_{4N-2j}^2
 +\sum_{j=0}^{2N}\norm{\partial_t^j q}_{4N-2j }^2
 +\sum_{j=0}^{2N}\abs{\partial_t^j\eta}_{4N-2j }^2,
\end{align}
the high-order dissipation as
\begin{align}\label{d}
\mathfrak{D}_{2N}:=&\norm{u_h}_{4N}^{2} +\norm{\nabla_\a u_h}_{4N}^{2}
+\norm{u_d}_{4N+1}^{2}+\sum_{j=1}^{2N}\norm{\partial_t^ju}_{4N-2j+1}^2
\nonumber\\&
+\norm{\nabla q }_{4N-1}^2 +\norm{\partial_tq}_{4N-1}^2+\sum_{j=2}^{2N+1}\norm{\partial_t^jq}_{4N-2j+2}^2
\nonumber\\&
+\abs{D\eta}_{4N-3/2}^2+\abs{\partial_t\eta}_{4N-1}^2+\sum_{j=2}^{2N+1}\abs{\partial_t^j\eta}_{4N-2j+5/2}^2
\end{align}
and
\begin{align} \label{F}
\mathcal{F}_{2N}:=\abs{\eta}_{4N+1/2}^{2}.
\end{align}
We also define the low-order energy as
\begin{align}\label{le}
\mathcal{E}_{N+2,1}:=&
\sum_{j=0}^{N+2}\norm{\partial_t^ju}_{2(N+2)-2j}^2
+\norm{\nabla q}_{2(N+2)-1}^2
+\sum_{j=1}^{N+2}\norm{ \partial_t^jq}_{2(N+2)-2j}^2
\nonumber\\&
+\abs{D\eta}_{2(N+2)-1}^2
+\sum_{j=1}^{N+2}\abs{\partial_t^j\eta}_{2(N+2)-2j}^2
\end{align}
and the corresponding low-order dissipation as
\begin{align} \label{ld}
\mathcal{D}_{N+2,1}:= & \ns{D u_h }_{2(N+2) } +\abs{\partial_du_h}_{2(N+2)-1/2}^2+ \ns{  u_d }_{2(N+2)+1}
+\sum_{j=1}^{N+2}  \ns{\dt^j u}_{2(N+2)-2j+1}\nonumber
 \\& + \ns{D^2 q}_{2(N+2)-2}  + \ns{\p_d q}_{2(N+2)-1}
+\sum_{j=1}^{N+3}\norm{ \partial_t^jq }_{2(N+2)-2j+2}^2\nonumber
   \\& +\as{D^2 \eta}_{2(N+2)-5/2}   + \sum_{j=1}^{N+3} \as{\dt^j \eta}_{2(N+2)-2j+5/2}.
\end{align}
Here the subscripts ``$,1$" of $\mathcal{E}_{N+2,1}$ and $\mathcal{D}_{N+2,1}$ stem from that of $\bar{\mathcal{E}}_{N+2,1}$ (defined by \eqref{bare}), which refers to that the solution has the ``minimal derivative count 1" in the definition \eqref{bare}, $i.e.$, the solution itself without derivatives is not included in $\bar{\mathcal{E}}_{N+2,1}$.

The main result of this paper is stated as follows.
\begin{thm}\label{thm1}
Let $d=2$ or $3$ and $N\geq 4$. Assume that the initial data $ q_0, u_{0}\in H^{4N}(\Omega)$ and $\eta_{0}\in H^{4N+1/2}(\Sigma)$ are given such that the necessary $2N$ compatibility conditions for the local well-posedness of \eqref{neweq}
are satisfied. There exists an $\varepsilon_{0}>0$ so that if
$\mathcal{E}_{2N}(0)+\mathcal{F}_{2N}(0) \leq \varepsilon_{0}$, then there exists a unique solution $(q, u, \eta)$ to \eqref{neweq}
on $[0,\infty)$ which obeys the estimate
\begin{align}\label{obey}
&
\mathcal{E}_{2N}(t)+\int_0^t\mathfrak{D}_{2N}(r)dr
 +\frac{\mathcal{F}_{2N}(t)}{(1+t)}
+\int_0^t\frac{\mathcal{F}_{2N}(r)}{(1+r)^{2}}dr
+(1+t)\mathcal{E}_{N+2,1}(t)
+\int_0^t(1+r)\mathcal{D}_{N+2,1}(r)dr
\nonumber\\
&\quad
\lesssim \mathcal{E}_{2N}(0)+\mathcal{F}_{2N}(0), \quad~\forall~ t>0.
\end{align}

\end{thm}

\begin{rem}
Note that the estimate \eqref{obey} does not imply the time integrability of $\norm{u_h}_{4N+1}^2$. Indeed, by \eqref{obey}, we can prove
\begin{align}\label{Hey}
\int_0^t\frac{\norm{u_h(r)}_{4N+1}^2}{(1+r)^{\kappa_d}}dr
\lesssim
\mathcal{E}_{2N}(0)+\mathcal{F}_{2N}(0),
\end{align}
where $\kappa_2=1/2$ when $d=2$ and $\kappa_3>0$ is any sufficiently small constant when $d=3$. Moreover, following the way of proving that for $\mathcal{E}_{N+2,1}$ and $\mathcal{D}_{N+2,1}$, we could also prove
\begin{align}
(1+t)^2\mathcal{E}_{N+2,2}(t)
+\int_0^t(1+r)^2\mathcal{D}_{N+2,2}(r)\,dr
\lesssim\mathcal{E}_{2N}(0)+\mathcal{F}_{2N}(0),
\end{align}
where $``,2"$ means the ``minimal derivative count 2" and we omit to present the definitions of $\mathcal{E}_{N+2,2}$ and $\mathcal{D}_{N+2,2}$.
If we would have done this, then we could refine the estimate \eqref{Hey} in $3D$ to be with $\kappa_3=0$. We have forgone this improvement as
 the derivation of the estimates of $\mathcal{E}_{N+2,1}$ and $\mathcal{D}_{N+2,1}$ is much simpler, which makes the presentation much more clear.
\end{rem}

\begin{rem}
The argument of proving Theorem \ref{thm1} works also for the incompressible case, which yields an improvement of the result in the second author \cite{WYJ1} in which $\mathcal{F}_{2N}$ and $\abs{\eta}_{4N}^{2}$ may grow at the rate $(1+t)^{1+\vartheta}$ and $(1+t)^{\vartheta}$ for any small $\vartheta>0$, respectively; here we have $\vartheta=0$. This is due to the new control of $\norm{D_\a u}_{4N}^{2}$ for $D_\a:=(\nabla_\a)_h$, which benefits the cancelation ``earlier" comparing with the good unknown $D^{4N}u-\p_d^\a u D^{4N}\varphi$ considered in \cite{WYJ1}.
\end{rem}

\subsection{Strategy of the proof}\label{sj}
The local well-posedness of \eqref{neweq} in our functional
framework under the assumption of Theorem \ref{thm1} can be obtained by taking the one-phase problem here as a special case of the two-phase problem considered in \cite{JTW}; even \cite{JTW} deals with the horizontally periodic case, but the analysis holds also for the horizontally infinite case here. Therefore, by a standard continuity argument as in \cite{GT_inf} for instance, to prove Theorem \ref{thm1} it suffices to derive the {\it a priori} estimates for the solution to \eqref{neweq} as recorded in Theorem \ref{main}.
The proof of Theorem \ref{main} is inspired by those works \cite{GT_inf,WYJ1,WYJ2}. In the below, we will sketch the main steps, and explain the key new difficulties and ideas.

Let us denote $\mathcal{G}_{2N}(t)$
by the supremum in $[0,t]$ of the left hand side of \eqref{obey}.
Note that the physical energy-dissipation identity in terms of the density $\rho$ ($i.e.,$ (2.34) in \cite{WYJ2}) can not be used directly for the horizontally infinite setting here, and we need to utilize a variant in terms of the perturbation $q$ obtained from linearizing around the equilibrium: by \eqref{neweq},
\begin{align}\label{intro0}
&\frac{1}{2}\frac{d}{dt}\left(
\int_{\Omega}J\(\frac{1}{h'(\rho)} \abs{q}^2
+ \rho   \abs{u}^2\)
+\int_{\Sigma}\rho^\star g\abs{ \eta}^2\right)
\nonumber\\&\quad+\int_{\Omega}J\(\frac{\mu}{2}\abs{\mathbb{D}_\mathcal{A}^0u}^2
+\mu' \abs{\diverge_\mathcal{A}u}^2\)=h.o.t..
\end{align}
However, differently from the incompressible case \cite{GT_inf,WYJ1} or the compressible case of the horizontally periodic setting \cite{WYJ2}, the nonlinear terms $h.o.t.$ in the right hand side of \eqref{intro0} can not be bounded directly by $\sqrt{\se{2N}}\fd{2N}$; there are some exceptional terms with fewer derivatives like $\int_\Sigma  |q|^2\dt\eta$, etc., which are innocently controlled by
$\se{2N}\sqrt{\fd{2N}}$ that is harmful for the global-in-time stability analysis. The key ingredient here is to refine these nonlinear estimates by making use of the time weighted dissipation estimates: for instance, by using interpolation estimates in \eqref{ei},
\begin{align}\label{intro1}
\int_\Sigma  |q|^2\dt\eta\ls \norm{q}_{L^\infty(\Sigma)}\abs{q}_0\abs{\dt\eta}_0
&\ls\sqrt{\mathcal{D}_{N+2,1}^{1/4}\mathcal{E}_{2N}^{3/4}}\sqrt{ \mathcal{E}_{2N} }\sqrt{\mathcal{D}_{N+2,1} }
\nonumber\\&= \mathcal{E}_{2N}^{7/8}((1+t)\mathcal{D}_{N+2,1})^{5/8}(1+t)^{-5/8},
\end{align}
which is time integrable in our functional framework.

Basing on the energy-dissipation structure of \eqref{intro0}, we next derive the higher order tangential energy evolution estimates by applying the horizontal spatial and time derivatives to \eqref{neweq}. However, as for the incompressible case \cite{GT_inf,WYJ1}, when estimating the $4N$ order horizontal spatial derivatives of the solution there are some other nonlinear terms of higher regularity that can not be controlled by $\sqrt{\se{2N}}\fd{2N}$ but rather with the estimates involving $\f$; the trouble lies in the control of $\f$, since the only way to estimate it is through the kinematic boundary condition, which leads to a potential time growth as already seen in \eqref{obey}. Recall here that those nonlinear terms are  the ones in the bulk related to the commutators between $D^{4N}$ and $\partial_i^\a u$ in the viscous stress tensor, which lead to the nonlinear estimates $\sqrt{\norm{\nabla u}_{L^\infty(\Omega)}^2\f\fd{2N}}$, and the ones on the boundary related to the commutators when using the dynamic boundary condition, which lead to the nonlinear estimates $\sqrt{\norm{\nabla u}_{C^1(\Sigma)}^2\f\fd{2N}}$. Note that, thanks to the crucial anisotropy, $\norm{\nabla u}_{C^1(\Sigma)}^2\ls \mathcal{D}_{N+2,1}$ and hence by making full use of the time weighted estimates in the definition of $\mathcal{G}_{2N}$ the time integration of the latter nonlinear estimates is bounded by $(\mathcal{G}_{2N}(t))^{3/2}$, however, we have only $\norm{\nabla u}_{L^\infty(\Omega)}^2\ls \mathcal{D}_{N+2,1}^{1-\gamma}\se{2N}^{\gamma}$ for some $\gamma>0$, whose time weighted estimate is not sufficient for controlling the former nonlinear estimates. The key point here is, motivated by \cite{WYJ1}, to  estimate instead the $4N-1$ order horizontal spatial derivatives of the ``Eulerian horizontal spatial derivatives" $D_\mathcal{A}$ of the solution;
since $\partial_j^\mathcal{A}$ commutes with $\partial_i^\mathcal{A}$, the former nonlinear terms in the bulk are then canceled out. Such idea dates back to the work of Alinhac \cite{Alinhac}, see also Masmoudi and Rousset \cite{Masmoudi} and Wang and Xin \cite{WX}.

After having controlled the horizontal spatial and time derivatives of the solution, we then turn to control the full derivatives by employing the elliptic estimates. However, unlike the incompressible case \cite{GT_inf,WYJ1}, the tangential energy and dissipation here do not control everything necessary for these elliptic estimates. More precisely, as \cite{WYJ2} we will use the Lam\'e system for the energy estimates, which requires the control of $\p_d  q$, and use the Stokes system for the dissipation estimates, which requires the control of $\p_d \dt q$. These estimates of $\p_d  q$ and $\p_d \dt q$ follow from the energy evolution estimates for $\p_d q$, basing on an evolution equation for $\p_d q$ with the damping of dissipative structure that resembles the ODE $\dt f+f=g$, which arises by taking an appropriate linear combination of the continuity equation and the vertical component of the momentum equations, see also Kanel' \cite{Kanel} and Matsumura and Nishida \cite{Matsumura} for the case with $\bar{\rho}$ a positive constant.
Again, to avoid the appearance of $\f$ in the nonlinear estimates here, we will need to consider instead the evolution equation for $\partial_d \partial_i^\mathcal{A} q$ (involving $ \partial_i^\mathcal{A} u$) when estimating for the $4N-1$ order spatial derivatives of $\p_d q$.
By combining these energy evolution estimates and elliptic regularity estimates in a recursive way in terms of the number of vertical derivatives, together with using the equations in \eqref{neweq}, we get the desired full energy-dissipation estimates of $\mathcal{E}_{2N}$  and $\mathfrak{D}_{2N}$. Then a time weighted argument basing on the transport estimates for the kinematic boundary condition leads to the desired time weighted (growth) estimates of $\mathcal{F}_{2N}$.

Now, the remaining thing is to show the low-order decay estimates of $\mathcal{E}_{N+2,1}$ and $\mathcal{D}_{N+2,1}$.
The strategy is mostly similar to that of high-order energy estimates, but with the nonlinear estimates needed to be controlled by $\sqrt{\se{2N}}\mathcal{D}_{N+2,1}$ (or $\se{2N}\mathcal{E}_{N+2,1}$ for the estimates in the energy).
The trouble lies in that $\mathcal{D}_{N+2,1}$ has a minimal count of derivatives, but we need to estimate terms with fewer derivatives. We must then resort to various interpolation estimates as recorded in Lemma \ref{q,u} and make full use of the anisotropy  of $\mathcal{D}_{N+2,1}$ and the structure of the nonlinear terms. Furthermore, another new difficulty compared to the high-order energy estimates is that the crucial needed dissipation estimate of $\norm{\dt q}_0^2$ here can not be derived from the tangential dissipation estimate of $\norm{Du}_1^2$
(with minimum derivatives count 1). Our way here is to apply the time derivative to the vertical component of the momentum equations and then test the resulting by $u_d$, which yields the desired estimates of $\norm{\dt q}_0^2$. We thus conclude that
$
 \frac{d}{dt}\mathcal{E}_{N+2,1}
+ \mathcal{D}_{N+2,1}
\le 0,
$
which together with a time weighted argument leads to the desired time weighted (decay) estimates of $\mathcal{E}_{N+2,1}$ and $\mathcal{D}_{N+2,1}$.

Consequently, summing over these estimates closes the a priori estimates \eqref{main1} provided that $\mathcal{E}_{2N}(0) +\mathcal{F}_{2N}(0)$ is sufficiently small. The proof of Theorem \ref{main} is thus completed.

\subsection{Notation}\label{nota}

We write $\mathbb{N }= \{0, 1, 2, \ldots \}$ for the collection of non-negative integers. When using
space-time differential multi-indices, we write $\mathbb{N }^{1+m} = \{\alpha = (\alpha_0, \alpha_1, \ldots , \alpha_m)\}$ to emphasize that the 0-index term is related to temporal derivatives. For just spatial derivatives we write $\mathbb{N }^m$. For $\alpha \in \mathbb{N }^{1+m} $, we write $\partial^{\alpha}=\partial_{t}^{\alpha_{0}}\partial_{1}^{\alpha_{1}}\ldots\partial_{m}^{\alpha_{m}}$. We define the parabolic
counting of such multi-indices by writing $|\alpha|=2\alpha_{0}+\alpha_{1}+\ldots+\alpha_{m}$.
For $\alpha \in \mathbb{N }^{m} $ or $\mathbb{N }^{1+m} $, we denote $\alpha_h$ for its horizontal component.
For integers $k, \ell\geq 0$, we denote the anisotropic Sobolev norm
\begin{equation}\label{def3}
\norm{f}_{k,\ell} :=\sum_{\alpha\in\mathbb{N}^{d-1},|\alpha|\leq \ell}\norm{\partial^{\alpha}f}_{k}.
\end{equation}
We define the standard commutator
\begin{equation}
\[\partial^{\alpha},f\]g:=\partial^{ \alpha}(fg)-f\partial^{\alpha}g.
\end{equation}

We employ the Einstein convention of summing over repeated indices for vector operations.
$\delta_{ij}$ stands for the kronecker symbol. $C>0$ denotes a generic constant that can depend on the parameters of the problem,
$N\geq 4$, $d = 2 $ or 3, but does not depend on the data, etc. We refer to such constants as
``universal''. They are allowed to change from line to line. We employ the notation $A_{1}\lesssim A_{2}$
to mean that $A_{1}\leq C A_{2}$ for a universal constant  $C>0$, and $A_1\sim A_2$ means that $A_1\lesssim A_2$ and $A_2\lesssim A_1$.
To avoid the constants in various time differential inequalities, we employ the following two conventions:
\begin{align}\label{timeq1}
\partial_t A_1+A_2\lesssim A_3
~\text{ means }
\partial_t \tilde{A}_1+A_2\lesssim A_3
~
\text{ for ~any }
A_1\sim\tilde{A}_1
\end{align}
and
\begin{align}\label{timeq2}
\partial_t( A_1+A_2)+A_3\lesssim A_4
~\text{ means }
\partial_t (C_1 A_1+C_2A_2)+A_3\lesssim A_4
~
\text{ for~ constants }
C_1, C_2 > 0.
\end{align}

We omit the differential elements of the integrals over $\Omega $ and $\Sigma ${\color{blue}.}

\section{Preliminaries}
We assume throughout the rest of the paper that the solution to \eqref{neweq} is given on the time interval $[0,T]$ with $T>0$
and obey the a priori assumption
\begin{equation}\label{apriori1}
\mathcal{E}_{2N}(t)\leq  \delta,\quad\forall~ t\in[0,T]
\end{equation}
for an integer $N\geq 3$ and a sufficiently small constant $\delta> 0$. This implies in particular that
\begin{align}\label{aprii}
\frac{1}{2}\leq J(t,x)\leq \frac{3}{2}
\text{ and }
0<\frac{\rho^\star}{2} \leq \rho(t,x) \leq C <+\infty
, \quad\forall~ (t,x)\in[0,T]\times\bar{\Omega}.
\end{align}

\subsection{Perturbed linear form}
In order to use the linear structure of \eqref{neweq}, we rewrite it  as a perturbation of the linearized system:
\begin{equation}\label{q2}
\begin{cases}
\partial_t q+h'(\bar{\rho})\diverge(\bar{\rho}u)=G^1   &\text{in }\Omega
\\ \bar{\rho}\partial_tu+\bar{\rho}\nabla q
-\diverge\mathbb{S}u=G^2   &\text{in }\Omega
\\ \partial_t\eta=u_d+G^3  &\text{on }\Sigma
\\(\rho^\star qI-\mathbb{S}u)e_d=\rho^\star g\eta e_d+G^4  &\text{on }\Sigma
\\ u=0     &\text{on }\Sigma_b.
\end{cases}
\end{equation}
Here we have written the function  $G^1=G^{1,1}+G^{1,2}$ for
\begin{align}
&G^{1,1}=J^{-1}\partial_t\varphi\partial_dq-u_l\mathcal{A}_{lk}\partial_kq,\label{G11}\\
&G^{1,2}=gu_l\mathcal{A}_{lk}\partial_k\varphi
-h'(\bar{\rho})(\mathcal{A}_{lk}-\delta_{lk})\partial_k(\bar{\rho}u_l)
-(P'(\rho)-P'(\bar{\rho}))\mathcal{A}_{lk}\partial_ku_l,\label{G12}
\end{align}
the vector $G^2=G^{2,1}+G^{2,2}+G^{2,3}+G^{2,4}$ for $i=1,d-1,d$,
\begin{align}
&G_i^{2,1}=-\bar{\rho}(\mathcal{A}_{il}-\delta_{il})\partial_lq
           -(\rho-\bar{\rho})\mathcal{A}_{il}\partial_lq,\label{G21}\\
&G_i^{2,2}=\mu(\mathcal{A}_{lk}\mathcal{A}_{lm}-\delta_{lk}\delta_{lm})\partial_{km}u_i
+\left(\frac{d-2}{d}\mu+\mu'\right)(\mathcal{A}_{ik}\mathcal{A}_{lm}-\delta_{ik}\delta_{lm})\partial_{km}u_l,\label{G22}\\
&G_i^{2,3}=\mu\mathcal{A}_{lk}\partial_k\mathcal{A}_{lm}\partial_mu_i
+\left(\frac{d-2}{d}\mu+\mu'\right)\mathcal{A}_{ik}\partial_k\mathcal{A}_{lm}\partial_mu_l,\label{G23}\\
&G_i^{2,4}=-(\rho-\bar{\rho})\partial_tu_i
-\rho\left(-J^{-1}\partial_t\varphi\partial_du_i+u_l\mathcal{A}_{lk}\partial_ku_i\right),\label{G24}
\end{align}
the function $G^3$ for
\begin{align}\label{G4}
G^3=-u_h\cdot D\eta,
\end{align}
and
the vector $G^4$ for $i=1,d-1$,
\begin{align}
G_i^4
=&-\mu D \eta \cdot\partial_i u_h-\mu D \eta \cdot D u_i
-\mu\partial_i \eta \mathcal{N}_k(\partial_k u_d+\partial_d u_k)
\nonumber\\&
+\mu((\mathcal{A}_{km}-\delta_{km})\partial_m u_l
+(\mathcal{A}_{lm}-\delta_{lm})\partial_m u_k)\mathcal{N}_l\tau_k^i,
\end{align}
where $\tau^i=e_i+\partial_i\eta e_d$, and
\begin{align}\label{g4}
G_d^4&=
\mu ((\mathcal{A}_{km}-\delta_{km})\partial_m u_l
+(\mathcal{A}_{lm}-\delta_{lm})\partial_m u_k)\mathcal{N}_l\mathcal{N}_k|\mathcal{N}|^{-2}
\nonumber\\&\quad
+\mu(\partial_m u_l +\partial_l u_m )\mathcal{N}_m\mathcal{N}_l\left(|\mathcal{N}|^{-2}-1\right)
+\mu \sum_{k+j< 2d}(\partial_k u_j +\partial_j u_k )\mathcal{N}_k\mathcal{N}_j
\nonumber\\&\quad
+\left(\mu'-\frac{2\mu}{d}\right)(\mathcal{A}_{lm}-\delta_{lm})\partial_m u_l
-\mathcal{R}_{P\circ h^{-1}}.
\end{align}

\subsection{Interpolation estimates}\label{pola}
Since $\mathcal{E}_{N+2,1}$ and $\mathcal{D}_{N+2,1}$ have a minimal count of derivatives, it may cause the trouble
when we need to estimate terms with fewer derivatives in terms of $\mathcal{E}_{N+2,1}$ and $\mathcal{D}_{N+2,1}$.
To handle this,
 we will prove various interpolation estimates of the form
$$
  \norm{X}^{2}\lesssim(\mathcal{E}_{N+2,1})^{\theta}(\mathcal{E}_{2N})^{1-\theta} \text{ and }
  \norm{X}^{2}\lesssim(\mathcal{D}_{N+2,1})^{\theta}(\mathcal{E}_{2N})^{1-\theta},
$$
where $\theta \in[0,1]$ and  $\norm{\cdot}$ is $L^2$ or $L^\infty$ norm.
For brevity, we will record these interpolation estimates in tables that only list the value of $\theta$; for example,
$$
\begin{array}{| c | c  | c |}
\hline
 L^\infty   \mid  \mathcal{E}_{N+2,1}   &    2D \\ \hline
\eta & 1/2\\ \hline
\end{array}
\text{~ means that~}
\|\eta\|_{L^\infty(\Sigma)}^{2}\lesssim(\mathcal{E}_{N+2,1})^{1/2}(\mathcal{E}_{2N})^{1/2} \text{ in } 2D.
$$

We shall record the interpolation estimates in the following lemma, where the norms for $q,u,\bar{\eta},G^{1}$ and $G^{2}$ are on $\Omega$
and the norms for $\eta $, $G^{3} $ and $G^{4}$ are on $\Sigma$. In the below $r>0$ denotes for any small positive constant.

\begin{lem}\label{q,u}

Let $(q,u,\eta)$ be the solution to \eqref{q2} and $G^i$ be defined in \eqref{G11}--\eqref{g4}.

\begin{itemize}

\item[(1)] The following tables encode the $L^2$ and $L^\infty$ interpolation estimates for the solution and their derivatives:
\beq \label{ei}
\begin{split}
&\,\,\,\begin{array}[t]{| c | c  | c |}
\hline
 L^2   \mid  \mathcal{E}_{N+2,1}   &    2D & 3D \\ \hline
\eta, \bar\eta & 0 & 0\\ \hline
q & 0 & 0\\ \hline
\end{array}
\quad
\begin{array}[t]{| c | c  | c |}
\hline
 L^\infty   \mid \mathcal{E}_{N+2,1}      &     2D &3D \\ \hline
\eta, \bar\eta & 1/2 & 1/(1+r)\\ \hline
q & 1/2 & 1/(1+r)\\ \hline
\end{array}
\\&
\begin{array}[t]{| c | c | c |}
\hline
 L^2   \mid \mathcal{D}_{N+2,1}  & 2D &3D      \\ \hline
\eta, \bar\eta & 0 & 0\\ \hline
D\eta, \nabla\bar\eta & 1/2 & 1/2\\ \hline
q & 0 & 0\\ \hline
Dq & 1/2 & 1/2\\ \hline
\partial_d^k u_h    & 1/2  & 1/2  \\ \hline
\end{array}
 \quad
\begin{array}[t]{| c | c  | c |}
\hline
 L^\infty   \mid \mathcal{D}_{N+2,1}  & 2D &3D      \\ \hline
\eta, \bar\eta & 1/4 & 1/2\\ \hline
D\eta, \nabla\bar\eta & 3/4 & 1/(1+r)\\ \hline
q & 1/4 & 1/2\\ \hline
Dq & 3/4 & 1/(1+r)\\ \hline
\partial_d^k u_h   & 3/4 & 1/(1+r)\\ \hline
\end{array}
\end{split}
\eeq
Here $k=0,1,\ldots,2(N+2)+1 $.

\item[(2)] The following tables encode the $L^2$ and $L^\infty$ interpolation estimates for the nonlinear terms $G^i$ and their derivatives:
\beq\label{Ge}
\begin{split}
&\begin{array}[t]{| c | c  | c |}
\hline
 L^2   \mid  \mathcal{E}_{N+2,1}   &    2D & 3D \\ \hline
 G^{1}   & 1 & 1   \\ \hline
  G^{2}  & 1 & 1   \\ \hline
 G^{3}  & 1 & 1\\  \hline
  G_h^{4}  & 1 & 1\\  \hline
  G_d^{4}  & 1/2 & 1/(1+r)\\  \hline
 D G_d^{4}  & 1 & 1\\  \hline
 \partial_t G_d^{4}  & 1 & 1\\  \hline
\end{array}
 \quad
 \begin{array}[t]{| c | c  | c |}
\hline
 L^\infty   \mid \mathcal{E}_{N+2,1}   &    2D & 3D  \\ \hline
 G^{1}   & 1  & 1   \\ \hline
  G^{2}  & 1  & 1   \\ \hline
  G^{3}  & 1& 1  \\  \hline
 G^{4}  & 1 & 1 \\  \hline
\end{array}
\\& \quad\quad\,
\begin{array}[t]{| c | c  | c |}
\hline
 L^2  \mid  \mathcal{D}_{N+2,1}   &    2D & 3D   \\ \hline
 G^{1}   &1  & 1 \\ \hline
  G_h^{2}  & 3/4   & 1 \\ \hline
D G_h^{2}  & 1   & 1 \\ \hline
 \partial_t G_h^{2}  & 1 & 1\\  \hline
G_d^{2}  & 1   & 1 \\ \hline
 G^{3}    & 1 & 1\\  \hline
  G_h^{4}  & 1& 1\\  \hline
 G_d^{4}  & 1/4& 1/2\\  \hline
 DG_d^{4}  & 3/4& 1\\  \hline
 D^2G_d^{4}  & 1& 1\\  \hline
 \partial_t G_d^{4}  & 1 & 1\\  \hline
\end{array}
 \quad
\begin{array}[t]{| c | c  | c |}
\hline
 L^\infty   \mid  \mathcal{D}_{N+2,1}   &    2D & 3D   \\ \hline
 G^{1}   &1  & 1 \\ \hline
  G^{2}  & 1   & 1 \\ \hline
 G^{3}    & 1 & 1\\  \hline
  G_h^{4}  & 1& 1\\  \hline
 G_d^{4}  & 1/2& 1\\  \hline
 D G_d^{4}  & 1 & 1\\  \hline
 \partial_t G_d^{4}  & 1 & 1\\  \hline
\end{array}
\end{split}
\eeq
\end{itemize}
\end{lem}

\begin{proof}
First, the interpolation estimates of $\eta,\bar{\eta},D\eta,\nabla\bar{\eta}$ and $q, Dq$ as recorded in the tables of \eqref{ei} follow by Lemmas \ref{pp2}--\ref{pp13} and the definitions \eqref{le} of $\mathcal{E}_{N+2,1}$ and \eqref{ld} of $\mathcal{D}_{N+2,1}$, which could not be improved.

Next, we estimate for $u_h$.
The horizontal component of the second equation in \eqref{q2} implies
\begin{align}\label{ro}
\partial_d^2u_h\sim \bar{\rho}Dq+\bar{\rho}\partial_t u_h+D\nabla u+G_h^{2,1}+D\bar{\eta}\partial_d^2 u_d+G_h^{2,3}+G_h^{2,4}+G_h^{2,5}.
\end{align}
It is then straightforward to check that the interpolation estimates of $\partial_d^2u_h$
in the tables of \eqref{ei} are determined by those of $Dq$, where \eqref{apriori1}, \eqref{aprii} and the fact that $\mathcal{E}_{N+2,1}$,
$\mathcal{D}_{N+2,1}\lesssim\mathcal{E}_{2N}$ have been implicitly used.
Note that due to the factor $\bar{\rho}$ in the first term $\bar{\rho}Dq$ of \eqref{ro},
the interpolation estimates of $\partial_d^{k}u_h$, $3\leq k\leq 2(N+2)+1$, are still determined by those of $Dq$.
On the other hand, since $u_h=0$ on $\Sigma_b$ and $\Omega$ is of finite depth $b$ in the vertical direction, by using the Poincar\'{e}-type inequalities of Lemma \ref{poincare_b}, we have
\begin{align}
&\norm{u_h}_0^2\lesssim\norm{\partial_du_h}_0^2\lesssim \abs{\partial_du_h}_0^2+\norm{\partial_d^2u_h}_0^2
\end{align}
and
\begin{align}
\norm{u_h}_{L^\infty(\Omega)}^2
\lesssim
\norm{\partial_du_h}_{L^\infty(\Omega)}^2
\lesssim \norm{\partial_du_h}_{L^\infty(\Sigma)}^2+\norm{\partial_d^2u_h}_{L^\infty(\Omega)}^2,
\end{align}
which imply that the interpolation estimates of $u_h$ and  $\partial_du_h$ in the tables of \eqref{ei} are determined by those of $\partial_d^{2}u_h$.
Hence, the proof of \eqref{ei} is concluded.

With the estimate \eqref{ei} in hand, it is then fairly routine to prove these tables of \eqref{Ge}.
We take the derivation of the estimate of $G_h^2$ for $2D$ in the first table of the second line in \eqref{Ge} for example.
Note that the worst terms are $\(\rho-\bar{\rho}\)\mathcal{A}_{il}\partial_lq$ for $i,l=1,d-1$ in $G_h^{2,1}$
and $\mu\(\mathcal{A}_{dd}^2-1\)\partial_d^2 u_h$ in $G_h^{2,2}$. By \eqref{ei}, we have
\begin{align}
\norm{(\rho-\bar{\rho})\mathcal{A}_{il}\partial_lq}_0^2
\lesssim
\(\norm{q}_0^2+\norm{\bar{\eta}}_0^2\)\norm{Dq}_{L^\infty(\Omega)}^2
\lesssim
\mathcal{E}_{2N}\mathcal{D}_{N+2,1}^{3/4}\mathcal{E}_{2N}^{1/4}
\end{align}
and
\begin{align}
\norm{\mu (\mathcal{A}_{dd}^2-1 )\partial_d^2 u_h}_0^2
\lesssim
\norm{\bar{\eta}}_1^2\norm{\partial_d^2 u_h}_{L^\infty(\Omega)}^2
\lesssim
\mathcal{E}_{2N}\mathcal{D}_{N+2,1}^{3/4}\mathcal{E}_{2N}^{1/4}.
\end{align}
It is remarked that, again, due to the factor $\rho-\bar{\rho}$ (cf. \eqref{ho}) the interpolation estimates of $\partial_d^kG_h^{2}$, $k\geq 1$ are same as $G_h^{2}$.
\end{proof}

\begin{rem}
It is easy to see that if the value $\theta$ of the interpolation estimate for a nonlinear term
is 1, then so is the one for its derivative
provided the derivative count does not exceed those of $\mathcal{E}_{N+2,1}$ or $\mathcal{D}_{N+2,1}$.
\end{rem}


\subsection{Nonlinear estimates}
We now present the estimates of the nonlinear terms $G^i$.
We first record the estimates  at the $N+2$ level.
\begin{lem}\label{lem1}
It holds that
\begin{align}\label{ne1}
&\sum_{j=0}^{N+1}\norm{\partial_t^jG^1}_{2(N+2)-2j-2}^2
+\sum_{j=0}^{N+1}\norm{\partial_t^jG^2}_{2(N+2)-2j-2}^2
+\sum_{j=0}^{N+1}\abs{\partial_t^jG^3}_{2(N+2)-2j-3/2}^2
\nonumber\\&\quad
+\abs{G_h^4}_{2(N+2)-3/2}^2
+\abs{DG_d^4}_{2(N+2)-5/2}^2
+\sum_{j=1}^{N+1}\abs{\partial_t^jG^4}_{2(N+2)-2j-3/2}^2
\lesssim\mathcal{E}_{2N}\mathcal{E}_{N+2,1}
\end{align}
and
\begin{align}\label{ne2}
&
\sum_{j=0}^{N+2}\norm{\partial_t^jG^{1}}_{2(N+2)-2j}^2
+\norm{DG_h^2}_{2(N+2)-2}^2
+\norm{G_d^2}_{2(N+2)-1}^2
+\sum_{j=1}^{N+1}\norm{\partial_t^jG^2}_{2(N+2)-2j-1}^2
\nonumber\\
&\quad
+\sum_{j=0}^{N+2}\abs{\partial_t^jG^3}_{2(N+2)-2j+1/2}^2
+\abs{G_h^4}_{2(N+2)-1/2}^2
+\abs{D^2G_d^4}_{2(N+2)-5/2}^2
\nonumber\\
&\quad
+\sum_{j=1}^{N+1}\abs{\partial_t^jG^4}_{2(N+2)-2j-1/2}^2
\lesssim\mathcal{E}_{2N}\mathcal{D}_{N+2,1}.
\end{align}
\end{lem}

\begin{proof}
 The estimates of these nonlinearities are fairly routine to derive: we note that all
$G^i$ terms are quadratic or of higher order; then we apply the differential operator and
expand by using the Leibniz rule; each term in the resulting sum is also at least quadratic,
and we estimate one factor in $L^2$ or $H^{1/2}$ (depending on $G^i$)
and other factors in $L^\infty$ or $C^1$, respectively,
using Sobolev embeddings, the trace theory and Lemmas \ref{pp2}--\ref{pp13} and  \ref{va1}.

 Precisely, by using the interpolation estimates
in the $L^{2}$ tables of \eqref{Ge} in Lemma \ref{q,u},
the estimate \eqref{ne1} follows directly
since there are  no terms whose derivative count involved exceed those of $\mathcal{E}_{N+2,1}$.
For the estimate \eqref{ne2}, we note that the derivative count of some terms exceeds those of $\mathcal{D}_{N+2,1}$. For these terms, we estimate one factor which is included in $\mathcal{D}_{N+2,1}$ by $\mathcal{D}_{N+2,1}$
and other factors multiplying it by $\mathcal{E}_{2N}$,
with the following exceptions when estimating the highest order spatial derivatives:
$\norm{\nabla^{2(N+2)+1}qu_h}_{0}^{2}$ when estimating $G^{1,1}$,
$\norm{\nabla^{2(N+2)+1}\bar{\eta}(u_h,\partial_d u_h)}_{0}^{2}$
when estimating  $G^{1,2}$ and $G^2$,
and $\abs{D^{2(N+2)}\eta u_h}_{1/2}^{2}$ and  $\abs{D^{2(N+2)+1}\eta u_h}_{1/2}^{2}$ when estimating $G^{3}$.
To control them, by the Sobolev interpolation, we have
\begin{align}\label{POP}
\norm{\nabla^{2(N+2)+1}q}_{0}^{2}
 \leq
\norm{\nabla^{2}q}_{2(N+2)-1}^{2}
 &\leq
 \left( \norm{\nabla^{2}q}_{2(N+2)-2}^{2}\right)^{(2N-5)/(2N-4)}
 \left(\norm{\nabla^{2}q}_{4N-2}^{2}\right)^{1/(2N-4)} \nonumber
 \\&\lesssim \mathcal{D}_{N+2,1}^{(2N-5)/(2N-4)}\mathcal{E}_{2N}^{1/(2N-4)},
\end{align}
\begin{align}\label{sob1}
\abs{D^{2(N+2)}\eta}_{1/2}^{2}
 \leq
  \abs{D^{2}\eta}_{2(N+2)-3/2}^{2}
 &\leq
 \left(\abs{D^2\eta}_{2(N+2)-5/2}^{2}\right)^{(4N-9)/(4N-7)}
 \left(\abs{D^2\eta}_{4N-2}^{2}\right)^{2/(4N-7)} \nonumber
 \\&\leq \mathcal{D}_{N+2,1}^{(4N-9)/(4N-7)}\mathcal{E}_{2N}^{2/(4N-7)}
\end{align}
and
\begin{align}\label{sobbb}
 \abs{D^{2(N+2)+1}\eta}_{1/2}^{2}
 \leq
 \abs{D^{2}\eta}_{2(N+2)-1/2}^{2}
 &\leq
 \left(\abs{D^2\eta}_{2(N+2)-5/2}^{2}\right)^{(4N-11)/(4N-7)}
 \left(\abs{D^2\eta}_{4N-2}^{2}\right)^{4/(4N-7)}
 \nonumber\\&
 \leq
 \mathcal{D}_{N+2,1}^{(4N-11)/(4N-7)}\mathcal{E}_{2N}^{4/(4N-7)},
\end{align}
which together with the $L^{\infty}$ table in \eqref{ei} in terms of $\mathcal{D}_{N+2,1}$ imply that
\begin{align}
\norm{\nabla^{2(N+2)+1}q u_h}_{0}^{2}
&\lesssim
\norm{\nabla^{2(N+2)+1}q}_{0}^{2}
\norm{u_h}_{L^\infty(\Omega)}^{2}
\nonumber
 \\&
\lesssim
\mathcal{D}_{N+2,1}^{(2N-5)/(2N-4)}\mathcal{E}_{2N}^{1/(2N-4)}\mathcal{D}_{N+2,1}^{3/4}\mathcal{E}_{2N}^{1/4}
\lesssim \mathcal{D}_{N+2,1}\mathcal{E}_{2N},
\end{align}
by  Lemma \ref{pp2},
\begin{align}\label{sob5}
\norm{\nabla^{2(N+2)+1}\bar{\eta}(u_h,\partial_d u_h)}_{0}^{2}
 &\lesssim \norm{\nabla^{2(N+2)+1}\bar{\eta}}_{0}^{2}
\norm{ u_h}_{C^1(\Omega)}^{2}
  \lesssim
  \abs{D^{2(N+2)}\eta}_{1/2}^{2} \norm{u_h}_{C^1(\Omega)}^{2}
 \nonumber
 \\&
 \lesssim \mathcal{D}_{N+2,1}^{(4N-9)/(4N-7)}\mathcal{E}_{2N}^{2/(4N-7)}\mathcal{D}_{N+2,1}^{3/4}\mathcal{E}_{2N}^{1/4}
 \lesssim \mathcal{D}_{N+2,1}\mathcal{E}_{2N},
 \end{align}
 and by Lemma \ref{va1},
 \begin{align}
\abs{D^{2(N+2)}\eta u_h}_{1/2}^{2}
  &\lesssim
  \abs{D^{2(N+2)}\eta}_{1/2}^{2} \norm{u_h}_{C^1(\Sigma)}^{2}
 \lesssim \mathcal{D}_{N+2,1}\mathcal{E}_{2N}
 \end{align}
and
 \begin{align}\label{sob3}
\abs{D^{2(N+2)+1}\eta u_h}_{1/2}^{2}
  &\lesssim
 \abs{D^{2(N+2)+1}\eta }_{1/2}^{2} \norm{ u_h}_{C^1(\Sigma)}^{2}
  \nonumber
 \\&\lesssim \mathcal{D}_{N+2,1}^{(4N-11)/(4N-7)}\mathcal{E}_{2N}^{4/(4N-7)}\mathcal{D}_{N+2,1}^{3/4}\mathcal{E}_{2N}^{1/4}
 \lesssim\mathcal{D}_{N+2,1}\mathcal{E}_{2N},
 \end{align}
 where $N\geq 4$ has been used.
 The proof of \eqref{ne2} is thus completed.
\end{proof}
Now we record the estimates  at the $2N$ level.
\begin{lem}\label{N2}
 It holds that
\begin{align}\label{ne3}
&\sum_{j=0}^{2N-1} \norm{\partial_t^jG^1}_{4N-2j-2}^2
+\sum_{j=0}^{2N-1} \norm{\partial_t^jG^2}_{4N-2j-2}^2
+\sum_{j=0}^{2N-1}\abs{\partial_t^jG^3}_{4N-2j-3/2}^2
\nonumber\\&\quad
+\sum_{j=0}^{2N-1}\abs{\partial_t^jG^4}_{4N-2j-3/2}^2
\lesssim\left(\mathcal{E}_{2N}\right)^2
\end{align}
and
\begin{align}\label{ne5}
& \norm{G^1}_{4N-1}^2
+ \norm{G^2}_{4N-2}^2
+\abs{G^3}_{4N-1}^2
+\abs{DG^4}_{4N-5/2}^2
\nonumber\\&\quad
+\sum_{j=1}^{2N} \norm{\partial_t^jG^1}_{4N-2j}^2
+\sum_{j=1}^{2N-1} \norm{\partial_t^jG^2}_{4N-2j-1}^2
+\sum_{j=1}^{2N}\abs{\partial_t^jG^3}_{4N-2j+1/2}^2
\nonumber\\&\quad
+\sum_{j=1}^{2N-1}\abs{\partial_t^jG^4}_{4N-2j-1/2}^2
\lesssim\mathcal{E}_{2N}\mathfrak{D}_{2N}.
\end{align}
\end{lem}

\begin{proof}
The proof of the estimate \eqref{ne3} is straightforward since there are  no terms whose derivative count involved exceed those of
$\mathcal{E}_{2N}$. There is also no problem for the estimate \eqref{ne5} even $\eta$, $\bar{\eta}$, $q$ and $D^{4N}\eta$ are not included in $\mathfrak{D}_{2N}$,
since we can always
control them by $\mathcal{E}_{2N}$ and other terms multiplying them by $\mathfrak{D}_{2N}$.
\end{proof}

\section{Tangential energy evolution}
In this section, we will derive the tangential energy evolution estimates for the solution to \eqref{neweq}.
Recall the conventions in the time differential inequalities \eqref{timeq1} and \eqref{timeq2}.
\subsection{Basic $L^2$ energy estimate}
We start with the estimate of the solution itself.
\begin{lem}\label{lemm}
It holds that
\begin{align}\label{L2e}
&\frac{d}{dt}
\left(\norm{q}_{0}^{2}
+\norm{u}_{0}^{2}
+\abs{\eta}_{0}^{2}\right)
+\norm{ u}_{1}^{2}
\lesssim\sqrt{\mathcal{E}_{2N}}\mathfrak{D}_{2N}+\mathcal{E}_{2N}^{7/8}\mathcal{D}_{N+2,1}^{5/8}.
\end{align}
\end{lem}
\begin{proof}
Taking the $L^2(\Omega)$ inner product of the second equation in \eqref{neweq} with $Ju$ yields
\begin{align}\label{ps1}
\int_{\Omega}J\rho  \left(\partial_t^\mathcal{A}u +u\cdot \nabla _\mathcal{A} u\right)\cdot u
+\int_{\Omega}J\rho  \nabla _\mathcal{A} q\cdot u-\int_{\Omega}J\diverge_\mathcal{A}\mathbb{S}_\mathcal{A}u\cdot u=0.
\end{align}
By the integration by parts and using \eqref{deq} and the third and fifth equations in \eqref{neweq}, we find
\begin{align}\label{jp1}
&\int_{\Omega}J\rho   \left(\partial_t^\mathcal{A}u +u\cdot \nabla _\mathcal{A} u\right)\cdot u
=
\frac{1}{2}\int_{\Omega}J\rho   \left(\partial_t^\mathcal{A}  +u\cdot \nabla _\mathcal{A}  \right) \abs{u}^2
\nonumber\\&\quad
=
\frac{1}{2}\frac{d}{dt}
\int_{\Omega} J \rho \abs{u}^2
-\frac{1}{2}\int_{\Omega}J\left(\partial_t^\mathcal{A}\rho+\diverge_{\mathcal{A}}(\rho u)\right)\abs{u}^2
-\frac{1}{2}\int_{\Sigma} \rho\left(\partial_t\eta-u\cdot \mathcal{N}\right)\abs{u}^2
\nonumber\\&\quad=
\frac{1}{2}\frac{d}{dt}
\int_{\Omega}J\rho \abs{u}^2.
\end{align}
By integrating by parts and  using the  fifth equation in \eqref{neweq}, we obtain
\begin{align}\label{jp2}
\int_{\Omega}J\rho \nabla _\mathcal{A} q\cdot u
&
=\int_{\Sigma}\rho q\mathcal{N}\cdot u
-\int_{\Omega}J q \diverge_\mathcal{A}(\rho u),
\end{align}
and using the first equation in \eqref{neweq},
\begin{align}\label{ps5}
&-\int_{\Omega}J q \diverge_\mathcal{A}(\rho u)
=\int_{\Omega}\frac{J}{h'(\rho)}q\partial_t^\mathcal{A} q
=\frac{1}{2}\int_{\Omega}\frac{J}{h'(\rho)}\partial_t^\mathcal{A} \abs{q}^2
\nonumber\\&\quad
=\frac{1}{2}\frac{d}{dt}
\int_{\Omega}\frac{J}{h'(\rho)} \abs{q}^2
-\frac{1}{2}\int_{\Omega}J\partial_t^\mathcal{A}\left(\frac{1}{h'(\rho)}\right) \abs{q}^2
-\frac{1}{2}\int_{\Sigma}\frac{1}{ h'(\rho)} \abs{q}^2\partial_t \eta.
\end{align}
By integrating by parts  and using the fifth equation in \eqref{neweq},
\begin{align}\label{wxm}
-\int_{\Omega}J\diverge_\mathcal{A}\mathbb{S}_\mathcal{A}u\cdot u
=-\int_{\Sigma}\mathbb{S}_\mathcal{A}u\mathcal{N}\cdot u
+\int_{\Omega} J\mathbb{S}_\mathcal{A}u:\nabla_\mathcal{A}u,
\end{align}
and noting that
\begin{align}
\mathbb{S}_\mathcal{A}u:\nabla_\mathcal{A}u
=\frac{\mu}{2}\abs{\mathbb{D}_\mathcal{A}u}^2-\frac{2\mu}{d}\abs{\diverge_\mathcal{A}u}^2
+\mu'\abs{\diverge_\mathcal{A}u}^2
=\frac{\mu}{2}\abs{\mathbb{D}_\mathcal{A}^0u}^2
+\mu'\abs{\diverge_\mathcal{A}u}^2,
\end{align}
we get
\begin{align}\label{jp3}
-\int_{\Omega}J\diverge_\mathcal{A}\mathbb{S}_\mathcal{A}u\cdot u
=-\int_{\Sigma}\mathbb{S}_\mathcal{A}u\mathcal{N}\cdot u
+\int_{\Omega}J\left(\frac{\mu}{2}\abs{\mathbb{D}_\mathcal{A}^0u}^2
+\mu'\abs{\diverge_\mathcal{A}u}^2\right).
\end{align}
On the other hand, by using \eqref{ho} and the third  and fourth  equations of \eqref{neweq}, we have
\begin{align}\label{jp4}
&\int_{\Sigma}\rho  q \mathcal{N}\cdot u-\int_{\Sigma}\mathbb{S}_\mathcal{A}u\mathcal{N}\cdot u
\nonumber\\&\quad=\int_{\Sigma}(\rho^\star q I- \mathbb{S}_\mathcal{A}u)\mathcal{N}\cdot u
+ \left(\frac{1}{h'(\rho^\star)}(q-g\eta)+\mathcal{R}_{h^{-1}}\right)q\mathcal{N}\cdot u
\nonumber\\&\quad= \int_{\Sigma}\left(\rho^\star g \eta  -\mathcal{R}_{P\circ h^{-1}}+\left(\frac{1}{h'(\rho^\star)}(q-g\eta)+\mathcal{R}_{h^{-1}}\right)q\right)\partial_t\eta
\nonumber\\&\quad
=\hal\frac{d}{dt}\int_{\Sigma}\rho^\star g|\eta |^2
-  \int_{\Sigma}\left(\mathcal{R}_{P\circ h^{-1}}-\left(\frac{1}{h'(\rho^\star)}(q-g\eta)+\mathcal{R}_{h^{-1}}\right)q
 \right)\partial_t\eta.
\end{align}
Hence, by \eqref{jp1}, \eqref{ps5}, \eqref{jp3} and \eqref{jp4}, we deduce from \eqref{ps1} that
\begin{align}\label{youd}
&\frac{1}{2}\frac{d}{dt}\left(
\int_{\Omega}J\(\frac{1}{h'(\rho)} \abs{q}^2
+ \rho  \abs{u}^2\)
+\int_{\Sigma}\rho^\star g| \eta|^2\right)
+\int_{\Omega}J\(\frac{\mu}{2}\abs{\mathbb{D}_\mathcal{A}^0u}^2
+\mu' \abs{\diverge_\mathcal{A}u}^2\)
\nonumber\\&\quad
=
\frac{1}{2}\int_{\Omega}J\partial_t^\mathcal{A}\left(\frac{1}{h'(\rho)}\right) |q|^2
+ \int_{\Sigma}\left(\frac{1}{2 h'(\rho)} |q|^2+\mathcal{R}_{P\circ h^{-1}}-\left(\frac{1}{h'(\rho^\star)}(q-g\eta)+\mathcal{R}_{h^{-1}}\right)q
 \right)\partial_t\eta,
\end{align}
 where, by \eqref{ei} and the definitions \eqref{Rh} of $\mathcal{R}_{h^{-1}}$ and \eqref{Rp} of $\mathcal{R}_{P\circ h^{-1}}$,
 the right hand side can be bounded by
\begin{align}\label{mi2}
&\(\norm{\partial_t\rho}_{0}+\norm{\partial_t\bar{\eta}}_{0}\)\norm{q}_{0}\norm{q}_{L^\infty(\Omega)}
+\(\norm{q}_{L^\infty(\Sigma)}+\norm{\eta}_{L^\infty(\Sigma)}\)\(\abs{q}_{0}+\abs{\eta}_0\)\abs{\partial_t\eta}_{0}
\nonumber\\&\quad
\lesssim
\sqrt{\mathcal{D}_{N+2,1} }\sqrt{ \mathcal{E}_{2N} }\sqrt{\mathcal{D}_{N+2,1}^{1/4}\mathcal{E}_{2N}^{3/4}}
+\sqrt{\mathcal{D}_{N+2,1}^{1/4}\mathcal{E}_{2N}^{3/4}}\sqrt{ \mathcal{E}_{2N} }\sqrt{\mathcal{D}_{N+2,1} }
=\mathcal{E}_{2N}^{7/8}\mathcal{D}_{N+2,1}^{5/8}.
\end{align}
On the other hand,
\begin{align}\label{fzd}
&\int_{\Omega}J\left(\frac{\mu}{2}\abs{\mathbb{D}_\mathcal{A}^0u}^2
+\mu'\abs{\diverge_\mathcal{A}u}^2\right)
\nonumber\\&\quad
\geq \hal
\int_{\Omega}\(\frac{\mu}{2}\abs{\mathbb{D}^0u}^2
+\mu'\abs{\diverge u}^2\)
-C\int_{\Omega}\(\abs{\mathbb{D}_\mathcal{A}^0u-\mathbb{D}^0u}^2
+\abs{\diverge_\mathcal{A}u-\diverge u}^2\)
\nonumber\\&\quad
\geq \hal
\int_{\Omega}\frac{\mu}{2}\abs{\mathbb{D}^0u}^2
+\mu'\abs{\diverge u}^2
-C\sqrt{\mathcal{E}_{2N}}\mathfrak{D}_{2N}.
\end{align}

Consequently, by \eqref{mi2} and \eqref{fzd},
we deduce from \eqref{youd} that
\begin{align}\label{L22e}
&
\frac{d}{dt}
\left(
\int_{\Omega}J\(\frac{1}{h'(\rho)} \abs{q}^2
+ \rho   \abs{u}^2\)
+\int_{\Sigma}\rho^\star g\abs{ \eta}^2\right)+\int_{\Omega}\(\frac{\mu}{2}\abs{\mathbb{D}^0u}^2
+\mu' \abs{\diverge u}^2\)
\nonumber\\&\quad
\lesssim\sqrt{\mathcal{E}_{2N}}\mathfrak{D}_{2N}+\mathcal{E}_{2N}^{7/8}\mathcal{D}_{N+2,1}^{5/8}.
\end{align}
Hence, since $u=0$ on $\Sigma_b$, by using \eqref{viscosity} and Korn's inequality of Lemma \ref{xm5}, we conclude \eqref{L2e}.
\end{proof}

\subsection{Energy evolution in perturbed linear form}
It is more convenient to derive the energy evolution of certain horizontal space-time derivatives of the solution
by utilizing the perturbed linear form \eqref{q2}.

We first record the estimates  at the $2N$ level. Recall the norm notation \eqref{def3}.
\begin{prop}\label{prop1}
It holds that
\begin{align}\label{tan1}
&\frac{d}{dt}
\sum_{j=1}^{2N-1}\left(\norm{\partial_t^jq}_{0,4N-2j}^{2}
+\norm{\partial_t^j u}_{0,4N-2j}^{2}
+\abs{\partial_t^j\eta}_{4N-2j}^{2}\right)
+\sum_{j=1}^{2N-1}\norm{\partial_t^j u}_{1,4N-2j}^{2}
\nonumber\\
&\quad
\lesssim\sqrt{\mathcal{E}_{2N}}\mathfrak{D}_{2N}.
\end{align}
\end{prop}

\begin{proof}
Let $\alpha \in \mathbb{N}^{1+(d-1)}$ be so that $ |\alpha| \leq 4N$ and $1\leq \alpha_0 \leq 2N-1$.
Applying $\partial^\alpha$ to the second equation of \eqref{q2} and then taking the $L^2(\Omega)$ inner product of the resulting with
$\partial^{\alpha}u$, similarly as the derivation of  \eqref{youd}, we obtain
\begin{align}\label{e1}
&\frac{1}{2}\frac{d}{dt}\left(\int_{\Omega}\left(\frac{1}{h'(\bar{\rho})}\abs{\partial^\alpha q}^2
+\bar{\rho}\abs{\partial^\alpha u}^2\right)
+\int_{\Sigma}\rho^\star g\abs{\partial^\alpha \eta}^2\right)
+\int_{\Omega}\left(\frac{\mu}{2}\abs{\mathbb{D}^0\partial^\alpha u}^2
+\mu'\abs{\diverge\partial^\alpha u}^2\right)
\nonumber\\
&\quad =\int_{\Omega}\frac{1}{h'(\bar{\rho})}\partial^\alpha q\partial^\alpha G^1+\int_{\Omega}\partial^\alpha u\cdot\partial^\alpha G^2
+\int_{\Sigma}\rho^\star g\partial^\alpha \eta \partial^\alpha G^3-\int_{\Sigma}\partial^\alpha u\cdot\partial^\alpha G^4.
\end{align}

We now estimate the right hand side of \eqref{e1}. For the $G^1$ and $G^3$ terms, by \eqref{ne5}, we directly have
\begin{align}\label{e4}
\int_{\Omega}\frac{1}{h'(\bar{\rho})}\partial^\alpha q\partial^\alpha G^1
+\int_{\Sigma}\rho^\star g\partial^\alpha \eta \partial^\alpha G^3
&\lesssim
\norm{\partial^\alpha q}_0\norm{\partial^\alpha G^1}_0
+\abs{\partial^\alpha \eta}_0\abs{\partial^\alpha G^3}_0
\nonumber\\
&
\lesssim  \sqrt{\mathfrak{D}_{2N}}\sqrt{\mathcal{E}_{2N}\mathfrak{D}_{2N}}.
\end{align}
For $G^2$ and $G^4$ terms, we must split into two cases: $|\alpha|\leq4N-1$ and $|\alpha|=4N$.
For the former case, by \eqref{ne5}, we get
\begin{align}\label{sur1}
\int_{\Omega}\partial^\alpha u\cdot\partial^\alpha G^2
\lesssim
\norm{\partial^\alpha u}_0\norm{\partial^\alpha G^2}_0
\lesssim\sqrt{\mathfrak{D}_{2N}}\sqrt{\mathcal{E}_{2N}\mathfrak{D}_{2N}},
\end{align}
and by using additionally the trace theory,
\begin{align}
-\int_{\Sigma}\partial^\alpha u\cdot\partial^\alpha G^4
\lesssim
\abs{\partial^\alpha u}_0\abs{\partial^\alpha G^4}_0
\lesssim
\norm{\partial^\alpha u}_1\abs{\partial^\alpha G^4}_0
\lesssim\sqrt{\mathfrak{D}_{2N}}\sqrt{\mathcal{E}_{2N}\mathfrak{D}_{2N}}.
\end{align}
For the latter case, since $\alpha_0\leq 2N-1$, $|\alpha_h|\geq 2$ and
we may write $\alpha=\beta+(\alpha-\beta)$ for some $\beta$ with $|\beta|=1$, which implies $|\alpha-\beta|\leq4N-1$.
Integrating by parts, by \eqref{ne5}, we then find
\begin{align}\label{e6}
\int_{\Omega}\partial^\alpha u\cdot\partial^\alpha G^2
&= -\int_{\Omega}\partial^{\alpha+\beta }u\cdot\partial^{\alpha-\beta} G^2
\lesssim\norm{\partial^{\alpha+\beta }u}_0\norm{\partial^{\alpha-\beta} G^2}_0
\lesssim\sqrt{\mathfrak{D}_{2N}}\sqrt{\mathcal{E}_{2N}\mathfrak{D}_{2N}},
\end{align}
and by the trace theory,
\begin{align}\label{sur2}
-\int_{\Sigma}\partial^\alpha u\cdot\partial^\alpha G^4
=
\int_{\Sigma}\partial^{\alpha+\beta }u\cdot\partial^{\alpha-\beta} G^4
&\lesssim
\abs{\partial^{\alpha+\beta }u}_{-1/2}\abs{\partial^{\alpha-\beta }G^4}_{1/2}
\nonumber\\
&
\lesssim
\norm{\partial^{\alpha}u}_{1}\abs{\partial^{\alpha-\beta }G^4}_{1/2}
\lesssim
\sqrt{\mathfrak{D}_{2N}}\sqrt{\mathcal{E}_{2N}\mathfrak{D}_{2N}}.
\end{align}

Consequently, by \eqref{e4}--\eqref{sur2}, summing \eqref{e1} over such $\alpha$ and  using \eqref{viscosity} and Korn's inequality of Lemma \ref{xm5},
since $\partial^\alpha u=0$ on $\Sigma_b$,
we then conclude \eqref{tan1}.
\end{proof}

We then record the estimates at the $N+2$ level.
\begin{prop}\label{prop2}
It holds that
\begin{align}\label{tan2}
&\frac{d}{dt}
\Bigg(
\norm{Dq}_{0,2(N+2)-1}^{2}
+\norm{Du}_{0,2(N+2)-1}^{2}
+\abs{D\eta}_{2(N+2)-1}^{2}
\Bigg.
\nonumber\\&\quad
\Bigg.
+\sum_{j=1}^{N+1}\left(\norm{\partial_t^jq}_{0,2(N+2)-2j}^{2}
+\norm{\partial_t^j u}_{0,2(N+2)-2j}^{2}
+\abs{\partial_t^j\eta}_{2(N+2)-2j}^{2}\right)
\Bigg)
\nonumber\\&\quad
+\norm{Du}_{1,2(N+2)-1}^{2}+\sum_{j=1}^{N+1}\norm{\partial_t^j u}_{1,2(N+2)-2j}^{2}
\lesssim \sqrt{\mathcal{E}_{2N}}\mathcal{D}_{N+2,1}.
\end{align}
\end{prop}

\begin{proof}
Let $\alpha \in \mathbb{N}^{1+(d-1)}$ be so that $1\leq |\alpha| \leq 2(N+2)$ and $\alpha_0\leq N+1$. Note that \eqref{e1} still holds,
and we will bound the right hand side of \eqref{e1} at this time by $\sqrt{\mathcal{E}_{2N}}\mathcal{D}_{N+2,1}$.

First, for $ 2\leq|\alpha|\leq 2(N+2)$, the estimates follow similarly as in Proposition \ref{prop1} by using \eqref{ne2} in place of \eqref{ne5},
except the following terms:
\begin{align}\label{s1}
\int_{\Omega}\frac{1}{h'(\bar{\rho})}\partial^\alpha q\partial^\alpha G^{1,1}
\text{ and }
\int_{\Sigma}\rho^\star g\partial^{\alpha}\eta\partial^{\alpha}G^{3}      \text{ when } |\alpha|=|\alpha_h|=2(N+2).
\end{align}
For the $G^{1,1}$ term, we write
\begin{align}\label{s3}
\partial^\alpha G^{1,1}
=J^{-1}\partial_t\varphi\partial_d \partial^\alpha q-u_i\mathcal{A}_{ij}\partial_j\partial^\alpha q
+\left[\partial^\alpha,J^{-1}\partial_t\varphi\partial_d\right] q+\left[\partial^\alpha,u_i\mathcal{A}_{ij}\partial_j\right]q.
\end{align}
Similarly as \eqref{ne2},
\begin{align}\label{s4}
&\int_{\Omega}\frac{1}{h'(\bar{\rho})}\partial^\alpha q
\big(\left[\partial^\alpha,J^{-1}\partial_t\varphi\partial_d\right] q+\[\partial^\alpha,u_i\mathcal{A}_{ij}\partial_j\]q\big)
\nonumber\\
&\quad\lesssim \norm{\partial^\alpha q}_0
\big(\norm{\[\partial^\alpha,J^{-1}\partial_t\varphi\partial_d \]q}_0
+\norm{\[\partial^\alpha,u_i\mathcal{A}_{ij}\partial_j\]q}_0\big)
\lesssim \sqrt{\mathcal{D}_{N+2,1}}\sqrt{\mathcal{E}_{2N}\mathcal{D}_{N+2,1}},
\end{align}
and by integrating by parts and using the third and fifth equations in \eqref{neweq},
\begin{align}\label{s5}
&\int_{\Omega}\frac{1}{h'(\bar{\rho})}\partial^\alpha q
\left(J^{-1}\partial_t\varphi\partial_d \partial^\alpha q-u_i\mathcal{A}_{ij}\partial_j\partial^\alpha q\right)
=\frac{1}{2}\int_{\Omega}\frac{1}{h'(\bar{\rho})}
\left(J^{-1}\partial_t\varphi\partial_d-u_i\mathcal{A}_{ij}\partial_j\right)\abs{\partial^\alpha q}^2
\nonumber\\
&\quad=
-\frac{1}{2}\int_{\Omega}\left(\partial_d\left(\frac{1}{h'(\bar{\rho})}J^{-1}\partial_t\varphi\right)
-\partial_j\left(\frac{1}{h'(\bar{\rho})}u_i\mathcal{A}_{ij}\right)\right)\abs{\partial^\alpha q}^2
\lesssim \sqrt{\mathcal{E}_{2N}}\mathcal{D}_{N+2,1},
\end{align}
hence, we have
\begin{align}\label{s6}
\int_{\Omega}\frac{1}{h'(\bar{\rho})}\partial^\alpha q\partial^\alpha G^{1,1}
\lesssim \sqrt{\mathcal{E}_{2N}}\mathcal{D}_{N+2,1}.
\end{align}
For the $G^3$ term, by \eqref{Pe2},
\eqref{sob1}, \eqref{sobbb}, the $L^\infty$ table in \eqref{ei} in terms of $\mathcal{D}_{N+2,1}$ and the trace theory, for $N\geq4$, we have
\begin{align}\label{s7}
&\int_{\Sigma}\rho_1g\partial^{\alpha}\eta\partial^{\alpha}G^{4}
\lesssim
 \abs{D^{2(N+2)}\eta}_0 \abs{D^{2(N+2)}G^4}_0
 \nonumber\\
& \quad
\lesssim \abs{D^{2(N+2)}\eta}_0\left(\abs{D^{2(N+2)+1}\eta}_0\norm{u_h}_{L^\infty(\Sigma)}
+\abs{D^{2(N+2)}u_h}_0\norm{D\eta}_{L^\infty(\Sigma)}\right)
\nonumber\\
& \quad
\lesssim\sqrt{\mathcal{D}_{N+2,1}^{(4N-9)/(4N-7)}\mathcal{E}_{2N}^{2/(4N-7)}}
\sqrt{\mathcal{D}_{N+2,1}^{3/4}\mathcal{E}_{2N}^{1/4}}
\times
\sqrt{\mathcal{D}_{N+2,1}^{(4N-11)/(4N-7)}\mathcal{E}_{2N}^{4/(4N-7)}+\mathcal{D}_{N+2,1}}
\nonumber\\
& \quad
\lesssim \sqrt{\mathcal{E}_{2N}}\mathcal{D}_{N+2,1}.
\end{align}
Hence, the terms in \eqref{s1} are bounded by $\sqrt{\mathcal{E}_{2N}}\mathcal{D}_{N+2,1}$.

Now, we estimate for the remaining case $|\alpha|=|\alpha_h|=1$.
It follows directly from \eqref{Ge} that
\begin{align}\label{sss}
&\int_{\Omega}D u\cdot D G^2
\lesssim \norm{D u}_0\norm{D G^2}_0
\lesssim\sqrt{\mathcal{D}_{N+2,1}}\sqrt{\mathcal{E}_{2N}\mathcal{D}_{N+2,1}} .
\end{align}
For the $G^3$ term, by the $L^2$ and $L^\infty$ tables in \eqref{ei} in terms of $\mathcal{D}_{N+2,1}$ and the trace theory,
\begin{align}
\int_{\Sigma}\rho ^\star gD\eta DG^{3}
&=\int_{\Sigma}\rho ^\star gD\eta\(Du_h\cdot D\eta+u_h\cdot D(D\eta)\)
\nonumber\\&
\lesssim
\abs{D\eta}_0\left(\abs{Du_h}_0\norm{D\eta}_{L^\infty(\Sigma)}+\norm{u_h}_{L^\infty(\Sigma)}\abs{D^2\eta}_0\right)
\lesssim
 \sqrt{\mathcal{E}_{2N}}\mathcal{D}_{N+2,1} .
\end{align}
Note that the value $\theta$ in the interpolation estimate of $Dq$ in the $L^2$ table of \eqref{ei}
in terms of $\mathcal{D}_{N+2,1}$ is $1/2$, and the one of $DG^{1}$ can not be greater than $3/2$.
To get around this, our idea is to integrate by parts to have, by \eqref{Ge},
\begin{align}\label{s8}
&\int_{\Omega}\frac{1}{h'(\bar{\rho})}Dq DG^{1}
=-\int_{\Omega}\frac{1}{h'(\bar{\rho})}D^2q G^{1}
\lesssim \norm{D^2q}_0\norm{G^1}_0
\lesssim\sqrt{\mathcal{D}_{N+2,1}}\sqrt{\mathcal{E}_{2N}\mathcal{D}_{N+2,1}}.
\end{align}
Now for the $G^4$ term, by \eqref{Ge} and the trace theory, we directly have
\begin{align}
-\int_{\Sigma}D u_h\cdot D G_h^4
&\lesssim
\abs{D u_h }_0\abs{D G_h^4}_0
\lesssim
\norm{D u_h }_1\abs{D G_h^4}_0
\nonumber\\&
\lesssim
\sqrt{\mathcal{D}_{N+2,1}}\sqrt{\mathcal{E}_{2N}\mathcal{D}_{N+2,1}}.
\end{align}
Since the value $\theta$ in the interpolation estimate of $G_d^4$ in the $L^2$ table of \eqref{Ge} in terms of $\mathcal{D}_{N+2,1}$ is not 1,
we again integrate by parts and use \eqref{Ge}
and the trace theory to obtain
\begin{align}\label{s9}
-\int_{\Sigma}Du_d D G_d^4
=
\int_{\Sigma}u_d D^2 G_d^4
&\lesssim
\abs{u_d}_0\abs{D^2 G_d^4}_0
\lesssim
\norm{u_d}_1\abs{D^2 G_d^4}_0
\nonumber\\&
\lesssim
\sqrt{\mathcal{D}_{N+2,1}}\sqrt{\mathcal{E}_{2N}\mathcal{D}_{N+2,1}}.
\end{align}
Therefore, we have
\begin{align}
-\int_{\Sigma}D u\cdot D G^4
\lesssim
\sqrt{\mathcal{D}_{N+2,1}}\sqrt{\mathcal{E}_{2N}\mathcal{D}_{N+2,1}}.
\end{align}
Hence, in this case the right hand side of \eqref{e1} is also bounded by $\sqrt{\mathcal{E}_{2N}}\mathcal{D}_{N+2,1}$.
\end{proof}


\subsection{Energy evolution of highest order temporal derivatives}
The energy evolution estimates of the highest order temporal derivatives can not be obtained by using the linear form \eqref{q2};
indeed, if we would use \eqref{e1} with $\alpha_{0}=2N$ or $\alpha_{0}=N+2$,
then the term $\partial_t^{\alpha_{0}}\nabla^2 u$ of $G^{2}$ and the term $\partial_t^{\alpha_{0}}\nabla u$ of $G^{3}$ in  the right hand side of \eqref{e1} are out of  control. Motivated by \cite{GT_inf},
we will use the original form \eqref{neweq} and apply $\partial_t^{\ell}$, $\ell\geq 1$, to  find
\begin{equation}\label{q3}
\begin{cases}
\partial_t (\partial_t^{\ell}q)+h'(\bar{\rho})\diverge_\mathcal{A}(\bar{\rho}\partial_t^{\ell}u)=F^{1,\ell}
 &\text{in }\Omega
\\ \rho\partial_t(\partial_t^{\ell}u)
+\bar{\rho}\nabla_\mathcal{A}(\partial_t^{\ell}q)
-\diverge_\mathcal{A}\mathbb{S}_\mathcal{A}(\partial_t^{\ell}u)=F^{2,\ell}   &\text{in }\Omega
\\ \partial_t(\partial_t^{\ell}\eta)=\partial_t^{\ell}u\cdot \mathcal{N}+ F^{3,\ell}   &\text{on }\Sigma
\\ \left(\rho^\star\partial_t^{\ell}qI-\mathbb{S}_\mathcal{A}(\partial_t^{\ell}u)\right)\mathcal{N}
=\rho^\star g\partial_t^{\ell}\eta\mathcal{N}+F^{4,\ell}  &\text{on }\Sigma
\\ \partial_t^{\ell}u=0       &\text{on }\Sigma_b,
\end{cases}
\end{equation}
where
\begin{align}\label{F1a}
F^{1,\ell}=&\partial_t^{\ell}\left(J^{-1}\partial_t\varphi\partial_dq\right)
-\left[\partial_t^{\ell}, h'(\rho)\mathcal{A}_{lk}\right]\partial_k(\rho u_l)
-h'(\rho)\mathcal{A}_{lm}\partial_m\left(\left[\partial_t^{\ell},\rho\right]u_l\right)
\nonumber\\&
-\(h'(\rho)-h'(\bar{\rho})\)\mathcal{A}_{lm}\partial_m\left(\rho\partial_t^{\ell} u_l\right)
-h'(\bar{\rho})\mathcal{A}_{lm}\partial_m\left((\rho-\bar{\rho} )\partial_t^{\ell}u_l\right),
\end{align}
\begin{align}\label{F2a}
F_i^{2,\ell}=&
-\left[\partial_t^{\ell}, \rho\right]\partial_tu_i
-\partial_t^{\ell}\left(-\rho J^{-1}\partial_t\varphi\partial_du_i+\rho u_l\mathcal{A}_{lm}\partial_mu_i\right)
-(\rho-\bar{\rho})\mathcal{A}_{ik}\partial_k\partial_t^{\ell}q
\nonumber\\&
-\left[\partial_t^{\ell}, \rho\mathcal{A}_{ik}\right]\partial_k q
+\mu\left(\left[\partial_t^{\ell},\mathcal{A}_{lk} \right]\partial_k(\mathcal{A}_{lm}\partial_mu_i)
+\mathcal{A}_{lk} \partial_k\left(\left[\partial_t^{\ell},\mathcal{A}_{lm}\right]\partial_mu_i\right)\right)
\nonumber\\&
+\left(\left(\frac{d-2}{d}\mu+\mu'\right)\left[\partial_t^{\ell},\mathcal{A}_{ik} \right]\partial_k(\mathcal{A}_{lm}\partial_mu_l)
+\mathcal{A}_{ik} \partial_k\left(\left[\partial_t^{\ell},\mathcal{A}_{lm}\right]\partial_mu_l\right)\right),
\end{align}
\begin{align}\label{F4a}
F^{3,\ell}=\left[\partial_t^\ell, D\eta\cdot\right]u_h
\end{align}
and
\begin{align}\label{F3a}
F_i^{4,\ell}=&
-\left[\partial_t^\ell,\mathcal{N}_i\right](\rho ^\star q-\rho ^\star g\eta)
+\mu\left[\partial_t^\ell,\mathcal{N}_l\mathcal{A}_{ik}\right]\partial_ku_l
+\mu\left[\partial_t^\ell,\mathcal{N}_l\mathcal{A}_{lk}\right]\partial_ku_i
\nonumber\\&
+\left(\mu'-\frac{2\mu}{d}\right)\left[\partial_t^\ell,\mathcal{N}_i\mathcal{A}_{lk}\right]\partial_ku_l
-\partial_t^\ell\left(\mathcal{R}_{P\circ h^{-1}}\mathcal{N}_i\right).
\end{align}

The estimates of these nonlinearities when $\ell=2N$ and $N+2$ are presented in the following.
\begin{lem}\label{t1}
 It holds that
\begin{align}\label{t2}
\norm{F^{1,2N}}_{0}^{2}
+\norm{F^{2,2N}}_{0}^{2}+\abs{F^{3,2N}}_{0}^{2}
+\abs{F^{4,2N}}_{0}^{2}
\lesssim \mathcal{E}_{2N}\mathfrak{D}_{2N}
\end{align}
and
\begin{align}\label{t3}
\norm{F^{1,N+2}}_{0}^{2}
+\norm{F^{2,N+2}}_{0}^{2}+\abs{F^{3,N+2}}_{0}^{2}
+\abs{F^{4,N+2}}_{0}^{2}
\lesssim \mathcal{E}_{2N}\mathcal{D}_{N+2,1}.
\end{align}
\end{lem}
\begin{proof}
The estimates \eqref{t2} and \eqref{t3} are straightforward; we control each term involving the highest order temporal derivative count
by $\mathfrak{D}_{2N}$ and $\mathcal{D}_{N+2,1}$, respectively, and other terms multiplying it by $\mathcal{E}_{2N}$.
\end{proof}

We first record the evolution estimate for $2N$ temporal derivatives.
\begin{prop}\label{prop3}
It holds that
\begin{align}\label{tan3}
&\frac{d}{dt}
\left(
\norm{\partial_t^{2N}q}_0^{2}
+\norm{\partial_t^{2N} u}_0^{2}
+\abs{\partial_t^{2N}\eta}_0^2
\right)
+\norm{\partial_t^{2N}u}_1^2
\lesssim \sqrt{\mathcal{E}_{2N}}\mathfrak{D}_{2N}.
\end{align}
\end{prop}

\begin{proof}
Taking the $L^2(\Omega)$ inner product of the second equation in \eqref{q3} for $\ell=2N$ with $J\partial_t^{2N}u$,
similarly as the derivation of \eqref{youd}, we obtain
\begin{align}\label{t4}
&\frac{1}{2}\frac{d}{dt}\left(\int_{\Omega}J\left(\frac{1}{h'(\bar{\rho})}\abs{\partial_t^{2N}q}_{0}^{2}
+\rho \abs{\partial_t^{2N}u}^2\right)
+\int_{\Sigma}\rho ^\star g\abs{\partial_t^{2N}\eta}^2\right)
\nonumber\\
&\quad+\int_{\Omega}J\left(\frac{\mu}{2}\abs{\mathbb{D}_\mathcal{A}^0\partial_t^{2N}u}^2
+\mu'\abs{\diverge_\mathcal{A}\partial_t^{2N}u}^2\right)
\nonumber\\
&\quad=\frac{1}{2}\int_{\Omega}\left(\frac{1}{h'(\bar{\rho})}\partial_tJ\abs{\partial_t^{2N} q}_{0}^{2}
+\partial_t(J\rho )\abs{\partial_t^{2N}u}^2\right)
+\int_{\Omega}J\frac{1}{h'(\bar{\rho})}\partial_t^{2N} qF^{1,2N}
\nonumber\\
&\quad\quad
+\int_{\Omega}J\partial_t^{2N}u\cdot F^{2,2N}
+\int_{\Sigma}\rho ^\star g\partial_t^{2N}\eta F^{3,2N}
-\int_{\Sigma}\partial_t^{2N}u\cdot F^{4,2N}.
\end{align}

Now we estimate the right hand side of \eqref{t4}. It follows directly that
\begin{align}\label{tt1}
&\frac{1}{2}\int_{\Omega}\frac{1}{h'(\bar{\rho})}\partial_tJ\abs{\partial_t^{2N} q}_{0}^{2}
+\partial_t(J\rho )\abs{\partial_t^{2N}u}^2
\lesssim  \sqrt{\mathcal{E}_{2N}}
\left(\norm{\partial_t^{2N} q}_{0}^2+\norm{\partial_t^{2N} u}_{0}^2\right)
\lesssim \sqrt{\mathcal{E}_{2N}}\mathfrak{D}_{2N}.
\end{align}
By \eqref{t2}, we have
\begin{align}\label{t5}
&\int_{\Omega}J\frac{1}{h'(\bar{\rho})}\partial_t^{2N} qF^{1,2N}
+\int_{\Omega}J\partial_t^{2N}u\cdot F^{2,2N}+\int_{\Sigma}\rho^\star g\partial_t^{2N}\eta F^{3,2N}
\nonumber\\&\quad
\lesssim
\norm{\partial_t^{2N} q}_{0}\norm{F^{1,2N}}_{0}
+\norm{\partial_t^{2N} u}_{0}
\norm{ F^{2,2N}}_{0}+\abs{\partial_t^{2N} \eta}_{0}\abs{ F^{3,2N}}_{0}
\lesssim \sqrt{\mathfrak{D}}_{2N}\sqrt{\mathcal{E}_{2N}\mathfrak{D}_{2N}},
\end{align}
and by the trace theory,
\begin{align}
&
-\int_{\Sigma}\partial_t^{2N}u\cdot F^{4,2N}
\lesssim\abs{\partial_t^{2N} u}_{0}
\abs{F^{4,2N}}_{0}
\lesssim\norm{\partial_t^{2N} u}_{1}
\abs{F^{4,2N}}_{0}
\lesssim \sqrt{\mathfrak{D}_{2N}}\sqrt{\mathcal{E}_{2N}\mathfrak{D}_{2N}}.
\end{align}
Arguing similarly as \eqref{fzd}, we find
\begin{align}\label{fzd2}
&\int_{\Omega}J\left(\frac{\mu}{2}\abs{\mathbb{D}_\mathcal{A}^0\partial_t^{2N}u}^2
+\mu'\abs{\diverge_\mathcal{A}\partial_t^{2N}u}^2\right)
\nonumber\\&\quad
\geq \hal
\int_{\Omega}\left(\frac{\mu}{2}\abs{\mathbb{D}^0\partial_t^{2N}u}^2
+\mu'\abs{\diverge\partial_t^{2N} u}^2\right)
-C\sqrt{\mathcal{E}_{2N}}\mathfrak{D}_{2N}.
\end{align}

Consequently, by \eqref{tt1}--\eqref{fzd2}, since $\partial_t^{2N}u=0 $ on $\Sigma_b$,
we deduce \eqref{tan3} from \eqref{t4}
by using  \eqref{viscosity}  and Korn's inequality  of Lemma \ref{xm5}.
\end{proof}
We then record a similar result for   $N+2$ temporal derivatives.
\begin{prop}\label{prop4}
It holds that
\begin{align}\label{tan4}
&\frac{d}{dt}
\left(
\norm{\partial_t^{N+2}q}_0^{2}
+\norm{\partial_t^{N+2} u}_0^{2}
+\abs{\partial_t^{N+2}\eta}_0^2
\right)
+\norm{\partial_t^{N+2}u}_1^2
\lesssim \sqrt{\mathcal{E}_{2N}}\mathcal{D}_{N+2,1}.
\end{align}
\end{prop}

\begin{proof}
The proof  is similar to that of  Proposition \ref{prop3} by using \eqref{t3} in place of \eqref{t2}.
\end{proof}


\subsection{Energy evolution of $4N$ order horizontal spatial derivatives}
As explained in Section \ref{sj}, when estimating directly the $4N$ horizontal spatial derivatives of the solution
 would lead to
the appearance of $\norm{\nabla u}_{L^\infty(\Omega)}\mathcal{F}_{2N}$, which causes the difficulty as the incompressible case \cite{WYJ1}.
To overcome this, we will estimate instead $\partial_i^\mathcal{A}u$ and $\partial_i^\mathcal{A}q$,
a variant of good unknowns used in \cite{WYJ1}.

We shall now derive the equations satisfied by $\partial_i^\mathcal{A}u$ and $\partial_i^\mathcal{A}q$, $i=1,d-1,d$.
\begin{lem}\label{Aq}
It holds that for $i=1,d-1,d$,
\begin{equation}\label{Aq1}
\begin{cases}
\partial_t  \partial_i^\mathcal{A}q
+h'(\bar{\rho})\diverge\left(\bar{\rho} \partial_i^\mathcal{A}u\right)=G^{1,i,\sharp}
                                                       -\partial_iP'(\bar{\rho})\diverge u
                                                        &\text{in }\Omega
\\ \bar{\rho}\partial_t\partial_i^\mathcal{A}u+\bar{\rho}\nabla \partial_i^\mathcal{A}q
-\diverge\mathbb{S}\partial_i^\mathcal{A}u=G^{2,i,\sharp}
 -\displaystyle\frac{\partial_i\bar{\rho}}{\bar{\rho}}\diverge\mathbb{S} u    &\text{in }\Omega,
\end{cases}
\end{equation}
where $G^{1,i,\sharp}$ and $G^{2,i,\sharp}$ are defined by \eqref{w11}--\eqref{w12} and \eqref{w2}, respectively.
\end{lem}
\begin{proof}
Let $i=1,d-1,d$. First, applying $\partial_i^\mathcal{A}$ to the first equation of \eqref{neweq}, we find
\begin{align}\label{alb}
\partial_t^\mathcal{A}\partial_i^\mathcal{A}q
+\partial_i^\mathcal{A}h'(\rho)\diverge_\mathcal{A}(\rho u)
+h'(\rho)\diverge_\mathcal{A}\left(\rho \partial_i^\mathcal{A}u\right)
+h'(\rho)\diverge_\mathcal{A}\left(\partial_i^\mathcal{A}\rho u\right)
=0.
\end{align}
Recalling \eqref{tsg}, we  have
\begin{align}\label{kea}
h'(\rho)\diverge_\mathcal{A}\left(\partial_i^\mathcal{A}\rho u\right)
&=\diverge_\mathcal{A}\left(\partial_i^\mathcal{A}h(\rho)u\right)-\partial_i^\mathcal{A}\rho u\cdot\nabla_\mathcal{A}h'(\rho)
\nonumber\\&
=u\cdot \nabla_\mathcal{A}\partial_i^\mathcal{A}h(\rho)
+\partial_i^\mathcal{A}h(\rho)\diverge_\mathcal{A} u
-\partial_i^\mathcal{A}h'(\rho)u\cdot\nabla_\mathcal{A}\rho
\nonumber\\&
=u\cdot \nabla_\mathcal{A}\partial_i^\mathcal{A}q
+\left(\partial_i^\mathcal{A}h(\rho)+\partial_i^\mathcal{A}h'(\rho)\rho\right)\diverge_\mathcal{A} u
-\partial_i^\mathcal{A}h'(\rho)\diverge_\mathcal{A} (\rho u).
\end{align}
Note that, by \eqref{bh},
\begin{align}\label{lei}
\partial_i^\mathcal{A}h(\rho)+\partial_i^\mathcal{A}h'(\rho)\rho
=\partial_i^\mathcal{A}P'(\rho).
\end{align}
It then follows from \eqref{alb}--\eqref{lei} that
\begin{align}\label{xw2}
\partial_t^\mathcal{A}\partial_i^\mathcal{A}q
+h'(\rho)\diverge_\mathcal{A}(\rho \partial_i^\mathcal{A}u)
+u\cdot \nabla_\mathcal{A}\partial_i^\mathcal{A}q
+\partial_i^\mathcal{A}P'(\rho)\diverge_\mathcal{A} u=0,
\end{align}
which yields  the first equation in \eqref{Aq1} with $G^{1,i,\sharp}:=G^{1,1,i,\sharp}+G^{1,2,i,\sharp}$ for
\begin{align}
G^{1,1,i,\sharp}:=&J^{-1}\partial_t\varphi\partial_d\partial_i^\mathcal{A}q-u\cdot \nabla_\mathcal{A}\partial_i^\mathcal{A}q,\label{w11}\\
G^{1,2,i,\sharp}:=&
                  -\left(h'(\rho)-h'(\bar{\rho})\right)\diverge_\mathcal{A}\left(\rho \partial_i^\mathcal{A}u\right)
                  -h'(\bar{\rho})(\diverge_\mathcal{A}-\diverge)\left(\rho \partial_i^\mathcal{A}u\right)
\nonumber\\&
                  -h'(\bar{\rho})\diverge\left((\rho-\bar{\rho})\partial_i^\mathcal{A}u\right)
                  -(\mathcal{A}_{ij}-\delta_{ij})\partial_j P'(\rho)\diverge_\mathcal{A} u
\nonumber\\&
-(\partial_i P'(\rho)-\partial_i P'(\bar{\rho}))\diverge_\mathcal{A} u
-\partial_i P'(\bar{\rho})(\diverge_\mathcal{A}-\diverge)u
.\label{w12}
\end{align}

Next, applying $\partial_i^\mathcal{A}$ to the second equation of \eqref{neweq}, we find
\begin{align}\label{yox1}
&\rho\left(\partial_t^\mathcal{A} \partial_i^\mathcal{A}u+u \cdot \nabla_\mathcal{A}  \partial_i^\mathcal{A}u\right)
+\rho\partial_i^\mathcal{A}u\cdot \nabla_\mathcal{A} u
+\partial_i^\mathcal{A}\rho\left(\partial_t^\mathcal{A}u+u \cdot \nabla_\mathcal{A} u\right)
\nonumber\\&\quad
+\rho \nabla_\mathcal{A}\partial_i^\mathcal{A}q
+\partial_i^\mathcal{A}\rho \nabla_\mathcal{A} q
-\diverge_\mathcal{A}\mathbb{S}_\mathcal{A}\partial_i^\mathcal{A}u
=0.
\end{align}
By using the second equation of \eqref{neweq}, we then have
\begin{align}\label{yox4}
&\rho\left(\partial_t^\mathcal{A} \partial_i^\mathcal{A}u+u \cdot \nabla_\mathcal{A}  \partial_i^\mathcal{A}u\right)
+\rho \nabla_\mathcal{A}\partial_i^\mathcal{A}q
-\diverge_\mathcal{A}\mathbb{S}_\mathcal{A}\partial_i^\mathcal{A}u
=
-\rho\partial_i^\mathcal{A}u\cdot \nabla_\mathcal{A} u
-\frac{1}{\rho}\partial_i^\mathcal{A}\rho\diverge_\mathcal{A}\mathbb{S}_\mathcal{A}u.
\end{align}
This yields that the second equation in \eqref{Aq1} with $G^{2,i,\sharp}$ defined by
\begin{align}\label{w2}
G^{2,i,\sharp}:=&-(\rho-\bar{\rho})\partial_t \partial_i^\mathcal{A}u
-\rho\left(-J^{-1}\partial_t\varphi\partial_d \partial_i^\mathcal{A}u+u \cdot \nabla_\mathcal{A}  \partial_i^\mathcal{A}u\right)
-(\rho-\bar{\rho})\nabla_\mathcal{A}\partial_i^\mathcal{A}q
\nonumber\\&
-\bar{\rho}\left(\nabla_\mathcal{A}-\nabla\right)\partial_i^\mathcal{A}q
+\left(\diverge_\mathcal{A}-\diverge\right)\mathbb{S}_\mathcal{A}\partial_i^\mathcal{A}u
+\diverge\left(\mathbb{S}_\mathcal{A}\partial_i^\mathcal{A}u-\mathbb{S}\partial_i^\mathcal{A}u\right)
-\rho\partial_i^\mathcal{A}u\cdot \nabla_\mathcal{A} u
\nonumber\\&
-\left(\frac{1}{\rho}-\frac{1}{\bar{\rho}}\right)\partial_i^\mathcal{A}\rho\diverge_\mathcal{A}\mathbb{S}_\mathcal{A} u
-\frac{1}{\bar{\rho}}(\mathcal{A}_{ij}-\delta_{ij})\partial_j \rho\diverge_\mathcal{A}\mathbb{S}_\mathcal{A} u
-\frac{1}{\bar{\rho}}(\partial_i\rho-\partial_i\bar{\rho})\diverge_\mathcal{A}\mathbb{S}_\mathcal{A} u
\nonumber\\&
-\frac{1}{\bar{\rho}}\partial_d\bar{\rho}(\diverge_\mathcal{A}-\diverge)\mathbb{S}_\mathcal{A} u
-\frac{1}{\bar{\rho}}\partial_d\bar{\rho}\diverge(\mathbb{S}_\mathcal{A} -\mathbb{S})u.
\end{align}
The lemma is thus concluded.
\end{proof}

The following records the boundary conditions for $\partial_i^\mathcal{A}u$ and $\partial_i^\mathcal{A}q$ when $i=1,d-1$.
\begin{lem}\label{Bq}
It holds that for $i=1,d-1$,
\begin{equation}\label{Bq1}
\begin{cases}
 \partial_t(\partial_i\eta)= \partial_i^\mathcal{A}u\cdot e_d+G^{3,i,\sharp}    &\text{on }\Sigma
\\ \left(\rho^\star\partial_i^\mathcal{A}q I-\mathbb{S} \partial_i^\mathcal{A}u\right)e_d
=\rho^\star g\partial_i\eta e_d+G^{4,i,\sharp}  &\text{on }\Sigma
\\ \partial_i^\mathcal{A}u=0     &\text{on }\Sigma_b,
\end{cases}
\end{equation}
where $G^{3,i,\sharp}$ and $G^{4,i,\sharp}$ are defined by \eqref{w4} and \eqref{w3}, respectively.
\end{lem}
\begin{proof}
Let $i=1,d-1$. First, applying $\partial_i$ to the third equation of \eqref{neweq}, we obtain
\begin{align}\label{lili}
\partial_t\partial_i\eta-\partial_i^\mathcal{A}u\cdot \mathcal{N}
=\partial_d^\mathcal{A}u\cdot \mathcal{N}\partial_i\eta-u_h\cdot D\partial_i\eta  \text{ on } \Sigma,
\end{align}
which implies the first equation in \eqref{Bq1} with $G^{3,i,\sharp}$ defined by
\begin{align}\label{w4}
G^{3,i,\sharp}=-u_h\cdot D\partial_i\eta+G^{3,1,i,\sharp},
\end{align}
where
\begin{align}\label{w41}
G^{3,1,i,\sharp}:=\partial_i^\mathcal{A}u\cdot (\mathcal{N}-e_d)
+\partial_d^\mathcal{A}u\cdot \mathcal{N}\partial_i\eta.
\end{align}

Next, applying $\partial_i$ to the fourth equation of \eqref{neweq}, we get
\begin{align}
&\left(\left(\rho^\star\partial_i q-\rho^\star g \partial_i\eta+\partial_i\mathcal{R}_{P\circ h^{-1}}\right)I
-\partial_i\mathbb{S}_\mathcal{A}u\right)\mathcal{N}
\nonumber\\&\quad
+\left(\left(\rho^\star q-\rho^\star g \eta+\mathcal{R}_{P\circ h^{-1}}\right)I
-\mathbb{S}_\mathcal{A}u\right)\partial_i\mathcal{N}=0 \text{ on } \Sigma.
\end{align}
Note that
\begin{align}
\partial_i q
=
\partial_i^\mathcal{A} q+\partial_d^\mathcal{A}q\partial_i\eta
\text{ and }
\partial_i\mathbb{S}_\mathcal{A}u
=
\mathbb{S}_\mathcal{A}\partial_i^\mathcal{A}u+\partial_d^\mathcal{A}\mathbb{S}_\mathcal{A} u\partial_i\eta,
\end{align}
and from the fourth equation in \eqref{neweq},
\begin{align}
\rho^\star q-\rho^\star g \eta+\mathcal{R}_{P\circ h^{-1}}
=
\frac{\mathbb{S}_\mathcal{A}u\mathcal{N}\cdot\mathcal{N}}{|\mathcal{N}|^2} \text{ on } \Sigma,
\end{align}
we  then obtain
\begin{align}\label{su}
&
\left(\left(\rho^\star\partial_i^\mathcal{A}q+\rho^\star\partial_d^\mathcal{A}q\partial_i\eta-\rho^\star g\partial_i\eta+\partial_i\mathcal{R}_{P\circ h^{-1}}\right)I  -\mathbb{S}_\mathcal{A}\partial_i^\mathcal{A}u-\partial_i\eta\partial_d^\mathcal{A}\mathbb{S}_\mathcal{A} u\mathcal{N}\right)\mathcal{N}
\nonumber\\&\quad
-\left(\mathbb{S}_\mathcal{A} u-\frac{\mathbb{S}_\mathcal{A}u\mathcal{N}\cdot\mathcal{N}}{|\mathcal{N}|^2}I\right)\partial_i\mathcal{N}
=0   \text{ on } \Sigma,
\end{align}
which implies the second equation in \eqref{Bq1} with $G^{4,i,\sharp}$ defined by
\begin{align}\label{w3}
G^{4,i,\sharp}:
=&
-\left(\left(\rho^\star \partial_i^\mathcal{A}q-\rho^\star g\partial_i\eta\right)I  -\mathbb{S}\partial_i^\mathcal{A}u\right)(\mathcal{N}-e_d)
\nonumber\\&
-\left(\rho^\star\partial_d^\mathcal{A}q\partial_i\eta+\partial_i\mathcal{R}_{P\circ h^{-1}}
-\partial_i\eta\mathbb{S}_\mathcal{A}\partial_d^\mathcal{A}u\right)\mathcal{N}
\nonumber\\&
+\left(\mathbb{S}_\mathcal{A} u-\frac{\mathbb{S}_\mathcal{A}u\mathcal{N}\cdot\mathcal{N}}{|\mathcal{N}|^2}I\right)\partial_i\mathcal{N}
-\left(\mathbb{S}_\mathcal{A}-\mathbb{S}\right)\partial_i^\mathcal{A}u\mathcal{N}.
\end{align}

Finally, we have
\begin{align}
\partial_i^\mathcal{A}u=\partial_iu-\partial_d^\mathcal{A}u\partial_i\varphi=0-0=0  ~~~\text{ on } \Sigma_b.
\end{align}
The lemma is thus concluded.
\end{proof}
The estimates for these nonlinearities are presented in the following.
\begin{lem}\label{NW}
 It holds that for $i=1,d-1,d$,
\begin{align}\label{NW1}
\norm{G^{1,1,i,\sharp}}_{4N-2}^2
+\norm{G^{1,2,i,\sharp}}_{4N-1}^2
+\norm{G^{2,i,\sharp}}_{4N-2}^2
\lesssim \mathcal{E}_{2N}\mathfrak{D}_{2N}
\end{align}
and for $i=1,d-1$,
\begin{align}\label{NW3}
\abs{G^{3,1,i,\sharp}}_{4N-1}^2
+
\abs{G^{4,i,\sharp}}_{4N-2}^2
\lesssim \mathcal{E}_{2N}\mathfrak{D}_{2N}
\end{align}
and
\begin{align}\label{NW2}
\abs{G^{3,1,i,\sharp}}_{4N-1/2}^2
+\abs{G^{4,i,\sharp}}_{4N-3/2}^2
\lesssim \mathcal{E}_{2N}\mathfrak{D}_{2N}+\mathcal{D}_{N+2,1}\mathcal{F}_{2N}.
\end{align}

\end{lem}

\begin{proof}
The estimates \eqref{NW1} and \eqref{NW3} follow similarly as \eqref{ne5} in Lemma \ref{N2}. For the estimate \eqref{NW2},
all the terms can be controlled by $\mathcal{E}_{2N}\mathfrak{D}_{2N}$ except
$\abs{\nabla uD^{4N}\eta}_{1/2}^2$ when estimating  $G^{3,1,i,\sharp}$ and $G^{4,i,\sharp}$,
which is estimated as follows: by Lemma \ref{va1}, Sobolev's embedding and the trace theory,
\begin{align}\label{so2}
\abs{\nabla uD^{4N}\eta}_{1/2}^2
\lesssim \norm{\nabla u}_{C^1(\Sigma)}^2\abs{D^{4N}\eta}_{1/2}^2
\lesssim \mathcal{D}_{N+2,1}\mathcal{F}_{2N}.
\end{align}
The estimate  \eqref{NW2} then follows.
\end{proof}

We now record the energy evolution for $4N$ horizontal spatial derivatives.

\begin{prop}\label{prop5}
It holds that
\begin{align}\label{tan5}
&\frac{d}{dt}
\left(
\norm{D_\mathcal{A}q}_{0,4N-1}^{2}
+\norm{D_\mathcal{A}u}_{0,4N-1}^{2}
+\abs{D\eta}_{4N-1}^{2}
\right)
+\norm{ D_\mathcal{A}u}_{1,4N-1}^2
\nonumber\\&\quad
\lesssim \sqrt{\mathcal{E}_{2N}}\mathfrak{D}_{2N}
+ \sqrt{\mathfrak{D}_{2N}}
\sqrt{\mathcal{D}_{N+2,1}\mathcal{F}_{2N}}.
\end{align}
\end{prop}

\begin{proof}
Let $\alpha\in  \mathbb{N}^{d-1}$ be so that $0\leq|\alpha|\leq4N-1$.
Applying $\partial^\alpha$ to the second equation of \eqref{Aq1} for $i=1,d-1$,
and then taking the $L^2(\Omega)$ inner product of the resulting with $ \partial^\alpha\partial_i^\mathcal{A}u$, similarly as \eqref{youd},
we obtain
\begin{align}\label{NN1}
&\frac{1}{2}\frac{d}{dt}\left(\int_\Omega\left( \frac{1}{h'(\bar{\rho})}\abs{\partial^\alpha \partial_i^\mathcal{A}q}^2
+ \bar{\rho} \abs{\partial^\alpha\partial_i^\mathcal{A}u}^2\right)
+\int_{\Sigma}\rho^\star g\abs{\partial^\alpha\partial_i\eta}^2\right)
\nonumber\\
&\quad
+\int_{\Omega}\left(\frac{\mu}{2}\abs{\mathbb{D}^0\partial^\alpha \partial_i^\mathcal{A}u}^2
+\mu'\abs{\diverge \partial^\alpha \partial_i^\mathcal{A}u}^2\right)
\nonumber\\
&\quad
=\int_\Omega \(\frac{1}{h'(\bar{\rho})}\partial^\alpha \partial_i^\mathcal{A}q \partial^\alpha G^{1,i,\sharp}
+ \partial^\alpha \partial_i^\mathcal{A}u \cdot \partial^\alpha G^{2,i,\sharp}\)
+\int_{\Sigma}\(g\rho^\star \partial^\alpha  \partial_i\eta \partial^\alpha G^{3,i,\sharp}
-\partial^\alpha \partial_i^\mathcal{A}u\cdot \partial^\alpha G^{4,i,\sharp}\).
\end{align}

Now we estimate the right hand side of \eqref{NN1}.
For the $G^{1,i,\sharp}$ term, using \eqref{NW1}, we first directly get
\begin{align}\label{NN3}
\int_\Omega \frac{1}{h'(\bar{\rho})}\partial^\alpha \partial_i^\mathcal{A}q \partial^\alpha G^{1,2,i,\sharp}
\lesssim \norm{\partial_i^\mathcal{A}q}_{0,4N-1}\norm{ G^{1,2,i,\sharp}}_{0,4N-1}
\lesssim \sqrt{\mathfrak{D}_{2N}}\sqrt{\mathcal{E}_{2N}\mathfrak{D}_{2N}},
\end{align}
and arguing similarly as \eqref{s6}, we can derive
\begin{align}\label{NN4}
\int_\Omega \frac{1}{h'(\bar{\rho})}\partial^\alpha \partial_i^\mathcal{A}q \partial^\alpha G^{1,1,i,\sharp}
\lesssim
\sqrt{\mathfrak{D}_{2N}}\sqrt{\mathcal{E}_{2N}\mathfrak{D}_{2N}},
\end{align}
hence,
\begin{align}\label{NN5}
\int_\Omega \frac{1}{h'(\bar{\rho})}\partial^\alpha \partial_i^\mathcal{A}q \partial^\alpha G^{1,i,\sharp}
\lesssim \sqrt{\mathcal{E}_{2N}}\mathfrak{D}_{2N}.
\end{align}
For the $G^{2,i,\sharp}$ and $G^{4,i,\sharp}$ terms, we split into two cases: $|\alpha|\leq 4N-2$ and $|\alpha|= 4N-1$.
For the former case, it follows directly from \eqref{NW1} that
\begin{align}\label{NNa1}
&\int_\Omega  \partial^\alpha \partial_i^\mathcal{A}u \cdot \partial^\alpha G^{2,i,\sharp}
\lesssim
\norm{ \partial^\alpha \partial_i^\mathcal{A}u}_{0}\norm{ \partial^\alpha G^{2,i,\sharp}}_{0}
\lesssim  \sqrt{\mathfrak{D}_{2N}}
\sqrt{\mathcal{E}_{2N}\mathfrak{D}_{2N}},
\end{align}
and from the trace theory and \eqref{NW3} that,
\begin{align}
-\int_{\Sigma} \partial^\alpha \partial_i^\mathcal{A}u\cdot \partial^\alpha G^{4,i,\sharp}
&\lesssim
\abs{ \partial^\alpha \partial_i^\mathcal{A}u}_{0}\abs{ \partial^\alpha G^{4,i,\sharp}}_{0}
\lesssim
\norm{ \partial^\alpha \partial_i^\mathcal{A}u}_{1}\abs{ \partial^\alpha G^{4,i,\sharp}}_{0}
\nonumber\\&
\lesssim  \sqrt{\mathfrak{D}_{2N}}\sqrt{\mathcal{E}_{2N}\mathfrak{D}_{2N}}.
\end{align}
For the latter case, by integrating by parts and using \eqref{NW1},  we get
\begin{align}\label{NN2}
&\int_\Omega  \partial^\alpha \partial_i^\mathcal{A}u \cdot \partial^\alpha G^{2,i,\sharp}
\lesssim
\norm{ \partial_i^\mathcal{A}u}_{0,4N}\norm{ G^{2,i,\sharp}}_{0,4N-2}
\lesssim  \sqrt{\mathfrak{D}_{2N}}
\sqrt{\mathcal{E}_{2N}\mathfrak{D}_{2N}},
\end{align}
and using \eqref{NW2} and the trace theory,
\begin{align}\label{moli}
-\int_{\Sigma} \partial^\alpha \partial_i^\mathcal{A}u\cdot \partial^\alpha G^{4,i,\sharp}
&\lesssim
\abs{ \partial_i^\mathcal{A}u}_{4N-1/2} \abs{  G^{4,i,\sharp}}_{4N-3/2}
\lesssim
\norm{\partial_i^\mathcal{A}u}_{1,4N-1} \abs{  G^{4,i,\sharp}}_{4N-3/2}
\nonumber\\&
\lesssim  \sqrt{\mathfrak{D}_{2N}}
\sqrt{\mathcal{E}_{2N}\mathfrak{D}_{2N}+\mathcal{D}_{N+2,1}\mathcal{F}_{2N}}.
\end{align}
For the $G^{3,i,\sharp}$ term, it follows directly from \eqref{NW2} that
\begin{align}\label{NN6}
\int_{\Sigma}g\rho^\star\partial^\alpha \partial_i\eta \partial^\alpha G^{3,1,i,\sharp}
&\lesssim  \abs{D \eta}_{4N-3/2} \abs{ G^{3,1,i,\sharp}}_{4N-1/2}
\nonumber\\&
\lesssim \sqrt{\mathfrak{D}_{2N}}
\sqrt{\mathcal{E}_{2N}\mathfrak{D}_{2N}+\mathcal{D}_{N+2,1}\mathcal{F}_{2N}}.
\end{align}
For the term related to $u_h\cdot D \partial_i\eta$, we split into two cases: $|\alpha|=0$ and $|\alpha|\geq 1$.
For the former case, by Sobolev's embedding and the trace theory,
\begin{align}\label{etta}
-\int_{\Sigma}g\rho^\star \partial_i\eta   u_h \cdot D \partial_i\eta
\lesssim
\abs{D\eta}_{0}\norm{u_h}_{L^\infty(\Sigma)}\abs{D^2\eta}_{0}
\lesssim
\abs{D\eta}_{0}\norm{u_h}_{2}\abs{D^2\eta}_{0}
\lesssim
\sqrt{\mathcal{E}_{2N}}\mathfrak{D}_{2N}.
\end{align}
For the latter case,
note that
\begin{align}\label{NNm1}
\partial^\alpha \left(u_h\cdot D \partial_i\eta\right)
= u_h \cdot D \partial^\alpha\partial_i\eta+\left[\partial^\alpha,u_h\right]\cdot D \partial_i\eta.
\end{align}
By integrating by parts and using Lemma \ref{va1}, Sobolev's embedding and the trace theory, we find
\begin{align}\label{NNm2}
-\int_{\Sigma}g\rho^\star\partial^\alpha \partial_i\eta D \partial^\alpha\partial_i\eta \cdot u_h
&=-\frac{1}{2}\int_{\Sigma}g\rho^\star  D\abs{\partial^\alpha\partial_i\eta}^2\cdot u_h
= \frac{1}{2}\int_{\Sigma}g\rho^\star  \abs{\partial^\alpha\partial_i\eta}^2D \cdot u_h
\nonumber\\&
\lesssim
\abs{\partial^\alpha\partial_i\eta}_{-1/2}\abs{\partial^\alpha\partial_i\eta}_{1/2}\norm{Du_h}_{C^1(\Sigma)}
\lesssim
\sqrt{\mathfrak{D}_{2N}\mathcal{F}_{2N}\mathcal{D}_{N+2,1}},
\end{align}
and by \eqref{Pe3}, it holds that
\begin{align}
\abs{\left[\partial^\alpha,u_h\right] \cdot D \partial_i\eta}_{1/2}^2
&\lesssim
\abs{D^{|\alpha|} u_h}_{1/2}^2\norm{D \partial_i\eta}_{C^1(\Sigma)}^2
+\abs{D^{|\alpha|} \partial_i\eta}_{1/2}^2\norm{Du_h}_{C^1(\Sigma)}^2
\nonumber\\&
\lesssim
\norm{ u_h}_{4N}^2\norm{D ^2\eta}_{C^1(\Sigma)}^2+\abs{\eta}_{4N+1/2}^2\norm{Du_h}_{C^1(\Sigma)}^2
\nonumber\\&
\lesssim
\mathcal{E}_{2N}\mathfrak{D}_{2N}+\mathcal{F}_{2N}\mathcal{D}_{N+2,1}
\end{align}
and
\begin{align}\label{NNm3}
-\int_{\Sigma}g\rho^\star\partial^\alpha\partial_i\eta \left[\partial^\alpha,u_h\right]\cdot D \partial_i\eta
&\lesssim
\abs{\partial^\alpha\partial_i\eta}_{-1/2}\abs{\left[\partial^\alpha,u_h\right]\cdot D \partial_i\eta}_{1/2}
\nonumber\\&
\lesssim
 \sqrt{\mathfrak{D}_{2N}}
\sqrt{\mathcal{E}_{2N}\mathfrak{D}_{2N}+\mathcal{F}_{2N}\mathcal{D}_{N+2,1}},
\end{align}
hence, we have
\begin{align}\label{NNm4}
-\int_{\Sigma}g\rho^\star\partial^\alpha \partial_i\eta \partial^\alpha \left(u_h\cdot D \partial_i\eta\right)
\lesssim
 \sqrt{\mathfrak{D}_{2N}}
\sqrt{\mathcal{E}_{2N}\mathfrak{D}_{2N}+\mathcal{F}_{2N}\mathcal{D}_{N+2,1}}.
\end{align}
In light of \eqref{NN6}, \eqref{etta} and \eqref{NNm4}, we deduce that
\begin{align}\label{NN8}
\int_{\Sigma}g\rho^\star\partial^\alpha \partial_i\eta \partial^\alpha G^{3,i,\sharp}
\lesssim \sqrt{\mathfrak{D}_{2N}}
\sqrt{\mathcal{E}_{2N}\mathfrak{D}_{2N}+\mathcal{D}_{N+2,1}\mathcal{F}_{2N}}.
\end{align}

Consequently, by \eqref{NN5}--\eqref{moli}  and \eqref{NN8}, summing \eqref{NN1} over such $\alpha$,
using \eqref{viscosity} and Korn's inequality of Lemma \ref{xm5},
since $\partial^\alpha D_\mathcal{A}u=0$ on $\Sigma_b$, we conclude \eqref{tan5}.
\end{proof}


\section{Energy evolution of $\partial_d q$}\label{pdq}
In order to go from the control of the tangential energy evolution estimates to the full one, we will employ the elliptic estimates.
However, unlike the incompressible case \cite{GT_inf,WYJ1}, as in \cite{WYJ2} we are forced to control  $\partial_d q$ first.
Therefore, in this section, we shall provide the energy evolution estimates for $\partial_d q$.

Recall the definitions \eqref{G11} of $G^{1,1}$ and \eqref{G12} of $G^{1,2}$.
We define
\begin{align}\label{dq1}
\mathbb{Q}:=\partial_tq-J^{-1}\partial_t\varphi\partial_dq+u\cdot \nabla_\mathcal{A}q
=\partial_tq-G^{1,1}=-h'(\bar{\rho})\diverge (\bar{\rho}u)+G^{1,2}.
\end{align}
Applying $\partial_d$ to \eqref{dq1} yields
\begin{align}\label{dq2}
\partial_d\mathbb{Q}+h'(\bar{\rho})\bar{\rho}\partial_d\diverge u
=-\partial_dP'(\bar{\rho})\diverge u+g\partial_d u_d+\partial_dG^{1,2}.
\end{align}
On the other hand, recall the vertical component of the second equation in \eqref{q2}:
\begin{align}\label{dq3}
\bar{\rho}\partial_tu_d+\bar{\rho}\partial_dq
-\mu \Delta u_d-\left(\frac{d-2}{d}\mu+\mu'\right)\partial_d\diverge u=G_d^2.
\end{align}
According to \eqref{dq2} and \eqref{dq3}, we then eliminate $\partial_{d}^2u_d$ to get
\begin{align}\label{dq4}
\frac{2(d-1)\mu+d\mu'}{dh'(\bar{\rho})\bar{\rho}^2}\partial_d\mathbb{Q}
+\partial_dq
=&
-\frac{2(d-1)\mu+d\mu'}{dh'(\bar{\rho})\bar{\rho}^2}\left(\partial_dP'(\bar{\rho})\diverge u-g\partial_d u_d\right)
-\partial_tu_d
\nonumber\\&
+\frac{\mu }{\bar{\rho}}(\Delta_hu_d-\partial_d\diverge_h u_h)
\frac{2(d-1)\mu+d\mu'}{dh'(\bar{\rho})\bar{\rho}^2}\partial_dG^{1,2}+\frac{1}{\bar{\rho}}G_d^2
,
\end{align}
where $\Delta_h $ denotes  the horizontal  Laplacian   and  $\diverge_h $ the  horizontal  divergence.
By the definition \eqref{dq1} of $\mathbb{Q}$, we may view \eqref{dq4} as an evolution equation with a damping term for  $\partial_dq$. This equation
resembles the ODE $\partial_t f+f=g$, up to some errors, and this ODE displays natural dissipation structure.

Again, if we would use \eqref{dq4}  to estimate  the $4N-1$ order spatial derivatives of $\partial_d q$,
then it would lead to the appearance of $\|\nabla u\|_{L^\infty(\Omega)}\mathcal{F}_{2N}$
in the nonlinear estimates. To overcome this, we consider instead the evolution of  $\partial_d\partial_i^\mathcal{A} q$.
More precisely, for $i=1,d-1,d$, recall the definitions \eqref{w11}  of $G^{1,1,i,\sharp}$ and \eqref{w12} of $G^{1,2,i,\sharp}$.
We define
\begin{align}\label{xx2}
\mathbb{Q}^{i,\sharp}
:&=\partial_t\partial_i^\mathcal{A}q-J^{-1}\partial_t\varphi\partial_d\partial_i^\mathcal{A}q+u\cdot \nabla_\mathcal{A}\partial_i^\mathcal{A}q
=\partial_t\partial_i^\mathcal{A}q-G^{1,1,i,\sharp}
\nonumber\\&
=-h'(\bar{\rho})\diverge\left(\bar{\rho} \partial_i^\mathcal{A}u\right)
-\partial_i P'(\bar{\rho})\diverge u
+G^{1,2,i,\sharp}
.
\end{align}
Similarly as \eqref{dq4}, we utilize the vertical component of the second equation
in \eqref{Aq1} and \eqref{xx2} by eliminating $\partial_d^2 \partial_i^\mathcal{A}u$, to find
\begin{align}\label{hj}
&\frac{2(d-1)\mu+d\mu'}{dh'(\bar{\rho})\bar{\rho}^2}
\partial_d\mathbb{Q}^{i,\sharp}
+\partial_d\partial_i^\mathcal{A}q
\nonumber\\&\quad
=
-\frac{2(d-1)\mu+d\mu'}{dh'(\bar{\rho})\bar{\rho}^2}\Big(\partial_dP'(\bar{\rho})\diverge \partial_i^\mathcal{A}u
-g \partial_d\partial_i^\mathcal{A}u_d\Big.
\Big.
                                                         +\partial_d(\partial_iP'(\bar{\rho})\diverge u)
                                                         \Big)
\nonumber\\&\quad\quad
-\partial_t\partial_i^\mathcal{A}u_d
-\frac{\partial_i\bar{\rho}}{\bar{\rho}^2}\(\diverge\mathbb{S} u\)_d
+\frac{\mu }{\bar{\rho}}\Big(\Delta_h\partial_i^\mathcal{A}u_d-\partial_d\diverge_h \partial_i^\mathcal{A}u_h\Big)
\nonumber\\&\quad\quad
+\frac{2(d-1)\mu+d\mu'}{dh'(\bar{\rho})\bar{\rho}^2}\partial_dG^{1,2,i,\sharp}
+\frac{1}{\bar{\rho}}G_d^{2,i,\sharp}
.
\end{align}

We first record the energy evolution of $\partial_dq$ at the $2N$ level.
\begin{prop}\label{prop6}
$(1)$ Fix $1\leq j\leq 2N-1$ and $0\leq k \leq 4N-2j-1$. It holds that
\begin{align}\label{v1}
&\frac{d}{dt}
\norm{\partial_d\partial_t^jq}_{k,4N-2j-k-1}^2
+\norm{\partial_d\partial_t^jq}_{k,4N-2j-k-1}^2
+\norm{\partial_d\partial_t^j\mathbb{Q}}_{k,4N-2j-k-1}^2
\nonumber\\&\quad
\lesssim \norm{\partial_t^{j+1} u}_{4N-2j-1}^2+\norm{\partial_t^j\mathbb{Q}}_{0,4N-2j-k-1}^2
+\norm{\partial_t^j u}_{k+1,4N-2j-k}^2
+\sqrt{\mathcal{E}_{2N}}\mathfrak{D}_{2N}.
\end{align}
$(2)$ Fix $0\leq k \leq 4N-2$. It holds that for $i=1,d-1,d$,
\begin{align}\label{vs}
&\frac{d}{dt}
\norm{\partial_d\partial_i^\mathcal{A}q}_{k,4N-k-2}^2
+\norm{\partial_d\partial_i^\mathcal{A}q}_{k,4N-k-2}^2
+\norm{\partial_d\mathbb{Q}^{i,\sharp}}_{k,4N-k-2}^2
\nonumber\\&\quad
\lesssim
\norm{\partial_t\partial_i^\mathcal{A}u}_{4N-2}^2
+\norm{\mathbb{Q}^{i,\sharp}}_{0,4N-k-2}^2
+\norm{\partial_i^\mathcal{A}u}_{k+1,4N-k-1}^2
+\delta_{id}\norm{u}_{4N}^2
+\sqrt{\mathcal{E}_{2N}}\mathfrak{D}_{2N}.
\end{align}

\end{prop}

\begin{proof}
We may first focus on the derivation of the estimate \eqref{vs} for $i=d$. We first fix $0\leq k\leq 4N-2$,
and then let $0\leq k'\leq k$ and $\alpha\in \mathbb{N}^{d-1}$
so that $|\alpha|\leq 4N-2-k'$. Applying $\partial^\alpha\partial_d^{k'}$ to \eqref{hj}
and then taking the $L^2(\Omega)$ inner product of the resulting with $\partial^\alpha\partial_d^{k'+1}\left(\partial_d^\mathcal{A}q+\mathbb{Q}^{d,\sharp}\right)$,
we obtain
\begin{align}\label{v5}
\uppercase\expandafter{\romannumeral1}
+\uppercase\expandafter{\romannumeral2}
=\uppercase\expandafter{\romannumeral3},
\end{align}
where
\begin{align}\label{v6}
\uppercase\expandafter{\romannumeral1}
=\int_{\Omega}\partial^\alpha\partial_d^{k'}
\left(\frac{2(d-1)\mu+d\mu'}{dh'(\bar{\rho})\bar{\rho}^2}\partial_d\mathbb{Q}^{d,\sharp}\right)
\times \left(\partial^\alpha\partial_d^{k'+1}\partial_d^\mathcal{A}q
+\partial^\alpha\partial_d^{k'+1}\mathbb{Q}^{d,\sharp}\right),
\end{align}

\begin{align}\label{v7}
\uppercase\expandafter{\romannumeral2}
=\int_{\Omega}\abs{\partial^\alpha\partial_d^{k'+1} \partial_d^\mathcal{A}q}^2
+\int_{\Omega} \partial^\alpha\partial_d^{k'+1}\partial_d^\mathcal{A}q\partial^\alpha\partial_d^{k'+1}\mathbb{Q}^{d,\sharp}
\end{align}
and
\begin{align}\label{v8}
\uppercase\expandafter{\romannumeral3}
=&\int_{\Omega}\partial^\alpha\partial_d^{k'}
\left\{
-\frac{2(d-1)\mu+d\mu'}{dh'(\bar{\rho})\bar{\rho}^2}\left(\partial_dP'(\bar{\rho})\left(\diverge \partial_d^\mathcal{A}u
                                                +\partial_d\diverge u        \right)
                                                +\partial_d ^2P'(\bar{\rho})\diverge u
                                                -g \partial_d\partial_d^\mathcal{A}u_d
                                                \right)
\right.
\nonumber\\&\quad\quad\quad\quad\;\,\,
\left.
-\partial_t\partial_d^\mathcal{A}u_d
+\frac{\partial_i\bar{\rho}}{\bar{\rho}^2}(\diverge\mathbb{S} u)_d
+\frac{\mu }{\bar{\rho}}\left(\Delta_h\partial_d^\mathcal{A}u_d-\partial_d\diverge_h \partial_d^\mathcal{A}u_h\right)
\right.
\nonumber\\&\quad\quad\quad\quad\;\,\,
\left.
+\frac{2(d-1)\mu+d\mu'}{dh'(\bar{\rho})\bar{\rho}^2}\partial_dG^{1,2,d,\sharp}
+\frac{1}{\bar{\rho}}G_d^{2,d,\sharp}
\right\}
\times
\left(\partial^\alpha\partial_d^{k'+1}\partial_d^\mathcal{A}q
+\partial^\alpha\partial_d^{k'+1}\mathbb{Q}^{d,\sharp}\right)
 .
\end{align}
Recall the definition of $\mathbb{Q}^{d,\sharp}$, we find
\begin{align}\label{I2}
\uppercase\expandafter{\romannumeral2}
=\frac{1}{2}\frac{d}{dt}\int_{\Omega}\abs{\partial^\alpha\partial_d^{k'+1}\partial_d^\mathcal{A}q}^2
+\int_{\Omega}\abs{\partial^\alpha\partial_d^{k'+1}\partial_d^\mathcal{A}q}^2
-\int_{\Omega}\partial^\alpha\partial_d^{k'+1}\partial_d^\mathcal{A}q\partial^\alpha\partial_d^{k'+1}G^{1,1,d,\sharp}
\end{align}
and
\begin{align}\label{I3}
\uppercase\expandafter{\romannumeral1}
\geq&
\frac{1}{2}\frac{d}{dt}\int_{\Omega}\frac{2(d-1)\mu+d\mu'}{dh'(\bar{\rho})\bar{\rho}^2}
\abs{\partial^\alpha\partial_d^{k'+1}\partial_d^\mathcal{A}q}^2
+\int_{\Omega}\frac{2(d-1)\mu+d\mu'}{dh'(\bar{\rho})\bar{\rho}^2}\abs{\partial^\alpha\partial_d^{k'+1}\mathbb{Q}^{d,\sharp}}^2
\nonumber\\&
-C\sum_{k''\leq k'}\norm{\partial^\alpha\partial_d^{k''}\mathbb{Q}^{d,\sharp}}_0
\left(\norm{\partial^\alpha\partial_d^{k'+1}\partial_d^\mathcal{A}q}_0
+\norm{\partial^\alpha\partial_d^{k'+1}\mathbb{Q}^{d,\sharp}}_0\right)
\nonumber\\&
-\int_{\Omega}\frac{2(d-1)\mu+d\mu'}{dh'(\bar{\rho})\bar{\rho}^2}
\partial^\alpha\partial_d^{k'+1}\partial_d^\mathcal{A}q \partial^\alpha\partial_d^{k'+1}G^{1,1,d,\sharp}.
\end{align}
For $\uppercase\expandafter{\romannumeral3}$, we can derive
\begin{align}\label{I4}
\uppercase\expandafter{\romannumeral3}
\lesssim &
\Bigg\{
\norm{\partial_t\partial_d^\mathcal{A}u}_{4N-2}
+\sum_{k''\leq k'}\left(\norm{\partial^\alpha\partial_d^{k''}\nabla \partial_d^\mathcal{A}u}_{0}
+\norm{\partial^\alpha\partial_d^{k''}D\nabla \partial_d^\mathcal{A}u}_{0}
+\norm{\partial^\alpha\partial_d^{k''}\nabla u}_{1}\right)
\Bigg.
\nonumber\\&\quad
\Bigg.
+\norm{G^{1,2,d,\sharp}}_{4N-1}+\norm{G^{2,d,\sharp}}_{4N-2}
\Bigg\}
\times
\displaystyle\left(\norm{\partial^\alpha\partial_d^{k'+1}\partial_d^\mathcal{A}q}_0
+\norm{\partial^\alpha\partial_d^{k'+1}\mathbb{Q}^{d,\sharp}}_0\right).
\end{align}
Similarly as \eqref{s6}, we have
\begin{align}\label{ti1}
&\int_{\Omega}\left(1+\frac{2(d-1)\mu+d\mu'}{dh'(\bar{\rho})\bar{\rho}^2}\right)
\partial^\alpha\partial_d^{k'+1}\partial_d^\mathcal{A}q
\partial^\alpha\partial_d^{k'+1}G^{1,1,d,\sharp}
\lesssim \sqrt{\mathcal{E}_{2N}}\mathfrak{D}_{2N}.
\end{align}

In light of \eqref{I2}--\eqref{ti1}, by Cauchy's inequality, we deduce from \eqref{v5}  that,  by \eqref{NW1},
\begin{align}\label{I5}
&\frac{d}{dt}\norm{\partial^\alpha\partial_d^{k'+1}\partial_d^\mathcal{A}q}_0^2
+\norm{\partial^\alpha\partial_d^{k'+1}\partial_d^\mathcal{A}q}_0^2
+\norm{\partial^\alpha\partial_d^{k'+1}\mathbb{Q}^{d,\sharp}}_0^2
\nonumber\\&\quad
\lesssim
\sum_{k''\leq k'}\norm{\partial^\alpha\partial_d^{k''}\mathbb{Q}^{d,\sharp}}_0^2
+\sum_{k''\leq k'}\left(\norm{\partial^\alpha\partial_d^{k''}\nabla \partial_d^\mathcal{A}u}_{0}^2
+\norm{\partial^\alpha\partial_d^{k''}D\nabla \partial_d^\mathcal{A}u}_{0}^2\right)
\nonumber\\&\quad\quad
+\norm{u}_{4N}^2
+\norm{\partial_t\partial_d^\mathcal{A}u}_{4N-2}^2
+\norm{G^{1,2,d,\sharp}}_{4N-1}^2
+\norm{G^{2,d,\sharp}}_{4N-1}^2
\nonumber\\&\quad
\lesssim
\sum_{k''\leq k'}\norm{\partial^\alpha\partial_d^{k''}\mathbb{Q}^{d,\sharp}}_0^2
+\sum_{k''\leq k'}\norm{\partial^\alpha\partial_d^{k''}\partial_d^\mathcal{A}u}_{1,1}^2
+\norm{u}_{4N}^2
+\norm{\partial_t\partial_d^\mathcal{A}u}_{4N-2}^2
+\sqrt{\mathcal{E}_{2N}}\mathfrak{D}_{2N}.
\end{align}
Summing \eqref{I5} over such $\alpha$ yields
\begin{align}\label{i5}
&\frac{d}{dt}\norm{\partial_d^{k'+1}\partial_d^\mathcal{A}q}_{0,4N-2-k'}^2
+\norm{\partial_d^{k'+1}\partial_d^\mathcal{A}q}_{0,4N-2-k'}^2
+\norm{\partial_d^{k'+1}\mathbb{Q}^{d,\sharp}}_{0,4N-2-k'}^2
\nonumber\\&\quad
\lesssim
\sum_{k''\leq k'}\norm{\partial_d^{k''}\mathbb{Q}^{d,\sharp}}_{0,4N-2-k'}^2
+\sum_{k''\leq k'}\norm{\partial_d^{k''}\partial_d^\mathcal{A}u}_{1,4N-1-k'}^2
\nonumber\\&\quad\quad
+\norm{u}_{4N}^2
+\norm{\partial_t\partial_d^\mathcal{A}u}_{4N-2}^2
+\sqrt{\mathcal{E}_{2N}}\mathfrak{D}_{2N}.
\end{align}

A suitable linear combination of  \eqref{i5} from $k'=k$ to $k'=0$ leads to
\begin{align}\label{I6}
&\frac{d}{dt}
\norm{\partial_d\partial_d^\mathcal{A}q}_{k,4N-k-2}^2
+\norm{\partial_d\partial_d^\mathcal{A}q}_{k,4N-k-2}^2
+\norm{\partial_d\mathbb{Q}^{d,\sharp}}_{k,4N-k-2}^2
\nonumber\\&\quad
\lesssim
\norm{\partial_t\partial_d^\mathcal{A}u}_{4N-2}^2
+\norm{\mathbb{Q}^{d,\sharp}}_{0,4N-k-2}^2
+
\norm{\partial_d^\mathcal{A}u}_{k+1,4N-k-1}^2
+\norm{u}_{4N}^2
+\sqrt{\mathcal{E}_{2N}}\mathfrak{D}_{2N}.
\end{align}
This gives \eqref{vs} for $i=d$.

The estimates \eqref{vs} for $i=1,d-1$ follow similarly as that for $i=d$,
and the estimate \eqref{v1} follows similarly as \eqref{vs} by using \eqref{dq4} and \eqref{ne5}
in place of \eqref{hj} and \eqref{NW1}, respectively.
\end{proof}

We then record a similar result at the $N+2$ level.
\begin{prop}\label{prop8}
$(1)$ Fix $1\leq j\leq N+1$ and $0\leq k \leq 2(N+2)-2j-1$. It holds that
\begin{align}\label{I7}
&\frac{d}{dt}
\norm{\partial_d\partial_t^jq}_{k,2(N+2)-2j-k-1}^2
+\norm{\partial_d\partial_t^jq}_{k,2(N+2)-2j-k-1}^2
+\norm{\partial_d\partial_t^j\mathbb{Q}}_{k,2(N+2)-2j-k-1}^2
\nonumber\\&\quad
\lesssim \norm{\partial_t^{j+1} u}_{2(N+2)-2j-1}^2+\norm{\partial_t^j\mathbb{Q}}_{0,2(N+2)-2j-k-1}^2
+\norm{\partial_t^j u}_{k+1,2(N+2)-2j-k}^2
+\sqrt{\mathcal{E}_{2N}}\mathcal{D}_{N+2,1}.
\end{align}
$(2)$ Fix $0\leq k \leq 2(N+2)-1$. It holds that
\begin{align}\label{v2}
&\frac{d}{dt}
\norm{\partial_dq}_{k,2(N+2)-k-1}^2
+\norm{\partial_dq}_{k,2(N+2)-k-1}^2
+\norm{\partial_d\mathbb{Q}}_{k,2(N+2)-k-1}^2
\nonumber\\&\quad
\lesssim \norm{\partial_t u}_{2(N+2)-1}^2+\norm{\mathbb{Q}}_{0,2(N+2)-k-1}^2
+\norm{ D u}_{k+1,2(N+2)-k-1}^2
\nonumber\\&\quad\quad
+\norm{ u_d}_{k+1,2(N+2)-k-1}^2
+\sqrt{\mathcal{E}_{2N}}\mathcal{D}_{N+2,1}.
\end{align}
\end{prop}
\begin{proof}
The proof is similar to that of Proposition \ref{prop6} by using \eqref{dq4} and \eqref{ne2}
in place of \eqref{hj} and \eqref{NW1},  respectively.
\end{proof}

\section{Elliptic regularity}
In this section, we will provide the elliptic regularity estimates that will be employed to derive the estimates for the full derivatives
of the solution.
\subsection{The Lam\'{e} system}
We shall use the following Lam\'{e} system derived from \eqref{q2} for the estimates in the energy:
\begin{equation}\label{q5}
\begin{cases}
-\mu \Delta u-\left(\displaystyle\frac{d-2}{d}\mu+\mu'\right)\nabla \diverge u
            = -\bar{\rho}\partial_t u-\bar{\rho}\nabla q +G^2 & \text{ in } \Omega
\\-\mathbb{S}ue_d=-\rho^\star qe_d+\rho^\star g \eta e_d+G^4   & \text{ on  } \Sigma
\\u=0    & \text{ on  } \Sigma_b.
\end{cases}
\end{equation}

We first record the  estimates at the $2N$ level.
\begin{lem}\label{llem}
 It holds that for $0\leq j\leq 2N-1$,
\begin{align}\label{chf1}
\norm{\partial_t^ju}_{4N-2j}^2
\lesssim \norm{\partial_t^{j+1}u}_{4N-2j-2}^2
+\norm{\partial_t^j q}_{4N-2j-1}^2
+\abs{\partial_t^j\eta}_{4N-2j-3/2}^2
+(\mathcal{E}_{2N})^2.
\end{align}
\end{lem}
\begin{proof}
Fix $0\leq j\leq 2N-1$.
 Applying $\partial_t^j$ to \eqref{q5},
 and then employing the elliptic estimates of Lemma \ref{St1} with $r=4N-2j\geq 2$, by the trace theory and \eqref{ne3}, we  obtain
\begin{align}\label{ee1}
\norm{\partial_t^ju}_{4N-2j}^2
&\lesssim  \norm{\partial_t^{j+1}u}_{4N-2j-2}^2
+\norm{\partial_t^j\nabla q}_{4N-2j-2}^2+\norm{\partial_t^jG^2}_{4N-2j-2}^2
\nonumber\\&\quad
+\abs{\partial_t^jq}_{4N-2j-3/2}^2+\abs{\partial_t^j\eta}_{4N-2j-3/2}^2+\abs{\partial_t^jG^4}_{4N-2j-3/2}^2
\nonumber\\&
\lesssim \norm{\partial_t^{j+1}u}_{4N-2j-2}^2
+\norm{\partial_t^j q}_{4N-2j-1}^2
+\abs{\partial_t^j\eta}_{4N-2j-3/2}^2
+(\mathcal{E}_{2N})^2.
\end{align}
This completes \eqref{chf1}.
\end{proof}

We then record the  estimates at the $N+2$ level.
\begin{lem}\label{llem2}
 It holds that for $1\leq j\leq N+1$,
\begin{align}\label{chf2}
&\norm{\partial_t^ju}_{2(N+2)-2j}^2
\lesssim
\norm{\partial_t^{j+1}u}_{2(N+2)-2j-2}^2
+\norm{\partial_t^j q}_{2(N+2)-2j-1}^2
+\abs{\partial_t^j\eta}_{2(N+2)-2j-3/2}^2
+\mathcal{E}_{2N}\mathcal{E}_{N+2,1}
\end{align}
and that
\begin{align}\label{chf3}
&
\norm{Du}_{2(N+2)-1}^2
\lesssim
\norm{\partial_tu}_{2(N+2)-2}^2
+\norm{ Dq}_{2(N+2)-2}^2
+\abs{D\eta}_{2(N+2)-5/2}^2
+\mathcal{E}_{2N}\mathcal{E}_{N+2,1}.
\end{align}
\end{lem}

\begin{proof}
The proof is similar to that of Lemma \ref{llem} by using \eqref{ne1}
in place of \eqref{ne3}.
\end{proof}
\subsection{The Stokes problem}
We shall use the Stokes system for the estimates in the dissipation.
Using the second equation in \eqref{q2} and \eqref{dq1}, we can derive the following Stokes system:
\begin{equation}\label{Sto1}
\begin{cases}
-\mu\Delta \left(\displaystyle\frac{u}{\bar{\rho}}\right)+\nabla q
=-\partial_t u
-\mu\left(2\partial_d\left(\displaystyle\frac{1}{\bar{\rho}}\right)\partial_d u+\partial_d^2\left(\displaystyle\frac{1}{\bar{\rho}}\right)u\right)
 \\
\quad\quad\quad\quad\quad\quad\quad\quad\;\;
+\displaystyle\frac{(d-2)\mu+d\mu'}{d\bar{\rho}}\nabla\left(\frac{-\mathbb{Q}+gu_d+G^{1,2}}{h'(\bar{\rho})\bar{\rho}}
                                                 \right)
                                                 +\displaystyle\frac{1}{\bar{\rho}}G^2
&\text{in }\Omega
\\
\diverge\left(\displaystyle\frac{u}{\bar{\rho}}\right)
=\displaystyle\frac{-\mathbb{Q}+2gu_d+G^{1,2}}{h'(\bar{\rho})\bar{\rho}^2}  &\text{in }\Omega
\\
u=u     &\text{on }\Sigma
\\
u=0     &\text{on }\Sigma_b.
\end{cases}
\end{equation}
Similarly, using the second equation in \eqref{Aq1} and \eqref{xx2}, we derive that for $i=1,d-1,d$,
\begin{equation}\label{Stop}
\begin{cases}
-\mu\Delta \left(\displaystyle\frac{\partial_i^\mathcal{A}u}{\bar{\rho}}\right)+\nabla \partial_i^\mathcal{A}q
\\=-\partial_t \partial_i^\mathcal{A}u
-\displaystyle\frac{\partial_i\bar{\rho}}{\bar{\rho}^2}\diverge\mathbb{S} u
-\mu\left(2\partial_d\left(\displaystyle\frac{1}{\bar{\rho}}\right)\partial_d \partial_i^\mathcal{A}u+\partial_d^2\left(\displaystyle\frac{1}{\bar{\rho}}\right)\partial_i^\mathcal{A}u\right)
\\
\quad
+ \displaystyle\frac{(d-2)\mu+d\mu'}{d\bar{\rho}}
\nabla
\Bigg(\frac{ -\mathbb{Q}^{i,\sharp}  +g\partial_i^\mathcal{A}u_d
                                                           -\partial_iP'(\bar{\rho})\diverge u +G^{1,2,i,\sharp}}{h'(\bar{\rho})\bar{\rho}}
                                                  \Bigg)
                                                  +\displaystyle\frac{1}{\bar{\rho}}G^{2,i,\sharp}
&\text{in }\Omega
\\ \diverge\left(\displaystyle\frac{\partial_i^\mathcal{A}u}{\bar{\rho}}\right)
=\displaystyle\frac{-\mathbb{Q}^{i,\sharp}+2g\partial_i^\mathcal{A}u_d-\partial_iP'(\bar{\rho})\diverge u+G^{1,2,i,\sharp} }{h'(\bar{\rho})\bar{\rho}^2}
 &\text{in }\Omega
\\ \partial_i^\mathcal{A}u=\partial_i^\mathcal{A}u  &\text{on }\Sigma
\\ \partial_i^\mathcal{A}u=\delta_{id}\partial_i^\mathcal{A}u     &\text{on }\Sigma_b.
\end{cases}
\end{equation}
\begin{rem}
It should be remarked here that the dissipation estimates of energy evolution estimates in Section \ref{pdq}
 provided the needed boundary regularity control of $u$ and $D_\mathcal{A}u$.
 However, we do not have the energy evolution estimate for $\partial_d^\mathcal{A}u$.
Fortunately, after deriving the full dissipation estimate of $D_\mathcal{A}u$, we can get the boundary regularity control of $\partial_d^\mathcal{A}u$ due to the fact that $D_\mathcal{A}\partial_d^\mathcal{A}u=\partial_d^\mathcal{A}D_\mathcal{A} u$.
\end{rem}
We first record  the estimates at the $2N$ level.
\begin{lem}\label{lem3}
$(1)$ Fix $1\leq j\leq 2N-1$ and  $1\leq k\leq 4N-2j$. It holds that
\begin{align}\label{Se1}
&\norm{\partial_t^j u}_{k+1,4N-2j-k}^2+\norm{\nabla \partial_t^j q}_{k-1,4N-2j-k}^2
\nonumber\\&\quad
\lesssim \norm{\partial_t^{j+1}u}_{4N-2j-1}^2+\norm{\partial_t^j \mathbb{Q}}_{k,4N-2j-k}^2
+\norm{\partial_t^ju}_{1,4N-2j}^2+\mathcal{E}_{2N}\mathfrak{D}_{2N}.
\end{align}
$(2)$ Fix $1\leq k\leq 4N-1$. It holds that for $i=1,d-1,d$,
\begin{align}\label{Ses}
&\norm{ \partial_i^\mathcal{A}u}_{k+1,4N-k-1}^2+\norm{\nabla \partial_i^\mathcal{A}q}_{k-1,4N-k-1}^2
\nonumber\\&\quad
\lesssim \norm{\partial_t\partial_i^\mathcal{A}u}_{4N-2}^2+\norm{\mathbb{Q}^{i,\sharp}}_{k,4N-k-1}^2
+\norm{\partial_i^\mathcal{A}u}_{1,4N-1}^2
+\delta_{id}\norm{ u}_{4N}^2+\mathcal{E}_{2N}\mathfrak{D}_{2N}.
\end{align}
\end{lem}

\begin{proof}
We first prove \eqref{Ses} for $i=d$. Fix $1\leq k\leq 4N-1$ and then
let $\alpha \in \mathbb{N}^{d-1}$ be so that $0\leq |\alpha| \leq 4N-k-1$.
Applying $\partial^\alpha $ to \eqref{Stop}, and then employing
the elliptic estimates of Lemma \ref{St1} with $r=k'+1\geq2$ for any $1\leq k'\leq k$,
by \eqref{NW1}  and the trace theory, we obtain
\begin{align}\label{fq1}
&\norm{\partial^\alpha \partial_d^\mathcal{A}u}_{k'+1}^2+\norm{\nabla \partial^\alpha \partial_d^\mathcal{A}q}_{k'-1}^2
\nonumber\\&\quad
\lesssim
\norm{\partial^\alpha \left(\frac{\partial_d^\mathcal{A}u}{\bar{\rho}}\right)}_{k'+1}^2
+\norm{\nabla \partial^\alpha \partial_d^\mathcal{A}q}_{k'-1}^2
\nonumber\\&\quad
\lesssim
\norm{\partial^\alpha \partial_t\partial_d^\mathcal{A}u}_{k'-1}^2
+\norm{\partial^\alpha \nabla u}_{k'}^2+\norm{\partial^\alpha \partial_d^\mathcal{A}u}_{k'}^2
+\norm{\partial^\alpha\mathbb{Q}^{d,\sharp}}_{k'}^2
\nonumber\\&\quad\quad
+\norm{\partial^\alpha G^{1,2,d,\sharp}}_{k'}^2
+\norm{\partial^\alpha G^{2,d,\sharp}}_{k'-1}^2
+\abs{\partial^\alpha \partial_d^\mathcal{A}u}_{k'+1/2}^2
+\norm{\partial^\alpha\partial_d^\mathcal{A}u}_{H^{k'+1/2}(\Sigma_b)}^2
\nonumber\\&\quad
\lesssim
\norm{\partial_t\partial_d^\mathcal{A}u}_{4N-2}^2
+\norm{u}_{4N}^2
+\norm{\partial^\alpha \partial_d^\mathcal{A}u}_{k'}^2
+\norm{D_0^{4N-k-1}\mathbb{Q}^{d,\sharp}}_{k}^2
\nonumber\\&\quad\quad
+\abs{\partial_d^\mathcal{A}u}_{4N-1/2}^2
+\norm{\partial_d^\mathcal{A}u}_{H^{4N-1/2}(\Sigma_b)}^2
+\mathcal{E}_{2N}\mathfrak{D}_{2N}
\nonumber\\&\quad
\lesssim
\norm{\partial_t\partial_d^\mathcal{A}u}_{4N-2}^2
+\norm{ u}_{4N}^2
+\norm{\partial^\alpha \partial_d^\mathcal{A}u}_{k'}^2
+\norm{\mathbb{Q}^{d,\sharp}}_{k,4N-k-1}^2
\nonumber\\&\quad\quad
+\norm{\partial_d^\mathcal{A}u}_{1,4N-1}^2
+\mathcal{E}_{2N}\mathfrak{D}_{2N}.
\end{align}

A simple induction on \eqref{fq1} from $k'=k$ to $k'=1$ yields that
\begin{align}\label{fq2}
&\norm{\partial^\alpha \partial_d^\mathcal{A}u}_{k+1}^2+\norm{\nabla \partial^\alpha \partial_d^\mathcal{A}q}_{k-1}^2
\nonumber\\&\quad
\lesssim
\norm{\partial_t\partial_d^\mathcal{A}u}_{4N-2}^2
+\norm{\mathbb{Q}^{d,\sharp}}_{k,4N-k-1}^2
+\norm{\partial_d^\mathcal{A}u}_{1,4N-1}^2
+\norm{ u}_{4N}^2+\mathcal{E}_{2N}\mathfrak{D}_{2N}.
\end{align}
Summing  \eqref{fq2} over such $\alpha$, we then conclude \eqref{Ses}  for $i=d$.

The estimate \eqref{Ses} for $i=1,d-1$ follows similarly as that for $i=d$,
and the estimate \eqref{Se1} follows similarly as \eqref{Ses} by using \eqref{Sto1} and \eqref{ne5}
in place of \eqref{Stop} and \eqref{NW1}, respectively.
\end{proof}

We then record the estimates at the $N+2$ level.
\begin{lem}\label{lem4}
$(1)$ Fix $1\leq j\leq N+1$ and $1\leq k\leq 2(N+2)-2j$. It holds that
\begin{align}\label{Se3}
&\norm{\partial_t^j u}_{k+1,2(N+2)-2j-k}^2
+\norm{\nabla \partial_t^j q}_{k-1,2(N+2)-2j-k}^2
\nonumber\\&\quad
\lesssim \norm{\partial_t^{j+1}u}_{2(N+2)-2j-1}^2+\norm{\partial_t^j \mathbb{Q}}_{k,2(N+2)-2j-k}^2
+\norm{\partial_t^ju}_{1,2(N+2)-2j}^2
+\mathcal{E}_{2N}\mathcal{D}_{N+2,1}.
\end{align}
$(2)$ Fix $1\leq k\leq 2(N+2)-1$. It holds that
\begin{align}\label{Se5}
&\norm{D u}_{k+1,2(N+2)-k-1}^2+\norm{\nabla Dq}_{k-1,2(N+2)-k-1}^2
\nonumber\\&\quad
\lesssim \norm{\partial_tD u}_{2(N+2)-2}^2+\norm{D\mathbb{Q}}_{k,2(N+2)-k-1}^2
+\norm{Du}_{1,2(N+2)-1}^2
+\mathcal{E}_{2N}\mathcal{D}_{N+2,1}.
\end{align}
\end{lem}
\begin{proof}
The proof is similar to that of Lemma \ref{lem3} by using \eqref{Sto1} and \eqref{ne2}
in place of \eqref{Stop} and \eqref{NW1}, respectively.
\end{proof}


\section{Full energy-dissipation estimates}
In this section, we shall combine the energy evolution estimates in Sections 4 and 5 and the elliptic estimates in Section 6
to derive the full energy-dissipation estimates.

We first present the result at the $2N$ level.
\begin{prop}\label{Sy1}
It holds that
\begin{align}\label{Sy}
&\frac{d}{dt}\mathcal{E}_{2N}
+\mathfrak{D}_{2N}
\lesssim
\mathcal{E}_{2N}^{7/8}\mathcal{D}_{N+2,1}^{5/8}
+\mathcal{D}_{N+2,1}\mathcal{F}_{2N}.
\end{align}
\end{prop}

\begin{proof}
For $1\leq j\leq 2N-1$ and $1\leq k \leq 4N-2j-1$, we deduce from \eqref{v1} and \eqref{Se1} that
\begin{align}\label{ssw1}
&\frac{d}{dt}
\norm{\partial_d\partial_t^jq}_{k,4N-2j-k-1}^2
+\norm{\partial_d\partial_t^jq}_{k,4N-2j-k-1}^2
+\norm{ \partial_t^j\mathbb{Q}}_{k+1,4N-2j-k-1}^2
\nonumber\\&\quad
\lesssim \norm{\partial_t^{j+1} u}_{4N-2j-1}^2
+\norm{\partial_t^j \mathbb{Q}}_{k,4N-2j-k}^2
+\norm{\partial_t^ju}_{1,4N-2j}^2
+\sqrt{\mathcal{E}_{2N}}\mathfrak{D}_{2N}.
\end{align}
Then for fixed $1\leq j\leq 2N-1$, a suitable linear combination of \eqref{ssw1} from $k=4N-2j-1$ to $k=1$ leads to
\begin{align}\label{ooo}
&\frac{d}{dt}
\norm{\partial_d\partial_t^jq}_{4N-2j-1}^2
+\norm{\partial_d\partial_t^jq}_{4N-2j-1}^2
+\norm{ \partial_t^j\mathbb{Q}}_{4N-2j }^2
\nonumber\\&\quad
\lesssim \norm{\partial_t^{j+1} u}_{4N-2j-1}^2
+\norm{\partial_t^j\mathbb{Q}}_{1,4N-2j-1}^2
+\norm{\partial_t^ju}_{1,4N-2j}^2
+\sqrt{\mathcal{E}_{2N}}\mathfrak{D}_{2N}.
\end{align}
We recall from \eqref{v1} with $k=0$ that
\begin{align}\label{Sy22}
&\frac{d}{dt}
\norm{\partial_d\partial_t^jq}_{0,4N-2j-1}^2
+\norm{\partial_d\partial_t^jq}_{0,4N-2j-1}^2
+\norm{\partial_d\partial_t^j\mathbb{Q}}_{0,4N-2j-1}^2
\nonumber\\&\quad
\lesssim \norm{\partial_t^{j+1} u}_{4N-2j-1}^2+\norm{\partial_t^j\mathbb{Q}}_{0,4N-2j-1}^2
+\norm{\partial_t^j u}_{1,4N-2j}^2
+\sqrt{\mathcal{E}_{2N}}\mathfrak{D}_{2N}.
\end{align}
We thus conclude from \eqref{ooo} and \eqref{Sy22} that
\begin{align}\label{oo1}
&\frac{d}{dt}
\norm{\partial_d\partial_t^jq}_{4N-2j-1}^2
+\norm{\partial_d\partial_t^jq}_{4N-2j-1}^2
+\norm{ \partial_t^j\mathbb{Q}}_{4N-2j }^2
\nonumber\\&\quad
\lesssim \norm{\partial_t^{j+1} u}_{4N-2j-1}^2
+\norm{\partial_t^j\mathbb{Q}}_{0,4N-2j}^2
+\norm{\partial_t^ju}_{1,4N-2j}^2
+\sqrt{\mathcal{E}_{2N}}\mathfrak{D}_{2N}.
\end{align}
On the other hand, taking $k=4N-2j$ in \eqref{Se1} yields
\begin{align}\label{Sya}
&\norm{\partial_t^ju}_{4N-2j+1}^2
+\norm{\nabla\partial_t^jq}_{4N-2j-1}^2\nonumber\\
&\quad\lesssim
\norm{\partial_t^{j+1} u}_{4N-2j-1}^2
+\norm{\partial_t^j\mathbb{Q}}_{4N-2j}^2
+\norm{\partial_t^ju}_{1,4N-2j}^2
+\mathcal{E}_{2N}\mathfrak{D}_{2N},
\end{align}
which together with \eqref{oo1} implies
\begin{align}\label{Sy4}
&\frac{d}{dt}
\norm{\partial_d\partial_t^jq}_{4N-2j-1}^2
+\norm{\partial_t^ju}_{4N-2j+1}^2
+\norm{\nabla\partial_t^jq}_{4N-2j-1}^2
+\norm{\partial_t^j\mathbb{Q}}_{4N-2j}^2
\nonumber\\&\quad
\lesssim \norm{\partial_t^{j+1} u}_{4N-2j-1}^2+\norm{\partial_t^j\mathbb{Q}}_{0,4N-2j}^2
+\norm{\partial_t^ju}_{1,4N-2j}^2
+\sqrt{\mathcal{E}_{2N}}\mathfrak{D}_{2N}.
\end{align}
Note that by \eqref{dq1} and \eqref{ne5},
\begin{align}\label{11}
\norm{\partial_t^j\mathbb{Q}}_{0,4N-2j}^2
&\lesssim \norm{\partial_t^j\diverge\(\bar{\rho}u\)}_{0,4N-2j}^2+\norm{\partial_t^jG^{1,2}}_{0,4N-2j}^2
\nonumber\\
&\lesssim \norm{\partial_t^ju}_{1,4N-2j}^2+\mathcal{E}_{2N}\mathfrak{D}_{2N}.
\end{align}
Hence, we have
\begin{align}\label{Sy5}
&\frac{d}{dt}
\norm{\partial_d\partial_t^jq}_{4N-2j-1}^2
+\norm{\partial_t^ju}_{4N-2j+1}^2
+\norm{\nabla\partial_t^jq}_{4N-2j-1}^2
+\norm{\partial_t^j\mathbb{Q}}_{4N-2j}^2
\nonumber\\&\quad
\lesssim \norm{\partial_t^{j+1} u}_{4N-2j-1}^2
+\norm{\partial_t^ju}_{1,4N-2j}^2
+\sqrt{\mathcal{E}_{2N}}\mathfrak{D}_{2N}.
\end{align}
A suitable linear combination of \eqref{Sy5} from $j=1$ to $j=2N-1$ then leads to
\begin{align}\label{Syr}
&\frac{d}{dt}
\sum_{j=1}^{2N-1}\norm{\partial_d\partial_t^jq}_{4N-2j-1}^2
+\sum_{j=1}^{2N-1}\left(\norm{\partial_t^ju}_{4N-2j+1}^2
+\norm{\nabla\partial_t^jq}_{4N-2j-1}^2
+\norm{\partial_t^j\mathbb{Q}}_{4N-2j}^2\right)
\nonumber\\&\quad
\lesssim
\sum_{j=1}^{2N}\norm{\partial_t^ju}_{1,4N-2j}^2
+\sqrt{\mathcal{E}_{2N}}\mathfrak{D}_{2N}.
\end{align}

Now, similarly as the derivation of \eqref{Sy4}, basing on \eqref{vs} and \eqref{Ses}, we may deduce that for $i=1,d-1,d,$
\begin{align}\label{bbw0}
&\frac{d}{dt}
\norm{\partial_d\partial_i^\mathcal{A}q}_{4N-2}^2
+\norm{\partial_i^\mathcal{A}u}_{4N}^2
+\norm{\nabla \partial_i^\mathcal{A}q}_{4N-2}^2
+\norm{\mathbb{Q}^{i,\sharp}}_{4N-1}^2
\nonumber\\&\quad
\lesssim\norm{\partial_t\partial_i^\mathcal{A}u}_{4N-2}^2+\norm{\mathbb{Q}^{i,\sharp}}_{0,4N-1}^2
+\norm{\partial_i^\mathcal{A}u}_{1,4N-1}^2
+\delta_{id}\norm{u}_{4N}^2
+\sqrt{\mathcal{E}_{2N}}\mathfrak{D}_{2N}.
\end{align}
Note that by \eqref{xx2} and \eqref{NW1},
\begin{align}\label{ly2}
\norm{\mathbb{Q}^{i,\sharp}}_{0,4N-1}^2
&\lesssim
\norm{\partial_i^\mathcal{A}u}_{1,4N-1}^2
+\delta_{id}\norm{u}_{1,4N-1}^2
+\norm{G^{1,2,i,\sharp}}_{0,4N-1}^2
\nonumber\\&
\lesssim
\norm{\partial_i^\mathcal{A}u}_{1,4N-1}^2
+\delta_{id}\norm{u}_{4N}^2+\mathcal{E}_{2N}\mathfrak{D}_{2N}.
\end{align}
Also,
\begin{align}\label{lw}
\norm{\partial_t\nabla_\mathcal{A}u}_{4N-2}^2
&\lesssim \norm{\partial_t\nabla u}_{4N-2}^2
+\norm{\partial_t\left(\partial_d^\mathcal{A} u \nabla \varphi\right)}_{4N-2}^2
\lesssim\norm{\partial_tu}_{4N-1}^2+\mathcal{E}_{2N}\mathfrak{D}_{2N}
\end{align}
and
\begin{align}\label{bu4}
\norm{\partial_d^\mathcal{A}u}_{1,4N-1}^2
&\lesssim
\norm{\partial_d^\mathcal{A}u}_{0}^2
+\norm{D\partial_d^\mathcal{A}u}_{0,4N-1}^2
+\norm{\partial_d\partial_d^\mathcal{A}u}_{0,4N-1}^2
\nonumber\\&
\lesssim
\norm{\partial_d^\mathcal{A}u}_{1}^2
+\norm{D\partial_d^\mathcal{A}u}_{1,4N-2}^2
\nonumber\\&\lesssim
\norm{\partial_d^\mathcal{A}u}_{1}^2+\norm{\partial_d^\mathcal{A}D_\mathcal{A}u}_{1,4N-2}^2
+\norm{\partial_d^\mathcal{A}\partial_d^\mathcal{A}uD\varphi}_{1,4N-2}^2
\nonumber\\&\lesssim \norm{u}_{4N}^2
+\norm{D_\mathcal{A}u}_{4N}^2
+\mathcal{E}_{2N}\mathfrak{D}_{2N}.
\end{align}
In light of \eqref{ly2}--\eqref{bu4}, we may deduce from \eqref{bbw0} that
\begin{align}\label{face1}
&\frac{d}{dt}
\norm{\partial_d\nabla_\mathcal{A}q}_{4N-2}^2
+\norm{\nabla_\mathcal{A}u}_{4N}^2
+\norm{\nabla \nabla_\mathcal{A}q}_{4N-2}^2
+\sum_{i=1 }^d\norm{\mathbb{Q}^{i,\sharp}}_{4N-1}^2
\nonumber\\&\quad
\lesssim
\norm{\partial_tu}_{4N-1}^2
+\norm{D_\mathcal{A}u}_{1,4N-1}^2
+\norm{u}_{4N}^2
+\sqrt{\mathcal{E}_{2N}}\mathfrak{D}_{2N}.
\end{align}

We now improve the dissipation estimates in the left-hand sides of \eqref{Syr} and \eqref{face1}.
By Poincar\'{e}'s inequality and since $u_d=0$ on $\Sigma_b$, we have
\begin{align}\label{cdqh}
\norm{u_d}_{4N+1}^2
\lesssim
\norm{\nabla u_d}_{4N}^2
&\lesssim
\norm{ \nabla_\mathcal{A}u}_{4N}^2
+\norm{\partial_d^\mathcal{A}u_d\nabla\varphi}_{4N}^2
\nonumber\\&\lesssim
\norm{ \nabla_\mathcal{A}u}_{4N}^2+\mathcal{E}_{2N}\mathfrak{D}_{2N}
+\mathcal{D}_{N+2,1}\mathcal{F}_{2N}.
\end{align}
By the second equation in \eqref{neweq}, we find
\begin{align}
\norm{\nabla_\mathcal{A} q}_{4N-1}^2
&\lesssim
\norm{\partial_t^\mathcal{A} u}_{4N-1}^2
+\norm{u\cdot \nabla_\mathcal{A}  u}_{4N-1}^2
+\norm{\frac{1}{\rho}\diverge_\mathcal{A}\mathbb{S}_\mathcal{A}  u}_{4N-1}^2
\nonumber\\&
\lesssim
\norm{\partial_t u}_{4N-1}^2
+\norm{\nabla_\mathcal{A}u}_{4N}^2
+\mathcal{E}_{2N}\mathfrak{D}_{2N},
\end{align}
and hence,
\begin{align}\label{cdq}
\norm{\nabla q}_{4N-1}^2
&\lesssim
\norm{\nabla_\mathcal{A} q}_{4N-1}^2
+\norm{\partial_d^\mathcal{A}q\nabla\varphi}_{4N-1}^2
\nonumber\\&
\lesssim
\norm{\partial_t u}_{4N-1}^2
+\norm{\nabla_\mathcal{A}u}_{4N}^2
+\mathcal{E}_{2N}\mathfrak{D}_{2N}.
\end{align}
By the vertical component of the second equation in \eqref{Bq1},
using the trace theory and  \eqref{NW2},
we obtain
\begin{align}\label{ddc2}
\abs{D\eta}_{4N-3/2}^2
&\lesssim \abs{D_\mathcal{A}q}_{4N-3/2}^2
+\abs{\nabla D_\mathcal{A}u}_{4N-3/2}^2+\sum_{i=1,d-1}\abs{G_d^{4,i,\sharp}}_{4N-3/2}^2
\nonumber\\&
\lesssim
\norm{Dq}_{4N-1}^2
+\norm{ D_\mathcal{A}u}_{4N}^2+\mathcal{E}_{2N}\mathfrak{D}_{2N}
+\mathcal{D}_{N+2,1}\mathcal{F}_{2N}.
\end{align}
Using the third equation in \eqref{q2}, by \eqref{ne5}, \eqref{cdqh} and the trace theory,
we have
\begin{align}\label{ddc3}
&\abs{\partial_t\eta}_{4N-1}^2
\lesssim
\abs{u_d}_{4N-1}^2+\abs{G^3}_{4N-1}^2
\lesssim
\norm{u_d}_{4N}^2+\mathcal{E}_{2N}\mathfrak{D}_{2N},
\end{align}
and by the definition of $G^3$,
\begin{align}\label{wd1}
\abs{\partial_t^2 \eta}_{4N-3/2}^2
&\lesssim
\abs{\partial_t u_d}_{4N-3/2}^2
+\abs{\partial_t \(u_h D\eta\)}_{4N-3/2}^2
\lesssim
\norm{\partial_t u}_{4N-1}^2
+\abs{u_hD\partial_t  \eta}_{4N-3/2}^2
\nonumber\\&
\lesssim
\norm{\partial_t u}_{4N-1}^2
+\abs{u_hD\(u_d+u_h \cdot D\eta\)}_{4N-3/2}^2
\nonumber\\&
\lesssim
\norm{\partial_t u}_{4N-1}^2+\mathcal{E}_{2N}\mathfrak{D}_{2N}+\abs{u_hu_h D^{4N}\eta}_{1/2}^2
\nonumber\\&
\lesssim
\norm{\partial_t u}_{4N-1}^2+\mathcal{E}_{2N}\mathfrak{D}_{2N}+\mathcal{D}_{N+2,1}\mathcal{F}_{2N},
\end{align}
where we have estimated, by using Lemma \ref{va1} and \eqref{ei},
\begin{align}\label{wd2}
\abs{u_hu_h D^{4N}\eta}_{1/2}^2
\lesssim
\norm{u_h}_{C^{1}(\Sigma)}^{4}\abs{ D^{4N}\eta}_{1/2}^2
\lesssim
\mathcal{D}_{N+2,1}\mathcal{F}_{2N}.
\end{align}
For $j=3,\ldots, 2N+1$, by using again the third equation in \eqref{q2}, \eqref{ne5} and the trace theory,
we obtain
\begin{align}\label{xfq}
\abs{\partial_t^j \eta}_{4N-2j+5/2}^2
&\lesssim
\abs{\partial_t^{j-1} u_d}_{4N-2(j-1)+1/2}^2
+\abs{\partial_t^{j-1} G^3}_{4N-2(j-1)+1/2}^2
\nonumber\\&
\lesssim
\norm{\partial_t^{j-1} u_d}_{4N-2(j-1)+1}^2+\mathcal{E}_{2N}\mathfrak{D}_{2N}.
\end{align}
By the first equation in \eqref{q2} and \eqref{ne5}, we obtain
\begin{align}\label{ddc4}
 \norm{\partial_tq}_{4N-1}^2
\lesssim  \norm{u}_{4N}^2+ \norm{G^1}_{4N-1}^2
\lesssim  \norm{u}_{4N}^2+\mathcal{E}_{2N}\mathfrak{D}_{2N},
\end{align}
and for $j=2,3,\ldots,2N+1$,
\begin{align}\label{xfq2}
 \norm{\partial_t^j q}_{4N-2j+2}^2
&\lesssim  \norm{\partial_t^{j-1} u}_{4N-2(j-1)+1}^2
+ \norm{\partial_t^{j-1} G^1}_{4N-2(j-1)}^2
\nonumber\\&
\lesssim  \norm{\partial_t^{j-1} u}_{4N-2(j-1)+1}^2+\mathcal{E}_{2N}\mathfrak{D}_{2N}.
\end{align}
Therefore,
we can deduce from \eqref{Syr}, \eqref{face1}, \eqref{cdqh}, \eqref{cdq}--\eqref{wd1} and \eqref{xfq}--\eqref{xfq2} that
\begin{align}\label{cmw1}
&\frac{d}{dt}
\left(
 \norm{\partial_d\nabla_\mathcal{A}q}_{4N-2}^2
+\sum_{j=1}^{2N-1} \norm{\partial_d\partial_t^jq}_{4N-2j-1}^2\right)
+\mathfrak{D}_{2N}
\nonumber\\&\quad
\lesssim
\sum_{j=1}^{2N} \norm{\partial_t^ju}_{1,4N-2j}^2
+ \norm{D_\mathcal{A}u}_{1,4N-1}^2
+ \norm{u}_{4N}^2
+\sqrt{\mathcal{E}_{2N}}\mathfrak{D}_{2N}+\mathcal{D}_{N+2,1}\mathcal{F}_{2N}.
\end{align}
By the Sobolev interpolation, we obtain
\begin{align}\label{buh}
 \norm{u}_{4N}^2
\lesssim
 \norm{u}_{0}^2+ \norm{\nabla u}_{4N-1}^2
&\lesssim
 \norm{u}_{0}^2+ \norm{\nabla_\mathcal{A} u}_{4N-1}^2+\mathcal{E}_{2N}\mathfrak{D}_{2N}
\nonumber\\&
\lesssim
 \norm{u}_{0}^2+\varepsilon \norm{\nabla_\mathcal{A} u}_{4N}^2+C(\varepsilon) \norm{\nabla_\mathcal{A} u}_{0}^2+\mathcal{E}_{2N}\mathfrak{D}_{2N}
\nonumber\\&
\lesssim
 \varepsilon \norm{\nabla_\mathcal{A} u}_{4N}^2+\norm{u}_{1}^2+\mathcal{E}_{2N}\mathfrak{D}_{2N}.
\end{align}
Taking $\varepsilon>0$ sufficiently small in \eqref{buh}, we then refine \eqref{cmw1} to be
\begin{align}\label{face3}
&\frac{d}{dt}
\left( \norm{\partial_d\nabla_\mathcal{A}q}_{4N-2}^2
+\sum_{j=1}^{2N-1} \norm{\partial_d\partial_t^jq}_{4N-2j-1}^2
\right)
+\mathfrak{D}_{2N}
\nonumber\\&\quad
\lesssim
\sum_{j=1}^{2N} \norm{\partial_t^ju}_{1,4N-2j}^2
+ \norm{D_\mathcal{A}u}_{1,4N-1}^2
+ \norm{u}_{1}^2
+\sqrt{\mathcal{E}_{2N}}\mathfrak{D}_{2N}+\mathcal{D}_{N+2,1}\mathcal{F}_{2N}.
\end{align}
Combining the estimates \eqref{L2e},
\eqref{tan1},
\eqref{tan3},
 \eqref{tan5}
 and \eqref{face3},  by the smallness of $\mathcal{E}_{2N}$ and Cauchy's inequality,
 we conclude that
\begin{align}\label{cmw5}
&\frac{d}{dt}
\bar{\mathcal{E}}_{2N}
+\mathfrak{D}_{2N}
\lesssim
\mathcal{E}_{2N}^{7/8}\mathcal{D}_{N+2,1}^{5/8}
+
\mathcal{D}_{N+2,1}\mathcal{F}_{2N},
\end{align}
where
\begin{align}
\bar{\mathcal{E}}_{2N}:
=&
 \norm{q}_{0}^{2}
+ \norm{D_\mathcal{A}q}_{4N-1}^{2}
+ \norm{\partial_d\nabla_\mathcal{A}q}_{4N-2}^{2}
+\sum_{j=1}^{2N} \norm{\partial_t^jq}_{4N-2j}^{2}
\nonumber\\&
+ \norm{u}_{0}^{2}
+ \norm{D_\mathcal{A}u}_{0,4N-1}^{2}
+\sum_{j=1}^{2N} \norm{\partial_t^j u}_{0,4N-2j}^{2}
+\sum_{j=0}^{2N}\abs{\partial_t^j\eta}_{4N-2j}^{2}.
\end{align}

Finally, it remains to  improve the energy estimates in \eqref{cmw5}.
It holds that
\begin{align}\label{hao1}
 \norm{q}_{4N}^2
\lesssim
 \norm{q}_{0}^2+ \norm{\nabla ^2q}_{4N-2}^2
&\lesssim
 \norm{q}_{0}^2+ \norm{Dq}_{4N-1}^2+ \norm{\partial_d^2q}_{4N-2}^2
\nonumber\\&
\lesssim
 \norm{q}_{0}^2+ \norm{D_\mathcal{A}q}_{4N-1}^2
+ \norm{\partial_d\partial_d^\mathcal{A}q}_{4N-2}^2
+ \norm{\partial_d^\mathcal{A}q\nabla\varphi}_{4N-1}^2
\nonumber\\&
\lesssim
\bar{\mathcal{E}}_{2N}+(\mathcal{E}_{2N})^2.
\end{align}
A simple induction on \eqref{chf1} from $j=0$ to $j=2N-1$ leads to
\begin{align}\label{ee2}
\sum_{j=0}^{2N-1} \norm{\partial_t^ju}_{4N-2j}^2
&\lesssim
 \norm{\partial_t^{2N}u}_{0}^2
+\sum_{j=0}^{2N-1}\left( \norm{\partial_t^j q}_{4N-2j-1}^2
+\abs{\partial_t^j\eta}_{4N-2j-3/2}^2\right)
+\(\mathcal{E}_{2N}\)^2
\nonumber\\&
\lesssim
\bar{\mathcal{E}}_{2N}+(\mathcal{E}_{2N})^2.
\end{align}
By the smallness of $\mathcal{E}_{2N}$, we can deduce from \eqref{hao1} and \eqref{ee2} that
$
\bar{\mathcal{E}}_{2N}
\sim
\mathcal{E}_{2N},
$
which together with \eqref{cmw5} gives \eqref{Sy}.
\end{proof}

Now we present the result at the $N+2$ level.
\begin{prop}\label{lsy}
It holds that
\begin{align}\label{lsy1}
&\frac{d}{dt}\mathcal{E}_{N+2,1}
+ \mathcal{D}_{N+2,1}
\leq
0.
\end{align}
\end{prop}

\begin{proof}
 Basing on  \eqref{I7} and \eqref{Se3}, similarly as the derivation of \eqref{Sy4}, we may deduce that for $1\leq j\leq N+1$,
\begin{align}
&\frac{d}{dt}
 \norm{\partial_d\partial_t^jq}_{2(N+2)-2j-1}^2
+ \norm{\partial_t^ju}_{2(N+2)-2j+1}^2
+ \norm{\nabla\partial_t^jq}_{2(N+2)-2j-1}^2
+ \norm{\partial_t^j\mathbb{Q}}_{2(N+2)-2j}^2
\nonumber\\&\quad
\lesssim  \norm{\partial_t^{j+1} u}_{2(N+2)-2j-1}^2+ \norm{\partial_t^j\mathbb{Q}}_{0,2(N+2)-2j}^2
+ \norm{\partial_t^ju}_{1,2(N+2)-2j}^2
+\sqrt{\mathcal{E}_{2N}}\mathcal{D}_{N+2,1}.
\end{align}
Note that, by \eqref{dq1} and \eqref{ne2},
\begin{align}
 \norm{\partial_t^j\mathbb{Q}}_{0,2(N+2)-2j}^2
&\lesssim  \norm{\partial_t^j\diverge\(\bar{\rho}u\)}_{0,2(N+2)-2j}^2+ \norm{\partial_t^jG^{1,2}}_{0,2(N+2)-2j}^2
\nonumber\\
&\lesssim  \norm{\partial_t^ju}_{1,2(N+2)-2j}^2+\mathcal{E}_{2N}\mathcal{D}_{N+2,1}.
\end{align}
Hence, we have
\begin{align}
&\frac{d}{dt}
 \norm{\partial_d\partial_t^jq}_{2(N+2)-2j-1}^2
+ \norm{\partial_t^ju}_{2(N+2)-2j+1}^2
+ \norm{\nabla\partial_t^jq}_{2(N+2)-2j-1}^2
+ \norm{\partial_t^j\mathbb{Q}}_{2(N+2)-2j}^2
\nonumber\\&\quad
\lesssim  \norm{\partial_t^{j+1} u}_{2(N+2)-2j-1}^2
+ \norm{\partial_t^ju}_{1,2(N+2)-2j}^2
+\sqrt{\mathcal{E}_{2N}}\mathcal{D}_{N+2,1}.
\end{align}
A suitable linear combination of \eqref{Sy5} from $j=1$ to $j=N+1$ then leads to
\begin{align}\label{lp1}
&\frac{d}{dt}
\sum_{j=1}^{N+1} \norm{\partial_d\partial_t^jq}_{2(N+2)-2j-1}^2
+\sum_{j=1}^{N+1}\left( \norm{\partial_t^ju}_{2(N+2)-2j+1}^2
+ \norm{\nabla\partial_t^jq}_{2(N+2)-2j-1}^2
+ \norm{\partial_t^j\mathbb{Q}}_{2(N+2)-2j}^2\right)
\nonumber\\&\quad
\lesssim
\sum_{j=1}^{N+2} \norm{\partial_t^ju}_{1,2(N+2)-2j}^2
+\sqrt{\mathcal{E}_{2N}}\mathcal{D}_{N+2,1}.
\end{align}

Now, for $1\leq k\leq 2(N+2)-1$, we deduce from \eqref{v2} and \eqref{Se5} that
\begin{align}\label{lp4}
&\frac{d}{dt}
 \norm{\partial_dq}_{k,2(N+2)-k-1}^2
+ \norm{\partial_dq}_{k,2(N+2)-k-1}^2
+ \norm{\mathbb{Q}}_{k+1,2(N+2)-k-1}^2
\nonumber\\&\quad
\lesssim  \norm{\partial_t u}_{2(N+2)-1}^2
+ \norm{\mathbb{Q}}_{k,2(N+2)-k}^2
+ \norm{Du}_{1,2(N+2)-1}^2
+ \norm{ u_d}_{2(N+2)}^2
+\sqrt{\mathcal{E}_{2N}}\mathcal{D}_{N+2,1}.
\end{align}
A suitable linear combination of  \eqref{lp4} from $k= 2(N+2)-1$ to $k=1$ then leads to
\begin{align}\label{coo4}
&\frac{d}{dt}
 \norm{\partial_dq}_{2(N+2)-1}^2
+ \norm{\partial_dq}_{2(N+2)-1}^2
+ \norm{\mathbb{Q}}_{2(N+2)}^2
\nonumber\\&\quad
\lesssim  \norm{\partial_t u}_{2(N+2)-1}^2
+ \norm{\mathbb{Q}}_{1,2(N+2)-1}^2
+ \norm{Du}_{1,2(N+2)-1}^2
+ \norm{ u_d}_{2(N+2)}^2
+\sqrt{\mathcal{E}_{2N}}\mathcal{D}_{N+2,1}.
\end{align}
We recall from \eqref{v2} with $k=0$ that
\begin{align}\label{coo6}
&\frac{d}{dt}
 \norm{\partial_dq}_{0,2(N+2)-1}^2
+ \norm{\partial_dq}_{0,2(N+2)-1}^2
+ \norm{\mathbb{Q}}_{1,2(N+2)-1}^2
\nonumber\\&\quad
\lesssim \norm{\partial_t u}_{2(N+2)-1}^2+ \norm{\mathbb{Q}}_{0,2(N+2)}^2
+ \norm{D u}_{1,2(N+2)-1}^2
+ \norm{u_d}_{2(N+2)}^2
+\sqrt{\mathcal{E}_{2N}}\mathcal{D}_{N+2,1}.
\end{align}
We thus conclude from \eqref{coo4} and \eqref{coo6} that
\begin{align}\label{coo7}
&\frac{d}{dt}
 \norm{\partial_dq}_{2(N+2)-1}^2
+ \norm{\partial_dq}_{2(N+2)-1}^2
+ \norm{\mathbb{Q}}_{2(N+2)}^2
\nonumber\\&\quad
\lesssim  \norm{\partial_t u}_{2(N+2)-1}^2
+ \norm{\mathbb{Q}}_{0,2(N+2)}^2
+ \norm{Du}_{1,2(N+2)-1}^2
+ \norm{ u_d}_{2(N+2)}^2
+\sqrt{\mathcal{E}_{2N}}\mathcal{D}_{N+2,1}.
\end{align}
On the other hand, taking $k=2(N+2)-1$ in \eqref{Se5} yields
\begin{align}\label{coo8}
& \norm{D u}_{2(N+2)}^2+ \norm{\nabla Dq}_{2(N+2)-2}^2
\nonumber\\&\quad
\lesssim
 \norm{\partial_tD u}_{2(N+2)-2}^2+ \norm{D\mathbb{Q}}_{2(N+2)-2}^2
+ \norm{Du}_{1,2(N+2)-1}^2+\mathcal{E}_{2N}\mathcal{D}_{N+2,1}.
\end{align}
Unlike \eqref{11}, using \eqref{ne2} and \eqref{dq1}, we derive
\begin{align}\label{yh3}
 \norm{\mathbb{Q}}_{0,2(N+2)}^2
&\lesssim
 \norm{\mathbb{Q}}_0^2+ \norm{D\mathbb{Q}}_{0,2(N+2)-1}^2
\nonumber\\&
\lesssim
 \norm{\partial_t q}_{0}^2
+ \norm{ G^{1,1}}_{0}^2
+ \norm{ Du}_{1,2(N+2)-1}^2
+ \norm{D G^{1,2}}_{0,2(N+2)-1}^2
\nonumber\\&
\lesssim
 \norm{\partial_t q}_{0}^2
+ \norm{ Du}_{1,2(N+2)-1}^2
+\mathcal{E}_{2N}\mathcal{D}_{N+2,1}.
\end{align}
By \eqref{coo8} and \eqref{yh3}, we deduce from \eqref{coo7} that
\begin{align}\label{yh2}
&\frac{d}{dt}
 \norm{\partial_dq}_{2(N+2)-1}^2
+ \norm{D u}_{2(N+2)}^2+ \norm{ D^2q}_{2(N+2)-2}^2
+ \norm{\partial_dq}_{2(N+2)-1}^2
+ \norm{\mathbb{Q}}_{2(N+2)}^2
\nonumber\\&\quad
\lesssim
 \norm{\partial_t q}_{0}^2
+ \norm{\partial_t u}_{2(N+2)-1}^2
+ \norm{ Du}_{1,2(N+2)-1}^2
+ \norm{u_d}_{2(N+2)}^2
+\sqrt{\mathcal{E}_{2N}}\mathcal{D}_{N+2,1}.
\end{align}

We now improve the dissipation estimates in the left-hand sides of \eqref{lp1} and \eqref{yh2}.
Using the vertical component of the second equation in \eqref{q2}, by \eqref{ne2}, we obatin
\begin{align}\label{ond6}
 \norm{\partial_d^2u_d}_{2(N+2)-1}^2
&\lesssim  \norm{Du}_{2(N+2)}^2
+ \norm{\partial_t u}_{2(N+2)-1}^2+ \norm{\partial_dq}_{2(N+2)-1}^2
+ \norm{G_d^2}_{2(N+2)-1}^2
\nonumber\\&
\lesssim  \norm{Du}_{2(N+2)}^2
+ \norm{\partial_t u}_{2(N+2)-1}^2+ \norm{\partial_dq}_{2(N+2)-1}^2+\mathcal{E}_{2N}\mathcal{D}_{N+2,1}.
\end{align}
By Poincar\'{e}'s inequality, the first equation in \eqref{q2} and \eqref{ne1}, we have
\begin{align}\label{oond2}
\norm{u_d}_0^2+ \norm{\partial_du_d}_{0}^2
\lesssim
 \norm{\partial_d\(\bar{\rho}u_d\)}_{0}^2
&\lesssim
 \norm{\partial_t q}_{0}^2+ \norm{Du}_{0}^2+ \norm{G^1}_0^2
\nonumber\\&
\lesssim
 \norm{\partial_t q}_{0}^2+ \norm{Du}_{0}^2
+\mathcal{E}_{2N}\mathcal{D}_{N+2,1}.
\end{align}
By the Sobolev interpolation,
\begin{align}\label{ond8}
 \norm{u_d}_{2(N+2)}^2
&\lesssim
\varepsilon \norm{u_d}_{2(N+2)+1}^2
+C(\varepsilon) \norm{u_d}_{0}^2.
\end{align}
Hence, in light of \eqref{ond6}--\eqref{ond8} (with taking $\varepsilon>0$ sufficiently small), we can improve \eqref{yh2} to be
\begin{align}\label{tnn1}
&\frac{d}{dt}
 \norm{\partial_dq}_{2(N+2)-1}^2
+ \norm{D u}_{2(N+2)}^2
+ \norm{u_d}_{2(N+2)+1}^2+ \norm{ D^2q}_{2(N+2)-2}^2
+ \norm{\partial_dq}_{2(N+2)-1}^2
\nonumber\\&\quad
\lesssim
 \norm{\partial_t q}_{0}^2
+ \norm{\partial_t u}_{2(N+2)-1}^2
+ \norm{Du}_{1,2(N+2)-1}^2
+\sqrt{\mathcal{E}_{2N}}\mathcal{D}_{N+2,1}.
\end{align}
On the other hand, by the vertical component of the fourth equation in \eqref{q2},  \eqref{ne2} and the trace theory,
we get
\begin{align}\label{nit5}
\abs{D^2 \eta}_{2(N+2)-5/2}^2
&\lesssim\abs{D^2 q}_{2(N+2)-5/2}^2
 +\abs{D^2 u}_{2(N+2)-3/2}^2+\abs{D^2G_d^4}_{2(N+2)-5/2}^2
\nonumber\\&
\lesssim
 \norm{D^2 q}_{2(N+2)-2}^2
 +\norm{D^2 u}_{2(N+2)-1}^2
 +\mathcal{E}_{2N}\mathcal{D}_{N+2,1}.
\end{align}
 For $j=1,2,\ldots,N+3$, by the  third equation in \eqref{q2},
 we  have
\begin{align}\label{nit3}
\abs{\partial_t^j \eta}_{2(N+2)-2j+5/2}^2
&\lesssim
\abs{\partial_t^{j-1} u_d}_{2(N+2)-2(j-1)+1/2}^2
+\abs{\partial_t^{j-1} G^3}_{2(N+2)-2(j-1)+1/2}^2
\nonumber\\&
\lesssim \norm{\partial_t^{j-1} u_d}_{2(N+2)-2(j-1)+1}^2+\mathcal{E}_{2N}\mathcal{D}_{N+2,1},
\end{align}
and by the first equation in \eqref{q2} and \eqref{ne2},
 we obtain
\begin{align}\label{nnn}
\norm{\partial_t^j q}_{2(N+2)-2j+2}^2
&\lesssim
\norm{\partial_t^{j-1}D u_h}_{2(N+2)-2j+2}^2
+\norm{\partial_t^{j-1} u_d}_{2(N+2)-2j+3}^2
+\norm{\partial_t^{j-1} G^1}_{2(N+2)-2j+2}^2
\nonumber\\&
\lesssim
\norm{\partial_t^{j-1}D u_h}_{2(N+2)-2(j-1)}^2
+\norm{\partial_t^{j-1} u_d}_{2(N+2)-2(j-1)+1}^2+\mathcal{E}_{2N}\mathcal{D}_{N+2,1}.
\end{align}
By the horizontal component of the fourth equation in \eqref{q2}, \eqref{ne2} and the trace theory, we derive
\begin{align}\label{yus}
\abs{\partial_du_h}_{2(N+2)-1/2}^2
&\lesssim
\abs{D u_d}_{2(N+2)-1/2}^2
+\abs{G_h^4}_{2(N+2)-1/2}^2
\lesssim
\norm{D u_d}_{2(N+2)}^2+\mathcal{E}_{2N}\mathcal{D}_{N+2,1}.
\end{align}
Therefore, we may deduce from \eqref{lp1}, \eqref{tnn1}--\eqref{yus} that
\begin{align}\label{cmw2}
&\frac{d}{dt}
\sum_{j=0}^{N+1}\norm{\partial_d\partial_t^jq}_{2(N+2)-2j-1}^2
+\mathcal{D}_{N+2,1}
\nonumber\\&\quad
\lesssim
\norm{\partial_t q}_{0}^2
+\sum_{j=1}^{N+2}\norm{\partial_t^ju}_{1,2(N+2)-2j}^2
+\norm{Du}_{1,2(N+2)-1}^2
+\sqrt{\mathcal{E}_{2N}}\mathcal{D}_{N+2,1}.
\end{align}

Differently from \eqref{face3}, here we need to control additionally  the dissipation estimate of $\left\|\partial_t q\right\|_{0}^2$.
Note that it does not follow from the first equation in \eqref{q2}, since we do not have the control of $\|u_d\|_0^2$ due to
the tangential energy evolution estimates with minimal count 1 of derivatives.
To overcome this,
we apply $\partial_t$ to the vertical component of the second equation in \eqref{q2} and then take the $L^2(\Omega)$ inner product
of the resulting with $u_d$ to obtain
\begin{align}\label{zhen}
&\int_\Omega\bar{\rho}\partial_t\partial_d q u_d
-\int_\Omega \partial_t\(\diverge\mathbb{S}u\)_d u_d=-\int_\Omega \bar{\rho}\partial_t^2u_d u_d+\int_\Omega \partial_tG_d^2 u_d.
\end{align}
By integrating by parts  and using the fifth equation of \eqref{q2},
 we deduce that, using the first equation of \eqref{q2},
\begin{align}
\int_\Omega\bar{\rho}\partial_t\partial_d q u_d
&=\int_\Sigma \bar{\rho}\partial_t q u_d-\int_\Omega\partial_t q\partial_d\(\bar{\rho} u_d\)
\nonumber\\&
=\int_\Sigma \rho^\star\partial_t q u_d+\int_\Omega\frac{1}{h'(\bar{\rho})}|\partial_t q|^2+\int_\Omega\partial_t qD\cdot\(\bar{\rho} u_h\)
-\int_\Omega\frac{1}{h'(\bar{\rho})}\partial_t qG^1
\end{align}
and that
\begin{align}
-\int_\Omega \partial_t\(\diverge\mathbb{S}u\)_d u_d
&=
-\int_\Sigma \partial_t\(\mathbb{S}u\)_{dd}u_d+\int_\Omega \partial_t\(\mathbb{S}u\)_{dj} \partial_ju_d.
\end{align}
On the other hand, by the vertical component of the fourth equation and the third equation in \eqref{q2},
\begin{align}\label{zhen1}
\int_\Sigma \rho^\star\partial_t q u_d-\int_\Sigma \partial_t\(\mathbb{S}ue_d\)_du_d
&=\int_\Sigma \(\rho^\star \partial_t \eta +\partial_tG_d^4\)u_d
\nonumber\\&
=\int_\Sigma \rho^\star |\partial_t \eta|^2-\int_\Sigma \rho^\star \partial_t \eta G^3+\int_\Sigma \partial_tG_d^4u_d.
\end{align}
Hence, by \eqref{zhen}--\eqref{zhen1}, we have
\begin{align}\label{ok2}
\int_\Omega \frac{1}{h'(\bar{\rho})}\abs{\partial_tq}^2+\int_\Sigma \rho^\star g\abs{\partial_t\eta}^2
=&
-\int_\Omega \bar{\rho}\partial_t^2u_d u_d+\int_\Omega \partial_tG_d^2 u_d
-\int_\Omega\partial_t qD\cdot\(\bar{\rho} u_h\)
+\int_\Omega\frac{1}{h'(\bar{\rho})}\partial_t qG^1
\nonumber\\&
-\int_\Omega \partial_t\(\mathbb{S}u\)_{dj} \partial_ju_d
+\int_\Sigma \rho^\star \partial_t \eta G^3-\int_\Sigma \partial_tG_d^4u_d.
\end{align}
By \eqref{ne2} and the trace theory,  we  derive
\begin{align}\label{zhen3}
\norm{\partial_tq}_0^2+\abs{\partial_t\eta}_0^2
&\lesssim
\norm{\partial_t^2 u_d}_0\norm{u_d}_0
+\norm{\partial_tG_d^2}_0\norm{u_d}_0
+\norm{\partial_tq}_0\norm{D u_h}_0
+\norm{\partial_tq}_0\norm{G^1}_0
\nonumber\\&\quad
+\norm{\partial_t\nabla u}_0\norm{\nabla u_d}_0
+\abs{\partial_t\eta}_0\abs{G^3}_0^2
+\abs{\partial_tG_d^4}_0\abs{u_d}_0
\nonumber\\&
\lesssim
\norm{\partial_t^2 u_d}_0\norm{u_d}_0+\norm{\partial_tq}_0\norm{D u_h}_0+\norm{\partial_t\nabla u}_0\norm{\nabla u_d}_0
+\sqrt{\mathcal{E}_{2N}}\mathcal{D}_{N+2,1}.
\end{align}
Therefore, by \eqref{zhen3}, \eqref{oond2} and Cauchy's inequality, we obtain
\begin{align}\label{ok3}
\norm{\partial_tq}_0^2+\abs{\partial_t\eta}_0^2
\lesssim
\norm{\partial_t^2 u_d}_0^2+\norm{\partial_t u}_1^2+\norm{D u}_0^2
+\sqrt{\mathcal{E}_{2N}}\mathcal{D}_{N+2,1}.
\end{align}
By \eqref{cmw2} and \eqref{ok3}, we have
\begin{align}\label{wwt}
&\frac{d}{dt}
\sum_{j=0}^{N+1}\norm{\partial_d\partial_t^jq}_{2(N+2)-2j-1}^2
+\mathcal{D}_{N+2,1}
\nonumber\\&\quad
\lesssim
\sum_{j=1}^{N+2}\norm{\partial_t^ju}_{1,2(N+2)-2j}^2
+\norm{ Du}_{1,2(N+2)-1}^2
+\sqrt{\mathcal{E}_{2N}}\mathcal{D}_{N+2,1}.
\end{align}
Combining the estimates \eqref{tan2},
\eqref{tan4}, \eqref{wwt},
by the smallness of $\mathcal{E}_{2N}$, we conclude that
\begin{align}\label{okk}
&\frac{d}{dt}
\bar{\mathcal{E}}_{N+2,1}
+\mathcal{D}_{N+2,1}
\leq 0,
\end{align}
where
\begin{align}\label{bare}
\bar{\mathcal{E}}_{N+2,1}:
=&\norm{\nabla q}_{2(N+2)-1}^{2}
+\sum_{j=1}^{N+2}\norm{\partial_t^jq}_{2(N+2)-2j}^{2}
+\norm{Du}_{0,2(N+2)-1}^{2}
+\sum_{j=1}^{N+2}\norm{\partial_t^j u}_{0,2(N+2)-2j}^{2}
\nonumber\\&
+\abs{D\eta}_{2(N+2)-1}^{2}
+\sum_{j=1}^{N+2}\abs{\partial_t^j\eta}_{2(N+2)-2j}^{2}.
\end{align}

Finally, it remains to improve the energy estimates in \eqref{okk}.
A simple induction on \eqref{chf2} from $j=1$ to $j=N+1$ leads to
\begin{align}\label{chff}
\sum_{j=1}^{N+1}\norm{\partial_t^ju}_{2(N+2)-2j}^2
&\lesssim
\norm{\partial_t^{N+2}u}_{0}^2
+\sum_{j=1}^{N+1}\(\norm{\partial_t^j q}_{2(N+2)-2j-1}^2
+\abs{\partial_t^j\eta}_{2(N+2)-2j-3/2}^2\)
+\mathcal{E}_{2N}\mathcal{E}_{N+2,1}
\nonumber\\&
\lesssim
\bar{\mathcal{E}}_{N+2,1}+\mathcal{E}_{2N}\mathcal{E}_{N+2,1}.
\end{align}
Similarly as \eqref{ond6}, by the second equation in \eqref{q2} and \eqref{ne1}, we derive
\begin{align}\label{ond1}
\norm{\partial_d^2u}_{2(N+2)-2}^2
&\lesssim \norm{Du}_{2(N+2)-1}^2
+\norm{\partial_t u}_{2(N+2)-2}^2+\norm{\nabla q}_{2(N+2)-2}^2
+\mathcal{E}_{2N}\mathcal{E}_{N+2,1}.
\end{align}
Similarly as \eqref{oond2}, by \eqref{ne1}, we get
\begin{align}\label{ond2}
\norm{u_d}_0^2+\norm{\partial_du_d}_{0}^2
\lesssim
\norm{\partial_d(\bar{\rho}u_d)}_{0}^2
&\lesssim
\norm{\partial_t q}_{0}^2+\norm{Du}_{0}^2
+\mathcal{E}_{2N}\mathcal{E}_{N+2,1}
.
\end{align}
Utilizing the horizontal component of the fourth equation in \eqref{q2},
by \eqref{ne1} and the trace theory, we find
\begin{align}\label{rx8}
\abs{\partial_du_h}_{0}^2
\lesssim \abs{Du_d}_{0}^2
+\abs{G_h^4}_{0}^2
\lesssim \norm{Du_d}_{1}^2+\mathcal{E}_{2N}\mathcal{E}_{N+2,1}.
\end{align}
Note that, by Poincar\'{e}'s inequality,
\begin{align}
\norm{\partial_du_h}_{0}^2
\lesssim
\abs{\partial_du_h}_{0}^2+\norm{\partial_d^2u_h}_{0}^2.
\end{align}
Hence, by \eqref{chf3}, \eqref{chff} and Poincar\'{e}'s inequality, we have
\begin{align}\label{ond3}
\norm{u}_{2(N+2)}^2
&\lesssim
\norm{Du}_{2(N+2)-1}^2
+\norm{\partial_du}_{0}^2
+\norm{\partial_d^2u}_{2(N+2)-2}^2
\nonumber\\&
\lesssim
 \norm{Du}_{2(N+2)-1}^2
+\norm{\partial_t u}_{2(N+2)-2}^2+\norm{\partial_t q}_{0}^2+\norm{\nabla q}_{2(N+2)-2}^2+\mathcal{E}_{2N}\mathcal{E}_{N+2,1}
\nonumber\\&
\lesssim
\norm{\partial_tu}_{2(N+2)-2}^2
+\norm{\partial_t q}_{0}^2
+\norm{\nabla q}_{2(N+2)-2}^2
+\abs{D\eta}_{2(N+2)-5/2}^2
+\mathcal{E}_{2N}\mathcal{E}_{N+2,1}
\nonumber\\&
\lesssim
\bar{\mathcal{E}}_{N+2,1}+\mathcal{E}_{2N}\mathcal{E}_{N+2,1}.
\end{align}
Consequently, by the smallness of $\mathcal{E}_{2N}$, we can deduce from \eqref{bare}--\eqref{ond3} that
$
\bar{\mathcal{E}}_{N+2,1}
\sim\mathcal{E}_{N+2,1},
$
which together with \eqref{okk} gives \eqref{lsy1}.
\end{proof}





\section{A priori estimates}
To conclude the a priori estimates, it still needs to estimate $\mathcal{F}_{2N}$.
\begin{prop}\label{pf}
It holds that
\begin{align}\label{L4}
\frac{d}{dt}\mathcal{F}_{2N}
\lesssim \sqrt{\mathfrak{D}_{2N}}\sqrt{\mathcal{F}_{2N}}
+\sqrt{\mathcal{D}_{N+2,1}}\mathcal{F}_{2N}.
\end{align}
\end{prop}
\begin{proof}
We refer to Proposition 6.4 of \cite{WYJ1}, by using the third equation in \eqref{neweq}, to derive
\begin{align}\label{f2n}
\frac{d}{dt}\mathcal{F}_{2N}
&\lesssim
\abs{u}_{4N+1/2}\sqrt{\mathcal{F}_{2N}}+\abs{Du}_{L^\infty(\Sigma)}\mathcal{F}_{2N}
\nonumber\\&
\lesssim
\abs{u}_{4N+1/2}\sqrt{\mathcal{F}_{2N}}
+\sqrt{\mathcal{D}_{N+2,1}}\mathcal{F}_{2N}.
\end{align}
Note that, by Sobolev's embedding and the trace theory, we find
\begin{align}\label{f2nn}
\abs{u}_{4N+1/2}
\lesssim
\abs{u}_{0}+\abs{Du}_{4N-1/2}
&\lesssim
\abs{u}_{0}+\abs{D_\mathcal{A}u}_{4N-1/2}+\abs{\partial_d^\mathcal{A}uD\varphi}_{4N-1/2}
\nonumber\\&
\lesssim
\norm{u}_{1}+\norm{\nabla_\mathcal{A}u}_{4N}
+\norm{\nabla u}_{C^1(\Sigma)}\abs{\eta}_{4N+1/2}
\nonumber\\&
\lesssim \sqrt{\mathfrak{D}_{2N}}
+\sqrt{\mathcal{D}_{N+2,1}\mathcal{F}_{2N}}.
\end{align}
Therefore, combining \eqref{f2n} and \eqref{f2nn} implies \eqref{L4}.
\end{proof}
Now we shall present the proof of the a priori estimates.
\begin{thm}\label{main}
There exists a universal $0<\delta<1$ so that if $\mathcal{G}_{2N}(T)\leq \delta$, then
\begin{align}\label{main1}
\mathcal{G}_{2N}(T)\lesssim \mathcal{G}_{2N}(0).
\end{align}
\end{thm}

\begin{proof}
We divide the proof into three steps.

\textbf{Step 1: Bounded estimates of $\mathcal{E}_{2N}+\int_{0}^{t}\mathfrak{D}_{2N}$.}
We integrate  \eqref{Sy1} in Proposition \ref{Sy} directly
 in time to find, by the definition of $\mathcal{G}_{2N}(t)$ and H\"{o}lder's inequality,
\begin{align}\label{pr6}
&\mathcal{E}_{2N}(t)+\int_0^t \mathfrak{D}_{2N}(r)dr
\nonumber\\&\quad
\lesssim \mathcal{E}_{2N}(0)
+\int_0^t\mathcal{E}_{2N}^{7/8}(r)\left((1+r)\mathcal{D}_{N+2,1}(r)\right)^{5/8}(1+r)^{-5/8}dr
+\int_0^t(1+r)\mathcal{D}_{N+2,1}(r)\frac{\mathcal{F}_{2N}(r)}{(1+r)}dr
\nonumber\\&\quad
\lesssim
\mathcal{E}_{2N}(0)+(\mathcal{G}_{2N}(t))^{7/8}\left(\int_0^t(1+r)\mathcal{D}_{N+2,1}(r)dr\right)^{5/8}
\left(\int_0^t(1+r)^{-5/3}dr\right)^{3/8}
+(\mathcal{G}_{2N}(t))^{2}
\nonumber\\&\quad
\lesssim \mathcal{G}_{2N}(0)+(\mathcal{G}_{2N}(t))^{3/2}.
\end{align}

\textbf{Step 2: Growth estimate of $\mathcal{F}_{2N}$.}
Multiplying \eqref{L4} in Proposition \ref{pf} by $(1+t)^{-1}$ and then integrating in time, we derive
\begin{align}\label{L5}
&\frac{\mathcal{F}_{2N}(t)}{(1+t)}
+\int_0^t \frac{\mathcal{F}_{2N}(r)}{(1+r)^{2}}dr
\nonumber\\&\quad
\lesssim \int_0^t\sqrt{\mathfrak{D}_{2N}(r)}
\sqrt{\frac{\mathcal{F}_{2N}(r)}{(1+r)^{2}}}dr
+\int_0^t\sqrt{(1+r)\mathcal{D}_{N+2,1}(r)}\sqrt{\frac{\mathcal{F}_{2N}(r)}{(1+r)^{2}}}
\sqrt{\frac{\mathcal{F}_{2N}(r)}{(1+r)}}dr.
\end{align}
Since $\delta>0$ is small, by Cauchy's inequality and \eqref{pr6}, we find that
\begin{align}\label{L55}
&\frac{\mathcal{F}_{2N}(t)}{(1+t)}
+\int_0^t \frac{\mathcal{F}_{2N}(r)}{(1+r)^{2}}dr
\nonumber\\&\quad
\lesssim \mathcal{G}_{2N}(0)+\int_0^t \mathfrak{D}_{2N}(r)dr
+\int_0^t(1+r)\mathcal{D}_{N+2,1}(r)\frac{\mathcal{F}_{2N}(r)}{(1+r)}dr
\nonumber\\&\quad
\lesssim\mathcal{G}_{2N}(0)+(\mathcal{G}_{2N}(t))^{3/2}.
\end{align}

\textbf{Step 3: Decay estimates of $\mathcal{E}_{N+2,1}$ and $\mathcal{D}_{N+2,1}$.}
Multiplying \eqref{lsy1} in Proposition \ref{lsy} by $(1+t)$ and then integrating the resulting, by noting
$\mathcal{E}_{N+2,1}(t)\lesssim\mathfrak{D}_{2N}(t)$ and using \eqref{pr6}, we obtain
\begin{align}\label{dee3}
&(1+t)\mathcal{E}_{N+2,1}(t)+\int_0^t(1+r)\mathcal{D}_{N+2,1}(r)dr
\nonumber\\&\quad
\lesssim\mathcal{E}_{N+2,1}(0)+\int_0^t\mathcal{E}_{N+2,1}(r)dr
\lesssim \mathcal{E}_{2N}(0)+\int_0^t\mathfrak{D}_{2N}(r)dr
\lesssim \mathcal{G}_{2N}(0)+(\mathcal{G}_{2N}(t))^{3/2}.
\end{align}

Consequently, summing \eqref{pr6}, \eqref{L55} and \eqref{dee3} gives \eqref{main1}, by the smallness of $\delta$.
\end{proof}

\appendix
\section{Analytic tools}
\subsection{Poisson integral}
For a function $f$  defined on $\Sigma=\mathbb{R}^{d-1}$, the Poisson integral in $\mathbb{R}^{d-1}\times (-\infty,0)$
is defined by
\begin{align}\label{pp1}
\mathcal{P}f(x',x_{d})=\int_{\mathbb{R}^{2}}\hat{f}(\xi)e^{2\pi|\xi|x_{d}}e^{2\pi i x'\cdot \xi}d\xi,
\end{align}
where $\hat{f}$ is the horizontal Fourier transform of $f$.
\begin{lem}\label{pp2}
Let $\mathcal{P}f$ be the Poisson integral of a function $f$  that is either in $\dot{H}^{q}(\Sigma)$ or $\dot{H}^{q-1/2}(\Sigma)$
for $q\in \mathbb{N}$, where $\dot{H}^{s}$ is the usual homogeneous Sobolev  space of order $s$. Then
\begin{align}\label{pp3}
\norm{\nabla^{q}\mathcal{P}f}_{0}^{2}\lesssim\norm{f}_{\dot{H}^{q-1/2}(\Sigma)}^{2}
~and~ \norm{\nabla^{q}\mathcal{P}f}_{0}^{2}\lesssim\norm{f}_{\dot{H}^{q}(\Sigma)}^{2}.
\end{align}

\end{lem}
\begin{proof}
We refer to Lemma A.5 in \cite{GT_inf} for the proof.
\end{proof}
\subsection{Interpolation estimates}

Assume that $\Sigma=\mathbb{R}^{d-1}$ and $\Omega=\mathbb{R}^{d-1}\times (-b,0)$.
 We then give some interpolation results for Poisson integrals defined by \eqref{pp1}.
 \begin{lem}\label{pp4}
Let $\mathcal{P}f$ be the Poisson integral of $f$, defined on $\Sigma$. Then the following estimates hold.\\
 $(1)$ Let $q,s \in \mathbb{N}$ and
 \begin{align}\label{pp5}
 \theta=\frac{s}{q+s}~ and ~1-\theta=\frac{q}{q+s}.
 \end{align}
  Then
  \begin{align}\label{pp6}
  \norm{\nabla^{q}\mathcal{P}f}_{0}^{2}\lesssim\(\abs{f}_{0}^{2}\)^{\theta}\(\abs{D^{q+s}f}_{0}^{2}\)^{1-\theta}.
 \end{align}
 $(2)$Let  $q,s \in \mathbb{N}$,  $r\geq 0$, $r+s>1$ and
\begin{align}\label{pp7}
 \theta=\frac{r+s-1}{q+r+s}~ and ~1-\theta=\frac{q+1}{q+s+r}.
 \end{align}
 Then
 \begin{align}\label{pp8}
  \norm{\nabla^{q}\mathcal{P}f}_{L^{\infty}(\Omega)}^{2}\lesssim\(\abs{f}_{0}^{2}\)^{\theta}\(\abs{D^{q+s}f}_{r}^{2}\)^{1-\theta}.
 \end{align}
 $(3)$ Let $s>1$. Then
  \begin{align}\label{pp9}
  \norm{\nabla^{q}\mathcal{P}f}_{L^{\infty}(\Omega)}^{2}\lesssim \abs{D^{q}f}_{s}^{2}.
 \end{align}
 \end{lem}

 \begin{proof}
We refer to Lemma A.6 in \cite{GT_inf} for the proof.
\end{proof}
 The next result is a similar interpolation result for functions defined only on $\Sigma$.
 \begin{lem}\label{pp10}
 Let $f$ be defined on $\Sigma$. Then the following estimates hold.\\
 $(1)$ Let $q,s \in \mathbb{N}$ and $\theta$ be as in \eqref{pp5}. Then
 \begin{align}\label{pp11}
\abs{D^{q}f}_{0}^{2}\lesssim\(\abs{f}_{0}^{2}\)^{\theta}\(\abs{D^{q+s}f}_{0}^{2}\)^{1-\theta}.
 \end{align}
 $(2)$ Let $q,s\in \mathbb{N}$, $r\geq 0$, $r+s>1$ and $\theta$ be as in \eqref{pp7}. Then
 \begin{align}\label{pp12}
 \norm{D^{q}f}_{L^{\infty}(\Sigma)}^{2}\lesssim\(\abs{f}_{0}^{2}\)^{\theta}\(\abs{D^{q+s}f}_{r}^{2}\)^{1-\theta}.
 \end{align}
 \end{lem}

 \begin{proof}
We refer to Lemma A.7 in \cite{GT_inf} for the proof.
\end{proof}
The next result is a similar result for functions defined on $\Omega$ that are not on Poisson integrals.
 \begin{lem}\label{pp13}
 Let $f$ be defined on $\Omega$. Then the following estimates hold.\\
 $(1)$ Let $q,s \in \mathbb{N}$ and $\theta$ be as in \eqref{pp5}. Then
  \begin{align}\label{pp14}
  \norm{D^{q}f}_{0}^{2}\lesssim\(\norm{f}_{0}^{2}\)^{\theta}\(\norm{D^{q+s}f}_{0}^{2}\)^{1-\theta}.
 \end{align}
  $(2)$ Let $q,s\in \mathbb{N}$, $r\geq 0$, $r+s>1$ and $\theta$ be as in \eqref{pp7}. Then
 \begin{align}\label{pp15}
  \norm{D^{q}f}_{L^{\infty}(\Omega)}^{2}\lesssim\(\norm{f}_{1}^{2}\)^{\theta}\(\norm{D^{q+s}f}_{r+1}^{2}\)^{1-\theta}.
 \end{align}
 \end{lem}
\begin{proof}
We refer to Lemma A.8 in \cite{GT_inf} for the proof.
\end{proof}

\subsection{Poincar\'{e}-type  inequalities}
We will need the following types of Korn's inequality.
\begin{lem}\label{xm5}
If $u=0$ on $\Sigma_{b}$, then
it holds that
\begin{itemize}
  \item[(1)]  $\norm{u}_{1}\lesssim\norm{\mathbb{D}u}_{0}$;
  \item[ (2)] $\norm{u}_{1}\lesssim \norm{\mathbb{D}^0u}_{0}$ when $d=3$.
\end{itemize}

\end{lem}
\begin{proof}
We refer to Lemma 2.7 in \cite{beale_1} and Theorem 1.1 of \cite{SD}.
\end{proof}

Since $\Omega$ is of finite depth in the vertical direction, we have the following refined  Poincar\'{e}-type  inequalities.
\begin{lem}\label{poincare_b}
 It holds that
\beq\label{good2}
 \norm{f}_{0}  \ls  \abs{f}_{0}+ \norm{\p_d f}_{0}
\text{
and
}
 \norm{f}_{L^\infty(\Omega)}  \ls  \norm{f}_{L^\infty(\Sigma)} + \norm{\p_d f}_{L^\infty(\Omega)} ,
\eeq
and the inequalities hold also for $\Sigma$ replaced by $\Sigma_b.$
\end{lem}
\begin{proof}
It follows from the Newton--Leibniz formula in the vertical variable $x_d\in (-b,0)$ that
\beq\label{good2bb}
f(x_h, x_d)=f(x_h, 0)-\int_{x_d}^0 \p_d f(x_h,z)dz.
\eeq
This yields directly the second inequality in \eqref{good2} by taking the supremum of \eqref{good2bb} over $x_h\in \mathbb{R}^{d-1}$ and $x_d\in (-b,0)$ and using the fact that $b>0$ is finite. The first inequality follows similarly by taking the $L^2$ norm of \eqref{good2bb} and using further the Cauchy-Schwarz inequality. We may also refer to Lemma A.10 in \cite{GT_inf} for the proof.
\end{proof}

\subsection{Commutator and product estimates}
We will need some estimates of the product of functions defined on $\Sigma$ in Sobolev spaces.
\begin{lem}\label{va1}
It holds that
\begin{align}\label{va2}
\abs{fg}_{1/2}\lesssim\abs{f}_{1/2}\norm{g}_{C^{1}(\Sigma)}.
\end{align}
\end{lem}
\begin{proof}
We refer to Lemma A.2 in \cite{GT_inf} for the proof.
\end{proof}

We also need the following commutator and product estimates.
\begin{lem}\label{Pe1}
Let $q \in \mathbb{N}$. It holds that for $s=0,1/2$,
\begin{align}\label{Pe2}
\abs{D^q(fg)}_{s}\lesssim\abs{D^qf}_{s}\norm{g}_{C^{\kappa_s}(\Sigma)}+\abs{D^qg}_{s}\norm{f}_{C^{\kappa_s}(\Sigma)}
\end{align}
and
\begin{align}\label{Pe3}
\abs{\[D^q,f\]g}_{s}\lesssim\abs{D^qf}_{s}\norm{g}_{C^{\kappa_s}(\Sigma)}+\abs{D^{q-1}g}_{s}\norm{Df}_{C^{\kappa_s}(\Sigma)},
\end{align}
where $\kappa_0=0$ when $s=0$ and $\kappa_{1/2}=1$ when $s=1/2$.
\end{lem}
\begin{proof}
For $s=0$, we refer to Lemma 3.4 in \cite{Majda_B}.
For $s=1/2$, the estimates \eqref{Pe2} and \eqref{Pe3} follow  by combining Lemma \ref{va1} and the way of proving the case $s=0$.
\end{proof}

\subsection{Elliptic estimates}
We will need the elliptic estimates for the following Lam\'{e} problem.
\begin{lem}\label{St1}
Let $r\geq2$. Suppose that $f\in H^{r-2}(\Omega)$, $\psi\in H^{r-3/2}(\Sigma)$.
Then there exists a unique $u\in H^{r}(\Omega)$ solving the problem
\begin{equation} \label{St2}
\begin{cases}
- \mu\Delta u-\left(\displaystyle\frac{d-2}{d}\mu+\mu'\right)\nabla \diverge u=f &\text{in}~ \Omega
\\ -\mathbb{S}ue_{d}=\psi& \text{on }\Sigma
\\u=0 & \text{on }\Sigma_{b}.
\end{cases}
\end{equation}
Moreover,
\begin{align}\label{St3}
\norm{u}_{r}^{2}\lesssim\norm{f}_{r-2}^{2}+\abs{\psi}_{r-3/2}^{2}.
\end{align}
\end{lem}

\begin{proof}
We refer to Lemma A.10 in \cite{WYJ2} for the proof.
\end{proof}
We also need the elliptic estimates for the following Stokes problem.
\begin{lem}\label{St4}
Let $r\geq2$. Suppose that $f\in H^{r-2}(\Omega)$, $h\in H^{r-1}(\Omega)$, $\psi\in H^{r-1/2}(\Sigma)$ and that
$u\in H^{r}(\Omega)$, $\nabla p\in H^{r-2}(\Omega)$ solving the problem
\begin{equation} \label{St5}
\begin{cases}
- \mu\Delta u+\nabla p=f &\text{in}~ \Omega
\\  \diverge u=h& \text{in }\Omega
\\ u=\psi& \text{on }\Sigma
\\u=\phi & \text{on }\Sigma_{b}.
\end{cases}
\end{equation}
Then
\begin{align}\label{St6}
\norm{u}_{r}^{2}+\norm{\nabla p}_{r-2}^{2}\lesssim\norm{u}_{0}^{2}+\norm{f}_{r-2}^{2}+\norm{h}_{r-1}^{2}+\abs{\psi}_{r-1/2}^{2}
+\norm{\phi}_{H^{r-1/2}(\Sigma_b)}^{2}.
\end{align}
\end{lem}
\begin{proof}
We refer to (3.7) in \cite{KY} for the proof.
\end{proof}

\section*{Acknowledgements}
The authors are deeply grateful to the referees for the insightful comments and suggestions.\bigskip

\noindent{\bf Declarations}

\section*{Data Availability}
No datasets are analysed or generated in this manuscript.

\section*{Conflict of Interest}
The authors declare that there is no conflict of interest.

\end{document}